\documentclass[12pt]{amsart}

\usepackage[margin=2.5cm]{geometry}
\usepackage{graphicx}
\usepackage{float}
\usepackage{amsmath,amssymb}
\usepackage{amscd}
\usepackage{verbatim}
\usepackage{enumerate}
\usepackage{subcaption}
\usepackage{bbold}
\usepackage{tikz-cd}
\usepackage{multido}

\newtheorem{theorem}{Theorem}[section]
\newtheorem{definition}[theorem]{Definition}
\newtheorem{proposition}[theorem]{Proposition}
\newtheorem{lemma}[theorem]{Lemma}
\newtheorem{remark}[theorem]{Remark}

\newtheorem{corollary}[theorem]{Corollary}

\newcommand{\g}{\mathfrak{g}}
\newcommand{\hh}{\mathfrak{h}}
\newcommand{\nf}{\mathfrak{n}}
\newcommand{\bfr}{\mathfrak{b}}

\newcommand{\CC}{\mathbb{C}}

\newcommand{\ZZ}{\mathbb{Z}}

\newcommand{\Rep}{\mathcal{M}_{\textup{fd}}}

\newcommand{\qTr}{\operatorname{qTr}}

\newcommand{\id}{\mathrm{id}}
\newcommand{\mr}{\mathrm}
\newcommand{\ol}{\overline}

\newcommand{\wts}{\mathrm{wts}}

\newcommand{\str}{\mathrm{str}}
\newcommand{\cBrr}{\mathbb{B}}
\newcommand{\cBr}{\textup{B}}
\newcommand{\cF}{\mathcal{F}}
\newcommand{\cG}{\mathcal{G}}
\newcommand{\cO}{\mathcal O}
\newcommand{\cN}{\mathcal N}
\newcommand{\cC}{\mathcal C}
\newcommand{\cD}{\mathcal D}
\newcommand{\cR}{\mathcal R}
\newcommand{\cM}{\mathcal M}

\newcommand{\tens}{\boxtimes}
\newcommand{\Hom}{\mathrm{Hom}}

\newcommand{\mh}{\mathrm{h}}

\numberwithin{equation}{section}

\title[Graphical calculus for quantum vertex operators, I]{Graphical calculus for quantum vertex operators, I:\\ The dynamical fusion operator}
\author{Hadewijch De Clercq}
\address{H.D.C.: Department of Electronics and Information Systems, Mathematical Analysis Research Group, Ghent University,
	Belgium.}
\email{hadewijch.declercq@ugent.be}

\author{Nicolai Reshetikhin}
\address{N.R.: Department of Mathematics, University of California, Berkeley,
	CA 94720, USA \& St. Petersburg University, Russia \&KdV Institute for Mathematics, University of Amsterdam,
	Science Park 904, 1098 XH Amsterdam, The Netherlands.}
\email{reshetik@math.berkeley.edu}

\author{Jasper Stokman}
\address{J.S.: KdV Institute for Mathematics, University of Amsterdam,
	Science Park 904, 1098 XH Amsterdam, The Netherlands.}
\email{J.V.Stokman@uva.nl }

\begin{document}

\begin{abstract}
	This paper is the first in a series on graphical calculus for quantum vertex operators. We establish in great detail the foundations of graphical calculus for ribbon categories and braided monoidal categories with twist. We illustrate the potential of this approach by applying it to various categories of quantum group modules, in particular to derive an extension of the linear operator equation for dynamical fusion operators, due to Arnaudon, Buffenoir, Ragoucy and Roche, to a system of linear operator equations of \(q\)-KZ type.
\end{abstract}

\maketitle

\tableofcontents

%%%%%%%%%%%%%%%%%%%%%%%%%%%%%%%%%%
\section*{Introduction}

Graphical calculus provides a diagrammatic framework for performing topological computations with morphisms in monoidal categories. This amounts to a functorial identification of such morphisms with oriented diagrams colored by the corresponding monoidal category. Topological moves in the diagrams translate to algebraic relations between the morphisms, thereby providing graphical insight into intricate algebraic identities. The structure of the category determines the class of diagrams and their allowed local graphical moves.

The goal of this paper is threefold. Firstly, we establish and fine-tune the foundations of the existing graphical calculus for ribbon categories. In the literature, contradictory conventions occur and the existing definitions often do not fully fit into the surrounding categorical setup. It is our aim to describe in full detail and rigor the required concepts and their functorial interpretations. To this end, we reconcile the original definition of the Reshetikhin-Turaev functor \cite{Reshetikhin&Turaev-1990} with its more commonly used definition from Turaev's book \cite{Turaev-1994}, which leads to a precise distinction between 
fusing and bundling of colored parallel strands in diagrams. 
The subtle difference between the two definitions lies in the fact that \cite{Reshetikhin&Turaev-1990} takes as coloring category any ribbon category \(\cC\), whereas \cite{Turaev-1994} departs from a strict ribbon category \(\cD\). In case \(\cD\) is the strictification of \(\cC\), the diagrammatic category in \cite{Reshetikhin&Turaev-1990} becomes a subcategory of the one in \cite{Turaev-1994}, and bundling of parallel strands only makes sense in the category with colors from the strictification of $\cC$.

Secondly, we extend the graphical calculus to a class of categories with less structure, namely braided monoidal categories \(\cD\) equipped with a twist automorphism. This gives rise to an analog of the Reshetikhin-Turaev functor, allowing to perform diagrammatic computations with morphisms in \(\cD\) involving 
\(\cD\)-colored ribbon-braid graph diagrams. 
On the level of representations of a quantum group \(U_q(\g)\), this allows to extend the graphical calculus from the ribbon category \(\Rep\) of finite-dimensional \(U_q(\g)\)-modules to the category, denoted by \(\cM\), which we will introduce in Definition \ref{Cdef}. The category \(\cM\) is a braided monoidal category with twist, which encompasses the \(q\)-analog of the BGG category \(\cO\) and in particular contains the Verma modules \(M_\lambda\). We carefully establish the braiding and twist for \(\cM\). This requires the description of the universal R-matrix and Drinfeld's ribbon element, as well as their fundamental properties, within appropriate 
completions of multi-tensor products of \(U_q(\g)\).  The category $\cM$ is no longer a ribbon category by the lack of duality. 

Thirdly, we provide tools to perform graphical computations with a special class of morphisms, called quantum vertex operators, in the module category $\cM$. These quantum vertex operators are \(U_q(\g)\)-intertwiners of the form \(\phi: M_\lambda\to M_\mu\otimes V\), where the \emph{auxiliary spaces} \(M_\lambda\), \(M_\mu\) are irreducible Verma modules and the \emph{spin space} \(V\) is finite-dimensional. The terminology is motivated by the analogy with vertex operators in Wess-Zumino-Witten (WZW) conformal field theory.  In case \(V\) is itself a multifold tensor product of finite-dimensional \(U_q(\g)\)-modules \(V_1,\dots,V_k\), then we will refer to \(\phi\) as a \(k\)-point quantum vertex operator if it can be written as a composition of \(k\) individual quantum vertex operators, with each of the \(V_i\) occurring once as the spin space. 

A quantum vertex operator is uniquely labeled by its expectation value, which is a weight vector in the spin space. We introduce the following natural graphical notation for the coupon colored by the quantum vertex operator \(\phi:M_\lambda\to M_\mu\otimes V\), with expectation value \(v\in V\), in the diagrammatic category:
\begin{center}
	\includegraphics[scale = 0.85]{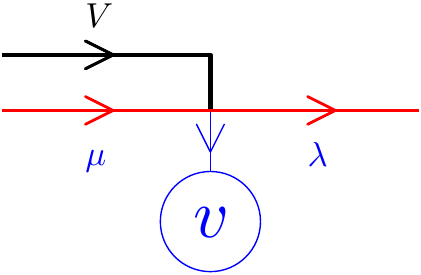}
\end{center}
Analogous graphical notations for \(k\)-point vertex operators will be given in Figure \ref{general intertwiner2}. These notations will prove particularly useful in the sequel \cite{DeClercq&Reshetikhin&Stokman-2021} of this paper, when we extend the graphical calculus to the parametrizing spaces of the quantum vertex operators. 

Other important classes of quantum vertex operators are those with auxiliary spaces in $\Rep$, which relate to multivariate orthogonal polynomials \cite{Etingof&Kirillov-1994,Kirillov-1996}, and (quantum) vertex operators with auxiliary spaces taken from the category of (quantum) Harish-Chandra modules, cf.\ e.g.\ \cite{Reshetikhin&Stokman-2020-A}.

This novel framework of graphical calculus for quantum vertex operators promises to be of great value for future research in representation theory of quantum groups and its applications to quantum integrable systems and multivariate special functions. Its power lies in the fact that it allows one to derive intricate algebraic relations between morphisms in the considered categories in an insightful fashion, based on topological considerations rather than tedious algebraic computations. This paper is the first in a series aimed at demonstrating this potential, by providing purely graphical proofs for algebraic identities. These include both known \(q\)-difference equations that were obtained in previous literature through lengthy algebraic calculations \cite{Etingof&Varchenko-2000}, as well as novel generalizations of such equations. This will be initiated in Section \ref{Section Generalized ABRR}, where we derive for each fixed generic highest weight the consistent system of quantum Knizhnik-Zamolodchikov (\(q\)-KZ) type equations for dynamical fusion operators, which describe the fusion of a \(k\)-point vertex operator. These equations are of crucial importance in the sequel to this paper \cite{DeClercq&Reshetikhin&Stokman-2021}, where we revisit Etingof's and Varchenko's \cite{Etingof&Varchenko-2000} normalized generalized trace functions, which arise from the \(k\)-point quantum vertex operators upon applying twisted cyclic boundary conditions. Concretely, we will give in \cite{DeClercq&Reshetikhin&Stokman-2021} completely transparent, graphical proofs of the dual Macdonald-Ruijsenaars (MR) and the dual quantum Knizhnik-Zamolodchikov-Bernard (\(q\)-KZB) equations for generalized trace functions, which were obtained in \cite{Etingof&Varchenko-2000} through intricate algebraic derivations. This approach will in fact give rise to novel, mutually commuting, dual MR type \(q\)-difference operators, in addition to the ones obtained in \cite{Etingof&Varchenko-2000}. This leads to a supplementary class of eigenvalue equations for the generalized trace functions, one for each component of the parametrizing spin space. This provides additional conserved quantities to 
the associated quantum integrable system, thereby turning it into a quantum superintegrable system. The latter can be regarded as a quantization of the classical superintegrable system on moduli spaces of flat connections over surfaces, as constructed in \cite{Artamonov&Reshetikhin} in the case where the surface is a punctured torus. 

Another important result of our upcoming paper \cite{DeClercq&Reshetikhin&Stokman-2021} will be the introduction of a dynamical twist functor. This tensor functor translates the action of a morphism in \(\Rep\) on the spin space of a \(k\)-point quantum vertex operator to an action on its expectation value, thereby allowing to assign to each morphism of \(\Rep\) a dynamical counterpart. In particular, it maps R-matrices to the corresponding dynamical R-matrices from \cite{Etingof&Schiffmann-2001}. We will extend the graphical calculus in \cite{DeClercq&Reshetikhin&Stokman-2021} by including graphical notations for the dynamical twist functor and the dynamical R-matrices, which will subsequently be used in the derivation of the dual MR and \(q\)-KZB equations.

Classical KZB operators form a system of mutually commuting first order differential operators for correlation functions of WZW conformal field theory on a torus \cite{Bernard-1988, Felder&Weiczerkowski-1996}. They are expressed in terms of position variables on a torus and dynamical variables from the regular part of $G/\mr{Ad}_G$, with \(G\) the underlying Lie group. The space of dynamical variables can be regarded as the base of a Lagrangian fibration of the moduli space of flat connections on a torus. In case the torus degenerates to a cylinder, such that the corresponding elliptic functions degenerate to trigonometric functions, the dynamical variables remain in $G/\mr{Ad}_G$, leading to trigonometric versions of KZB operators. Both in the elliptic and the trigonometric case, the class of KZB operators has a natural quantum analog, consisting of mutually commuting \(q\)-difference operators that relate to the representation theory of the corresponding quantum affine algebra \(U_q(\widehat{\g})\).

Quantum KZ and quantum KZB equations are solved by matrix coefficients or twisted traces of compositions of quantum vertex operators \cite{Frenkel&Reshetikhin-1992, Etingof&Varchenko-2000}. The derivation of the \(q\)-KZ equations follows the following steps. First of all, one derives an operator-valued \(q\)-KZ equation for a single quantum vertex operator using the quantum Casimir of the quantum affine algebra \cite[Theorem 5.1]{Frenkel&Reshetikhin-1992}. Subsequently, one obtains operator \(q\)-KZ equations for compositions of quantum vertex operators \cite[Theorem 5.2]{Frenkel&Reshetikhin-1992}. Finally, one resolves the terms interacting with the outer auxiliary spaces by considering matrix coefficients or taking a twisted trace, as in \cite[Theorem 5.3]{Frenkel&Reshetikhin-1992} and \cite{Etingof&Schiffmann&Varchenko-2002}. The last step should be thought of as imposing boundary conditions on the associated spin chain model. See \cite{Etingof&Schiffmann-1999}, \cite[Section 6]{Reshetikhin&Stokman-2020-A} and \cite[Subsection 2.1]{Stokman-2021} for this approach in the semiclassical limit with the position variables sent to infinity. In this case one obtains topological KZB equations for compositions of asymptotic vertex operators when twisted cyclic or reflecting boundary conditions are imposed. 

From this perspective, this paper deals with the quantum group version when highest weight to highest weight boundary conditions are imposed, in the limit where the position variables are sent to infinity. As a result, the role of the quantum affine algebra \(U_q(\widehat{\g})\) is taken over by \(U_q(\g)\). We provide proofs entirely involving the graphical calculi with colorings from the braided monoidal category $\cM$ with twist, from the ribbon category $\Rep$ and, after imposing the boundary conditions, from the symmetric tensor category of finite-dimensional $\mathfrak{h}^*$-graded vector spaces. The outcome is a system of linear equations for dynamical fusion operators, extending the Arnaudon-Buffenoir-Ragoucy-Roche equation from \cite{Arnaudon&Buffenoir&Ragoucy&Roche-1998, Etingof&Schiffmann-2001}, which may be viewed as the topological limit of the \(q\)-KZB equations for the generalized trace functions from \cite{Etingof&Varchenko-2000} when the geometric parameters are sent to infinity deep in the positive Weyl chamber. 

The outline of the paper is as follows. In Section \ref{repSection} we first recall some preliminaries on quantum groups, monoidal categories and categories of quantum group representations. This is where we give precise definitions for the categories \(\cM\) and \(\Rep\) that are studied throughout the paper. In Subsections \ref{Section R-matrix} and \ref{Section ribbon element} we will explain how topological extensions of multifold tensor products of quantum groups can be constructed in such a way that they allow to define a braiding and a twist on the category \(\cM\). Duality in \(\Rep\) will be touched upon in Subsection \ref{dualsection}. Section \ref{GcSection} is dedicated to the foundations of the graphical calculus. Subsections \ref{tangles} and \ref{topSection} establish the definitions of the diagrams that will serve as the building blocks for the respective graphical calculi. In Subsection \ref{StSection} we recall how one can assign to each monoidal category a strict monoidal category, and how the latter inherits structure from the former. Subsections \ref{GcBraid} and \ref{Section RT-functor} explain how to assign colors to the diagrams in a functorial fashion, and contain the Reshetikhin-Turaev functor for ribbon categories in the sense of \cite{Reshetikhin&Turaev-1990} and its analog for braided monoidal categories with twist. It is explained how the two graphical calculi can be fit together in Subsection \ref{mixedsection}. Subsections \ref{Section 2 strictified} and \ref{Section bundling} are concerned with the relation between fusing and bundling of parallel strands. Graphical notations for the quantum vertex operators will be introduced in Subsection \ref{Section 2 C^+}. Finally, we derive the topological \(q\)-KZ equations, first as operator equations for \(k\)-point quantum vertex operators in Subsection \ref{sectionqKZ} and subsequently as linear operator equations for dynamical fusion operators in Subsection \ref{Section expectation value identity}.\newline
%%%%%%%%%%%%%%%%%%%%%%%%%%%%%%%%%%

\noindent
{\bf Acknowledgments:} H.D.C. is a PhD fellow of the Research Foundation Flanders (FWO). The work of J.S. and N.R. was supported by the Dutch Research Council (NWO). The work of N.R. was supported by the NSF grant DMS-1902226 and by the RSF grant 18-11-00-297.

%%%%%%%%%%%%%%%%%%%%%%%%%%%%%%%%%%%%%%%%
\section{Quantum groups and their representations}\label{repSection}
%%%%%%%%%%%%%%%%%%%%%%%%%%%%%%%%%%%%%%%

%%%%%%%%%%%%%%%%%%%%%%%%%%%%%%%%%
\subsection{The quantized universal enveloping algebra $U_q(\g)$}
\label{Section Hopf algebra}
%%%%%%%%%%%%%%%%%%%%%%%%%%%%%%%%

Let \(\g\) be a semisimple finite-dimensional Lie algebra over \(\CC\)
with Cartan matrix \(A = (a_{ij})_{i,j=1,\dots,r}\). The matrix \(A\) is symmetrizable, i.e.\ there exists a diagonal matrix \(D = \mathrm{diag}(d_i)_{i=1,\dots,r}\), with mutually coprime and positive integer entries \(d_i\), such that \(DA\) is a symmetric matrix. Let \(\mathfrak{b}\) be a Borel subalgebra of \(\g\), and \(\hh\) a Cartan subalgebra contained in \(\mathfrak{b}\). We write \(\Pi = \{\alpha_i: i = 1,\dots,r\}\) for the corresponding set of simple roots
of the root system $\Phi\subset\mathfrak{h}^*$ of $\mathfrak{g}$ relative to $\mathfrak{h}$, and \(\Phi^+\) for the set of positive roots. We write \(Q=\mathbb{Z}\Pi\) for the root lattice, 
\(\Lambda\) for the lattice of integral weights and \(\Lambda^+\subset \Lambda\) for the set of dominant integral weights. Furthermore we set \(Q^{\pm} = \ZZ_{\pm}\Pi\).
Let \(\{h_i: i = 1,\dots,r\}\) be the linearly independent subset of \(\hh\) defined by \(\alpha_j(h_i) = a_{ij}\). Define a non-degenerate symmetric bilinear form $\langle \cdot,\cdot\rangle$ on $\mathfrak{h}^*$ by $\langle \alpha_i,\alpha_j\rangle:=d_ia_{ij}$. It satisfies $\langle \alpha,\alpha\rangle=2$ for all short roots $\alpha\in\Phi$.
It is, up to a constant multiple, the bilinear form on $\mathfrak{h}^*$ obtained by dualizing the restriction of the Killing form of $\mathfrak{g}$ to $\mathfrak{h}\times\mathfrak{h}$. 
In particular it is the complex bilinear extension of a scalar product on \(\bigoplus_{i=1}^r\mathbb{R}\alpha_i\).
Hence there exists a basis \(\{x_i: i=1,\ldots,r\}\) of \(\mathfrak{h}\) 
such that 
\[
\langle \mu,\nu\rangle=\sum_{i=1}^r\mu(x_i)\nu(x_i)
\]
for \(\mu,\nu\in\hh^*\). In what follows we identify $\mathfrak{h}^*$ with $\mathfrak{h}$ as linear spaces via the map $\mu\mapsto \sum_{i=1}^r\mu(x_i)x_i$. The isomorphism satisfies
$\alpha_i\mapsto d_ih_i$ for \(i=1,\ldots,r\). 

Fix \(\tau\in\CC\) with \(\textup{Im}(\tau)>0\). For \(c\in\CC\) and positive integers \(m\geq k\) we write 
\[
[c]_q := \frac{q^c-q^{-c}}{q-q^{-1}},\qquad 
[m]_q! := \prod_{\ell = 1}^m[\ell]_q, \qquad \begin{bmatrix}
m \\ k
\end{bmatrix}_q := \frac{[m]_q!}{[k]_q![m-k]_q!},
\]
where \(q^c:=e^{c\tau}\).

The quantized universal enveloping algebra or quantum group \(U_q := U_q(\g)\) is the unital associative \(\CC\)-algebra generated by the elements \(q^h\), with \(h\in \hh\), together with \(E_i\) and \(F_i\), with \(i = 1,\dots,r\), subject to the relations
\begin{equation}\label{qrel}
\begin{gathered}
q^0=1, \qquad q^hq^{h'} = q^{h+h'}, \qquad q^h E_i = q^{\alpha_i(h)} E_iq^h, \qquad q^hF_i = q^{-\alpha_i(h)}F_iq^h, \qquad \\
[E_i, F_j] = \delta_{ij} \frac{q^{d_ih_i}-q^{-d_ih_i}}{q^{d_i}-q^{-d_i}}, \\\sum_{k = 0}^{1-a_{ij}}(-1)^k\begin{bmatrix}
1-a_{ij} \\ k
\end{bmatrix}_{q^{d_i}} E_i^{1-a_{ij}-k}E_jE_i^{k} = \sum_{k = 0}^{1-a_{ij}}(-1)^k\begin{bmatrix}
1-a_{ij} \\ k
\end{bmatrix}_{q^{d_i}} F_i^{1-a_{ij}-k}F_jF_i^{k} = 0
\end{gathered}
\end{equation}
for any \(h,h'\in\hh\) and \(i,j = 1,\dots,r\). 

\(U_q\) is a Hopf algebra with comultiplication \(\Delta\), counit \(\epsilon\) and antipode \(S\) determined by the relations
\begin{align}
\label{Compultiplication, antipode, counit}
\begin{alignedat}{6}
\Delta(E_i) & = E_i\otimes q^{d_ih_i} + 1\otimes E_i, \quad& \Delta(F_i) &= F_i\otimes 1 + q^{-d_ih_i}\otimes F_i, \quad& \Delta(q^h) &= q^h\otimes q^h, \\
\epsilon(E_i) &=  0, & \epsilon(F_i) &= 0, & \epsilon(q^h) &= 1, \\
S(E_i) &= -E_iq^{-d_ih_i}, & S(F_i) &= -q^{d_ih_i}F_i, & S(q^h) &= q^{-h}.
\end{alignedat}
\end{align}
Here we follow the original Drinfeld \cite{Drinfeld-1986} convention\footnote{Drinfeld used generators $X_i=E_ie^{-\frac{hd_ih_i}{4}}$, $Y_i=e^{\frac{hd_ih_i}{4}}F_i$ and $h_i$, and defined the algebra over $\CC[[h]]$ with $q=e^{h/2}$.}
 for the comultiplication, which is also used in \cite{Etingof&Latour-2005,Etingof&Varchenko-2000}. 
In \cite{Lusztig-1994,Balagovic-Kolb-2019} the opposite comultiplication \(\Delta^{\mr{op}} = P \circ \Delta\) is used, where $P$ is the permutation map  \(P: U_q^{\otimes 2}\to U_q^{\otimes 2}: a\otimes b \mapsto b\otimes a\). 

Denote by \(\rho\in\Lambda\) half the sum of the positive roots. Then
\begin{equation}
\label{action of q^2rho}
q^{2\rho} X = S^2(X)q^{2\rho}
\end{equation}
for any \(X\in U_q\). 

Let \(\g = \nf^+\oplus\hh\oplus\nf^- \) be the triangular decomposition of \(\g\) relative to the choice \(\Phi^+\) of positive roots, so that \(\bfr= \nf^+\oplus\hh\). We write \(\bfr^- = \hh\oplus\nf^-\) for the negative Borel subalgebra.
Let \(U^0:=U_q(\hh)\) be the quantum Cartan subalgebra of \(U_q(\g)\), which is generated by the elements \(q^h\), \(h\in \hh\), and denote by \(U^{\pm}:=U_q(\nf^{\pm})\) the subalgebras of \(U_q(\g)\) generated by the sets \(\{E_i: i = 1,\dots,r\}\) and \(\{F_i: i = 1,\dots,r\}\) respectively. The quantized Borel subalgebras are defined by
\[
U_q(\bfr):=U^+U^0, \qquad U_q(\bfr^-):=U^0U^-.
\]
For \(\beta\in Q^{\pm}\) we write \(U^{\pm}[\beta]\) for the set of elements \(X\in U^{\pm}\) satisfying \(q^hXq^{-h}=q^{\beta(h)}X\) for all \(h\in\hh\). 

%%%%%%%%%%%%%%%%%%%%%%%%%%%%%%%%%%%%
\subsection{Braided monoidal categories with twist and ribbon categories}\label{bmsection}
%%%%%%%%%%%%%%%%%%%%%%%%%%%%%%%%%%%
We recall in this subsection some basic notions on monoidal and ribbon categories, which we will use to describe various module categories over $U_q(\mathfrak{g})$. 

Consider a mo\-noi\-dal category $\mathcal{D}=(\mathcal{D},\otimes,\mathbb{1},a,\ell,r)$ with tensor product $\otimes$, unit object $\mathbb{1}$, associativity constraint $a$, and left and right unit constraints $\ell$ and $r$. We will omit the associators and unit constraints in formulas, since the way in which the removing and adding of unit objects and the re-bracketing of tensor products is performed has no effect on the outcome, due to Mac Lane's coherence theorem.  

A {\it left duality} for a monoidal category $\cD$ is an assignment of a triple $(V^*,e_V,\iota_V)$ to each $V\in\cD$, consisting of an object $V^*\in\cD$
and of morphisms $e_V: V^*\otimes V\rightarrow \mathbb{1}$ and $\iota_V: \mathbb{1}\rightarrow V\otimes V^*$ satisfying
\[
(e_V\otimes\textup{id}_{V^*})(\textup{id}_{V^*}\otimes \iota_V)=\textup{id}_{V^*},\qquad
(\textup{id}_V\otimes e_V)(\iota_V\otimes\textup{id}_V)=\textup{id}_V.
\]
We call $e_V$ the evaluation morphism of $V$, and $\iota_V$ the injection morphism of $V$.
%%%%%%%%%%%%%%%%%%%%%%
\begin{remark}
	We follow here the conventions set in 
	\cite{Deligne-2002}. Other common notations for the morphisms $e_V, \iota_V$ are \(\mr{ev}_V\), \(\mr{coev}_V\) \textup{(}see e.g.\ \cite{Etingof&Gelaki&Nikshych&Ostrik-2015}\textup{)} and \(d_V\), \(b_V\) \textup{(}see \cite{Turaev-1994, Kassel-1995, Kassel&Rosso&Turaev-1997}\textup{)}. The injection morphisms are also called co-evaluation morphisms.
\end{remark}
%%%%%%%%%%%%%%%%%%%%%%
Left duality allows one to define the dual $A^*\in\textup{Hom}_{\cD}(W^*,V^*)$ of a morphism $A\in\textup{Hom}_{\cD}(V,W)$ by
\begin{equation}
\label{A transpose def}
A^\ast := (e_W\otimes\id_{V^\ast})(\id_{W^\ast}\otimes A\otimes\id_{V^\ast})(\id_{W^\ast}\otimes\iota_V).
\end{equation}

In a similar way, a {\it right duality} for a monoidal category $\cD$ is an assignment of a triple $(V^\star,\widetilde{e}_V,\widetilde{\iota}_V)$ for each $V\in\cD$
consisting of an object $V^\star\in\cD$
and of morphisms $\widetilde{e}_V: V\otimes V^\star\rightarrow \mathbb{1}$ and $\widetilde{\iota}_V: \mathbb{1}\rightarrow V^\star\otimes V$ satisfying
\[
(\widetilde{e}_V\otimes\textup{id}_{V})(\textup{id}_{V}\otimes\widetilde{\iota}_V)=\textup{id}_{V},\qquad
(\textup{id}_{V^\star}\otimes \widetilde{e}_V)(\widetilde{\iota}_V\otimes\textup{id}_{V^\star})=\textup{id}_{V^\star}.
\]
A monoidal category with a left and right duality is said to be {\it rigid}.

A monoidal category $\cD$ is said to be {\it braided} with commutativity constraint $c=(c_{V,W})_{V,W\in\cD}$ if the $c_{V,W}$ are isomorphisms $V\otimes W\overset{\sim}{\longrightarrow}W\otimes V$ satisfying 
\begin{equation}\label{hi}
\begin{split}
(B\otimes A)c_{V,W}&=c_{V^\prime,W^\prime}(A\otimes B),\\
c_{U,V\otimes W}&=(\textup{id}_V\otimes c_{U,W})(c_{U,V}\otimes\textup{id}_W),\\
c_{U\otimes V,W}&=(c_{U,W}\otimes\textup{id}_V)(\textup{id}_U\otimes c_{V,W})
\end{split}
\end{equation}
for $A\in\textup{Hom}_{\cD}(V,V^\prime)$ and $B\in\textup{Hom}_{\cD}(W,W^\prime)$.
The latter two identities are called the hexagon identities.
Recall the well-known fact that the hexagon relations and the naturality of the commutativity constraint result in the braid form of the
Yang-Baxter equation
\begin{equation}
\label{braid form YB}
(c_{V,W}\otimes\textup{id}_U)(\textup{id}_V\otimes c_{U,W})(c_{U,V}\otimes\textup{id}_W)=(\textup{id}_W\otimes c_{U,V})(c_{U,W}\otimes\textup{id}_V)(\textup{id}_U\otimes c_{V,W}).
\end{equation}

%%%%%%%%%%%%%%%%%%%%%%%%%%%%
\begin{definition}\label{scdef}
Let $\cD$ be a braided monoidal category. 
We call a natural automorphism $\theta=(\theta_V)_{V\in\cD}$ of $\textup{id}_{\cD}$ a twist for $\cD$ if it satisfies
\[
\theta_{V\otimes W}=(\theta_V\otimes\theta_W)c_{W,V}c_{V,W}
\]
for all $V,W\in\cD$.
\end{definition}
%%%%%%%%%%%%%%%%%%%%%%%%%%%%%

A {\it ribbon category} \cite{Drinfeld-1990} is a braided monoidal category $\cD$ with twist $\theta$ and left duality satisfying the compatibility condition
\[
\theta_V^*=\theta_{V^*}
\]
for all $V\in\cD$.

A ribbon category $\cD$ admits a right duality, turning it into a rigid category: for $V\in\cD$, the right dual is $(V^*,\widetilde{e}_V,\widetilde{\iota}_V)$ with (see \cite{Reshetikhin&Turaev-1990})
\begin{equation}
\label{right dualities}
\widetilde{e}_V=e_Vc_{VV^\ast}(\theta_V\otimes\textup{id}_{V^\ast}),\qquad \widetilde{\iota}_V=(\textup{id}_{V^\ast}\otimes\theta_V)c_{VV^\ast}\iota_V.
\end{equation}

%%%%%%%%%%%%%%%%%%%%%%%%%%%
\subsection{Categories of $U_q(\mathfrak{g})$-modules}
\label{Section 1.2}
%%%%%%%%%%%%%%%%%%%%%%%%%%%

Let \(\textup{Mod}_{U_q}\) be the category of left \(U_q\)-modules. The representation map for \(M\in\textup{Mod}_{U_q}\) will be denoted by \(\pi_M\). For any \(M,N\in \textup{Mod}_{U_q}\) we write \(\Hom_{U_q}(M,N)\) for the corresponding morphism space.
We call morphisms in \(\textup{Mod}_{U_q}\) intertwiners, or \(U_q\)-linear maps.

The weight space \(M[\mu]\) of \(M\in\textup{Mod}_{U_q}\) of weight \(\mu\in\hh^\ast\) is
\[
M[\mu]:=\{m\in M \,\, | \,\, q^hm=q^{\mu(h)}m\qquad \forall\, h\in\hh\}.
\]
Let us write
\[
\wts(M) = \{\mu\in\hh^\ast\,\, | \,\,  M[\mu]\neq \{0\} \}
\]
for the set of weights of \(M\).
Note that \(M^\prime:=\bigoplus_{\mu\in\hh^*}M[\mu]\) is a \(U_q\)-submodule of \(M\). We endow \(M^\prime\) with a compatible semisimple \(\hh\)-action by
\[
h\vert_{M[\mu]}:=\mu(h)\textup{id}_{M[\mu]}\qquad (\mu\in\hh^\ast).
\]
We will say that \(M\in\textup{Mod}_{U_q}\) is {\it \(\hh\)-semisimple} if \(M^\prime=M\). 
A linear basis of an $\hh$-semisimple \(U_q\)-module $M$ is said to be homogeneous if it is of the form
\[
\{b_i^{(\nu)}\,\, \vert\,\, \nu \in \mr{wts}(M),\,\ i\in I_\nu\}
\]
with \(\{b_i^{(\nu)}\}_{i\in I_\nu}\) linear bases of the weight spaces \(M[\nu]\) for all $\nu$. Throughout the whole paper, we will write \(\mathcal{B}_M\) to denote a homogeneous basis for an $\mathfrak{h}$-semisimple $U_q$-module \(M\).

A $U_q$-module $M$ is said to be locally $U^+$-finite if for all \(m\in M\) the $U^+$-submodule of \(M\) generated by $m$ is finite-dimensional.
%%%%%%%%%%%%%%%%%
\begin{definition}\label{Cdef}
We denote by \(\cM\) the full subcategory of \(\textup{Mod}_{U_q}\) consisting of the \(\hh\)-semisimple locally $U^+$-finite \(U_q\)-modules.
\end{definition}
%%%%%%%%%%%%%%%%%

The representation category \(\cM\) is abelian. Furthermore, $\cM$ is a monoidal subcategory of the monoidal category of complex vector spaces, with the \(U_q\)-action
on \(M\otimes N\) (\(M,N\in\cM\)) defined by 
\[
\pi_{M\otimes N}(X):=(\pi_M\otimes\pi_N)(\Delta(X))\qquad (X\in U_q).
\]
As unit object \(\mathbb{1}\) we take the one-dimensional vector space $\CC$ endowed with the \(U_q\)-action \(X\cdot\lambda:=\epsilon(X)\lambda\) for \(X\in U_q\) and \(\lambda\in\CC\).
Recall that the associativity isomorphisms \(a_{L,M,N}: (L\otimes M)\otimes N\to L\otimes (M\otimes N)\) amount to repositioning of brackets in pure tensors, and the left and right constraints \(\ell_M: \mathbb{1}\otimes M \to M\) and \(r_M: M\otimes\mathbb{1}\to M\) are the maps \(\lambda\otimes m\mapsto\lambda m\) and \(m\otimes\lambda\mapsto \lambda m\) for \(\lambda\in\CC\) and \(m\in M\). 
%%%%%%%%%%%%%%%%%%%%%%%%%%%%%%%%%%%%%%%%%%%%%%%%%
\begin{definition}\label{subcat}
We consider the following full subcategories of $\cM$:
\begin{enumerate}
\item $\cM_{\textup{adm}}$: the \(U_q\)-modules \(M\) in \(\cM\) satisfying
\begin{enumerate}
\item $\textup{dim}(M[\mu])<\infty$ for all $\mu\in\mathfrak{h}^*$,
\item\label{weight condition} $\textup{wts}(M)\subseteq \bigcup_{i=1}^k\{\lambda_i-Q^+\}$ for certain $\lambda_i\in\mathfrak{h}^*$ and \(k\in \mathbb{Z}_+\).
\end{enumerate}
\item $\cO$: the {\it finitely generated} \(U_q\)-modules in \(\cM\).
\item $\Rep$: the finite-dimensional \(U_q\)-modules in \(\cM\).
\end{enumerate}
\end{definition}
%%%%%%%%%%%%%%%%%%%%%%%%%%%%%%%%%%%%%%%%%%%%%%%%%
We call $\cM_{\textup{adm}}$ the subcategory of admissible $U_q$-modules in $\cM$.
It is an abelian, monoidal subcategory, closed under taking submodules. 

The representation category $\cO$ is the quantum group analog of the Bernstein-Gelfand-Gelfand category relative to \(\mathfrak{b}\). Since $U_q$ is Noetherian, $\cO$ is abelian and closed under taking submodules, but it is not monoidal. A standard argument using the PBW theorem for $U_q$ shows that $\cO$ is a subcategory of $\cM_{\textup{adm}}$ (cf.  \cite[\S 1.1]{Humphreys-2008}). The cyclic $U_q$-submodules of modules $M\in\mathcal{M}$ lie in $\cO$.

The representation category \(\Rep\) is a semisimple, abelian, monoidal subcategory of \(\cM\). It is a subcategory of $\cO$, and $\cO$ is a bimodule category over $\Rep$.  
The ribbon algebra structure of $U_q$ provides additional structures on the representation categories, which will be discussed in the upcoming subsections.

For $\lambda\in\mathfrak{h}^*$ denote by \(\CC_{\lambda}=\mathbb{C}1_\lambda\) the one-dimensional \(U_q(\bfr)\)-module defined by \(E_i1_{\lambda} = 0\) and \(q^h1_{\lambda} = q^{\lambda(h)} 1_{\lambda}\) for any \(i = 1,\dots,r\) and \(h\in\hh\). The induced module
\[
M_{\lambda} = \mr{Ind}_{U_q(\bfr)}^{U_q}\CC_{\lambda}
\]
is the Verma module with highest weight \(\lambda\). Let us denote its representation map by \(\pi_\lambda\). We write \(\mathbf{m}_{\lambda}:=1\otimes_{U_q(\mathfrak{b})}1_\lambda\in M_\lambda[\lambda]\), which is a highest weight vector for $M_\lambda$. In particular, $\mathbf{m}_\lambda$ is a cyclic vector for $M_\lambda$, and 
$\textup{wts}(M_\lambda)=\lambda-Q^+$.
The Verma modules \(M_\lambda\) (\(\lambda\in\hh^\ast\)) are the standard modules in category $\mathcal{O}$. 

The Verma module $M_\lambda$ has a unique irreducible quotient, which we denote by $L_\lambda$. By abuse of notation  
the push-forward of the highest weight vector \(\mathbf{m}_\lambda\in M_\lambda\) to \(L_\lambda\) is again denoted by \(\mathbf{m}_\lambda\).
The \(\{L_\lambda\}_{\lambda\in\hh^\ast}\) exhaust the simple $U_q$-modules in $\cO$ up to isomorphism, and \(\{L_\lambda\}_{\lambda\in\Lambda^+}\) is a complete set of representatives of the isoclasses of simple $U_q$-modules in $\Rep$. In particular, \(\textup{wts}(V)\subset \Lambda\) for all \(V\in\Rep\).

The Verma module \((M_\lambda,\pi_\lambda)\) is irreducible if and only if 
\begin{equation}
\label{irreducibility criterion}
\langle \lambda+\rho,\alpha^\vee\rangle\notin \mathbb{Z}_{>0}\qquad \forall\,\alpha\in\Phi^+,
\end{equation}
where \(\alpha^\vee = \frac{2\alpha}{\langle \alpha,\alpha\rangle}\).  We will call a weight \(\lambda\in\hh^\ast\) {\it generic} if the more stringent conditions
\[
\langle \lambda,\alpha^\vee\rangle\notin \ZZ\qquad \forall\,\alpha\in\Phi
\]
hold true. The set of generic weights is denoted by \(\hh^\ast_{\textup{reg}}\). It is stable under translation by integral weights. As a consequence, if \(\lambda\in\hh^\ast_{\textup{reg}}\) then \(M_{\lambda+\nu}\) is irreducible for all \(\nu\in\Lambda\). 
%%%%%%%%%%%%%%%%%%%%%%%%%%%%%%%
\subsection{The universal R-matrix and braiding}
\label{Section R-matrix}
%%%%%%%%%%%%%%%%%%%%%%%%%%%%%%%%
In this subsection we discuss how Drinfeld's \cite{Drinfeld-1986} universal $R$-matrix $\mathcal{R}$ for $U_q$ gives rise to a commutativity constraint for the monoidal category $\cM$ (see Definition \ref{Cdef}). 

Drinfeld defined the universal $R$-matrix as an element of $U_h(\g)\widehat{\otimes} U_h(\g)$, where $\widehat{\otimes}$ is the tensor product completed over formal power series in $h$, with $U_h(\g)$ the $\CC[[h]]$-complete algebra generated by $h_i, E_i, F_i$ with relations determined by (\ref{qrel}) and $q=e^{h/2}$. To describe the universal $R$-matrix in the category $\cM$ we require a completion of $U_q^{\otimes 2}$ that is suitable for the tensor products in $\mathcal{M}$. For this we will follow Lusztig \cite{Lusztig-1994}.

Let $\textup{ht}: Q^+\rightarrow\mathbb{Z}_+$ be the height function, $\textup{ht}(\sum_{i=1}^r\ell_i\alpha_i):=\sum_{i=1}^r\ell_i$.
For $k\in\mathbb{Z}_{>0}$ and $\ell\in\mathbb{Z}_+$ consider the subspace
\begin{equation}\label{Hspaces}
\mathcal{H}_\ell^{(k)}:=\sum_{j=1}^k\bigl(U_q^{\otimes (j-1)}\otimes U^-U^0U_\ell^+\otimes U_q^{\otimes (k-j)}\bigr)
\end{equation}
of $U_q^{\otimes k}$ with $U_\ell^+:=\bigoplus_{\beta\in Q^+: \textup{ht}(\beta)\geq \ell}U^+[\beta]$. It defines a filtration of left ideals 
\[
U_q^{\otimes k}=\mathcal{H}_0^{(k)}\supseteq\mathcal{H}_1^{(k)}\supseteq\cdots\supseteq\mathcal{H}_\ell^{(k)}\supseteq\cdots
\]
such that $\bigcap_{\ell\in\mathbb{Z}_+}\mathcal{H}_\ell^{(k)}=\{0\}$. Define the topology on $U_q^{\otimes k}$  by requiring that the subsets $\{X+\mathcal{H}_\ell^{(k)} \, | \, \ell\in\mathbb{Z}_+\}$ form a basis of open neighborhoods of $X$ for each $X\in U_q^{\otimes k}$. It defines a metric topology, which turns $U_q^{\otimes k}$ into a topological associative algebra  (cf. \cite[\S 4.1.1]{Lusztig-1994}).

Consider the inverse limit space
\[
(U_q^{\otimes k})^{\widetilde{}}:=\underset{\longleftarrow}{\lim}\, U_q^{\otimes k}/\mathcal{H}_\ell^{(k)}
\]
with its natural metric topology.
It provides an explicit realization of the completion of $U_q^{\otimes k}$, with the associated continuous embedding
$U_q^{\otimes k}\hookrightarrow (U_q^{\otimes k})^{\widetilde{}}$ defined by 
\[
X\mapsto (X+\mathcal{H}_0^{(k)},\ldots,X+\mathcal{H}_\ell^{(k)},\ldots).
\]
We will use the standard infinite sum notation for elements in $(U_q^{\otimes k})^{\widetilde{}}$\,. For example, if $X_\beta\in U^-U^0(U^+[\beta])\otimes U_q^{\otimes (k-1)}$ ($\beta\in Q^+$), then we will write
\[
\sum_{\beta\in Q^+}X_\beta
\]
for the sequence in $(U_q^{\otimes k})^{\widetilde{}}$\, whose $\ell^{\textup{th}}$ entry is equal to $\sum_{\beta\in Q^+: \textup{ht}(\beta)<\ell}X_\beta+\mathcal{H}_\ell^{(k)}$.

By extending the multiplication of $U_q^{\otimes k}$ to $(U_q^{\otimes k})^{\,\widetilde{}}\,$ by continuity, \((U_q^{\otimes k})^{\widetilde{}}\,\) becomes a complete topological associative algebra.

%%%%%%%%%%%%%%%%%%%%%%%%%%%%%%%%%%%%%%
\begin{lemma}\label{ExtensionLemma1}
For $M_i\in\cM$ \textup{(}$1\leq i\leq k$\textup{)} there exists a unique algebra map
\[
(U_q^{\otimes k})^{\,\widetilde{}\,}\rightarrow
\textup{End}(M_1\otimes\cdots\otimes M_k),\qquad X\mapsto X_{M_1,\ldots,M_k}
\]
satisfying the following two properties:
\begin{enumerate}
\item $X_{M_1,\ldots,M_k}=(\pi_{M_1}\otimes\cdots\otimes\pi_{M_k})(X)$ for any $X\in U_q^{\otimes k}$,
\item for $m_i\in M_i$ \textup{(}$1\leq i\leq k$\textup{)} the map
\[
(U_q^{\otimes k})^{\,\widetilde{}\,}\rightarrow M_1\otimes\cdots\otimes M_k:\quad X\mapsto X_{M_1,\ldots,M_k}(m_1\otimes\cdots\otimes m_k)
\]
is continuous, with the discrete topology on $M_1\otimes\cdots\otimes M_k$.
\end{enumerate}
\end{lemma}
%%%%%%%%%%%%%%%%%%%%%%%%%%%%%%%%%%%%%%%
\begin{proof}
Fix $m_i\in M_i$ ($1\leq i\leq k$). Since $U^+$ acts locally finitely on $M_i$ for each $i$, we have $\mathcal{H}_\ell^{(k)}(m_1\otimes\cdots\otimes m_k)=0$ for $\ell$ sufficiently large. The result now follows immediately. 
\end{proof}
%%%%%%%%%%%%%%%%%%%%%%%%%%%%%%%%%%%%%

We now recall the definition of Lusztig's \cite[Chpt. 4]{Lusztig-1994} quasi $R$-matrix in the present context. 
The mapping defined on generators as
\[
\overline{\Delta}(E_i)=1\otimes E_i+E_i\otimes q^{-d_ih_i},\quad
\overline{\Delta}(F_i)=q^{d_ih_i}\otimes F_i+F_i\otimes 1,\quad \overline{\Delta}(q^h)=q^h\otimes q^h
\]
extends uniquely to an algebra homomorphism $\overline{\Delta}: U_q\rightarrow U_q\otimes U_q$, and defines a new comultiplication of $U_q$\footnote{ When $U_q$ is defined over $\mathbb{C}(q)$ with $q$ an indeterminate, then $\overline{\Delta}=\bigl(\iota\otimes\iota\bigr)\circ\Delta\circ\iota$ with $\iota: U_q\rightarrow U_q$ Lusztig's involution.}.
Both comultiplications $\Delta$ and $\overline{\Delta}$ extend continuously to algebra homomorphisms   $(U_q)^{\widetilde{}}\rightarrow (U_q\otimes U_q)^{\widetilde{}}$.
Lusztig  \cite[Thm. 4.1.2]{Lusztig-1994} defines the {\it quasi $R$-matrix} $\overline{\mathcal{R}}$ as the unique element
\begin{equation}
\label{quasi R def}
\overline{\mathcal{R}}:=\sum_{\beta\in Q^+}\overline{\mathcal{R}}_\beta\in (U_q^{\otimes 2})^{\widetilde{}}
\end{equation}
with $\overline{\mathcal{R}}_0=1\otimes 1$ and
$\overline{\mathcal{R}}_\beta\in U^+[\beta]\otimes U^-[-\beta]$\, \textup{(}$\beta\in Q^+$\textup{)}
satisfying 
\begin{equation}
\label{quasi-R Delta prop}
\overline{\mathcal{R}}\Delta(X)=\overline{\Delta}(X)\overline{\mathcal{R}}
\end{equation}
in $(U_q^{\otimes 2})^{\widetilde{}}\,$ for any $X\in U_q$\footnote{It is in fact $\overline{\mathcal{R}}^{-1}$ which corresponds to Lusztig's $\Theta$ in \cite[Thm. 4.1.2]{Lusztig-1994}.}.

The quasi $R$-matrix $\overline{\mathcal{R}}$ is invertible in $(U_q^{\otimes 2})^{\widetilde{}}$\,, with inverse
\begin{equation}
\label{inverseexp}
\overline{\mathcal{R}}^{-1}=\sum_{\beta\in Q^+}(S\otimes\textup{id})(\overline{\mathcal{R}}_\beta)(q^\beta\otimes 1)=
\sum_{\beta\in Q^+}(\textup{id}\otimes S^{-1})(\overline{\mathcal{R}}_\beta)(1\otimes q^{-\beta}),
\end{equation}
see \cite[Cor.\ 4.1.3]{Lusztig-1994} for details.
Note that the $\beta$-terms in the sums \eqref{inverseexp} lie in $U^+[\beta]\otimes U^-[-\beta]$. Furthermore, in $(U_q^{\otimes 3})^{\widetilde{}}$ we have
\begin{equation}
\label{DeltaquasiR}
\begin{split}
(\Delta\otimes\textup{id})(\overline{\mathcal{R}})&=\sum_{\beta\in Q^+}\overline{\mathcal{R}}_\beta^{13}(1\otimes q^\beta\otimes 1)\overline{\mathcal{R}}^{23},\\
(\textup{id}\otimes\Delta)(\overline{\mathcal{R}})&=\sum_{\beta\in Q^+}\overline{\mathcal{R}}_\beta^{13}(1\otimes q^{-\beta}\otimes 1)\overline{\mathcal{R}}^{12}
\end{split}
\end{equation}
by \cite[Prop. 4.2.2]{Lusztig-1994}, where we use the standard tensor-leg notations to indicate the choice of embedding of $(U_q^{\otimes 2})^{\widetilde{}}$\, into $(U_q^{\otimes 3})^{\widetilde{}}$.

We now follow \cite[\S 3]{Balagovic-Kolb-2019} for relating Lusztig's quasi $R$-matrix to Drinfeld's $R$-matrix in the present context. It requires the following extension 
of Lemma \ref{ExtensionLemma1}.

Consider the functor
$\textup{For}^{(k)}: \cM^{\times k}\rightarrow\textup{Vec}$, with $\textup{Vec}$ the category of complex vector spaces, mapping $(M_1,\ldots,M_k) \mapsto M_1\otimes\cdots\otimes M_k$ for $M_i\in\cM$ and $(\psi_1,\ldots,\psi_k) \mapsto \psi_1\otimes\cdots\otimes\psi_k$ for morphisms $\psi_i\in\textup{Hom}_{U_q}(M_i,N_i)$. 
Consider the algebra
\[
\mathcal{U}^{(k)}:=\textup{End}(\textup{For}^{(k)})
\]
of natural endomorphisms of $\textup{For}^{(k)}$. It consists of sequences $\{\varphi_{M_1,\ldots,M_k}\}_{M_i\in\cM}$ of complex linear endomorphisms
$\varphi_{M_1,\ldots,M_k}$ 
of $M_1\otimes\cdots\otimes M_k$ satisfying 
\[
(\psi_1\otimes\cdots\otimes\psi_k)\circ\varphi_{M_1,\ldots,M_k}=\varphi_{N_1,\ldots,N_k}\circ (\psi_1\otimes\cdots\otimes\psi_k)
\]
for all $\psi_i\in\textup{Hom}_{U_q}(M_i,N_i)$. 

Lemma \ref{ExtensionLemma1} provides an algebra map
\begin{equation}\label{am}
(U_q^{\otimes k})^{\,\widetilde{}\,}\rightarrow \mathcal{U}^{(k)},\qquad X\mapsto \{X_{M_1,\ldots,M_k}\}_{M_i\in\cM}.
\end{equation}
Its restriction to $U_q^{\otimes k}$ is injective. For
$X\in (U_q^{\otimes k})^{\,\widetilde{}\,}$ we will sometimes denote the corresponding element $\{X_{M_1,\ldots,M_k}\}_{M_i\in\cM}\in\mathcal{U}^{(k)}$ simply by $X$ again, if no confusion can arise.

By \eqref{am}, the quasi $R$-matrix $\overline{\mathcal{R}}\in (U_q^{\otimes 2})^{\,\widetilde{}\,}$ thus provides the element
\[
\{\overline{\mathcal{R}}_{M,N}\}_{M,N\in\cM}\in\mathcal{U}^{(2)}
\]
consisting of complex linear automorphisms $\overline{\mathcal{R}}_{M,N}$ of $M\otimes N$.

%%%%%%%%%%%%%%%%%%%%%
\begin{definition}\label{univRdef}
Define $\kappa\in\mathcal{U}^{(2)}$ by
\[
\kappa_{M,N}\vert_{M[\mu]\otimes N[\nu]}:=q^{\langle\mu,\nu\rangle}\textup{id}_{M[\mu]\otimes N[\nu]}
\]
for $M,N\in\cM$ and $\mu,\nu\in\mathfrak{h}^*$, and $\mathcal{R}\in\mathcal{U}^{(2)}$ by
\begin{equation}\label{UniversalRmatrix}
\mathcal{R}:=\kappa\overline{\mathcal{R}}.
\end{equation}
\end{definition}
%%%%%%%%%%%%%%%%%%%%%%%
Note that $\kappa_{M,N}$ represents the action of 
\begin{equation}\label{asformallycompletedcartan}
q^{\sum_{i=1}^rx_i\otimes x_i}=\prod_{i=1}^r\Bigl(\sum_{n=0}^{\infty}\frac{\tau^n}{n!}(x_i^n\otimes x_i^n)\Bigr)
\end{equation}
on $M\otimes N$ (recall that $q=e^\tau$). Observe furthermore that
\begin{equation}\label{Deltakappa}
\Delta^{\textup{op}}(X)=\kappa\overline{\Delta}(X)\kappa^{-1},\qquad X\in U_q
\end{equation}
as identities in $\mathcal{U}^{(2)}$.
The element $\mathcal{R}$ (see \eqref{UniversalRmatrix}) is Drinfeld's \cite{Drinfeld-1986} universal R-matrix of $U_q$, considered as element in the completion $\mathcal{U}^{(2)}$ of $U_q^{\otimes 2}$. It is clearly a unit in $\mathcal{U}^{(2)}$.

Define an algebra map $\Delta: \mathcal{U}^{(1)}\rightarrow\mathcal{U}^{(2)}$ by 
\[
(\Delta(\varphi))_{M_1,M_2}:=\varphi_{M_1\otimes M_2}\qquad (\varphi\in\mathcal{U}^{(1)},\, M_i\in\cM).
\]
It is consistent with the extended comultiplication $\Delta: (U_q)^{\widetilde{}}\,\rightarrow (U_q\otimes U_q)^{\widetilde{}}$\, as defined before.
The properties (\ref{quasi-R Delta prop}) and (\ref{DeltaquasiR}) of the quasi R-matrix and \eqref{Deltakappa} now immediately translate to the following familiar properties of 
the universal $R$-matrix $\mathcal{R}$.
%%%%%%%%%%%%%%%%%%%%%%%%%%%
\begin{proposition}
	\label{universalRprop}
	\hfill
	\begin{enumerate}
		\item\label{universalRprop 1} We have $\Delta^{\textup{op}}(X)=\mathcal{R}\Delta(X)\mathcal{R}^{-1}$ in $\mathcal{U}^{(2)}$ for all $X\in U_q$.
		\item\label{universalRprop 2} In $\mathcal{U}^{(3)}$ we have
		\begin{equation*}
		(\Delta\otimes\textup{id})(\mathcal{R})=\mathcal{R}^{13}\mathcal{R}^{23},\qquad\quad
		(\textup{id}\otimes\Delta)(\mathcal{R})=\mathcal{R}^{13}\mathcal{R}^{12}.
		\end{equation*}
	\end{enumerate}
\end{proposition}
%%%%%%%%%%%%%%%%%%%%%%%%%%%%
For $M,N\in\cM$ consider the linear map 
\[
c_{M,N}:=P_{M,N}\mathcal{R}_{M,N}: M\otimes N\rightarrow N\otimes M,
\]
with $P_{M,N}: M\otimes N\rightarrow N\otimes M$ the linear map defined by $P_{M,N}(m\otimes n):=n\otimes m$ for $m\in M$ and $n\in N$.
By Proposition \ref{universalRprop}(\ref{universalRprop 1}), the map $c_{M,N}$ is a $U_q$-linear isomorphism. Definition \ref{univRdef} and Proposition \ref{universalRprop}(\ref{universalRprop 2}) then 
directly imply the following result.

%%%%%%%%%%%%%%%%%%%%%%%%%%%%
\begin{corollary}\label{CorBraided}
	The monoidal category \((\cM,\otimes,\mathbb{1},a,\ell,r)\) is a braided monoidal category with commutativity constraint \(c=(c_{M,N})_{M,N\in\cM}\). In particular, the monoidal subcategories $\cM_{\textup{adm}}$ and \(\Rep\) are braided. 
	\end{corollary}
%%%%%%%%%%%%%%%%%%%%%%%%%%%

The restriction of the commutativity constraint $c$ of $\cM$ to the subcategory $\cO$ turns $\cO$ into a braided bimodule category over $\Rep$ 
\textup{(}see \cite{Brochier-2013} for the notion of braided \textup{(}bi\textup{)}module categories\textup{)}.

%%%%%%%%%%%%%%%

Let us finally consider the familiar expression of the inverse of $\mathcal{R}$ involving the antipode $S$ in the present context.
Some care is needed, since the antipode $S: U_q\rightarrow U_q$ is not continuous with respect to the metric topology on $U_q$ that we introduced at the beginning of this subsection. But $S\otimes\textup{id}$ and $\textup{id}\otimes S^{-1}$ are continuous as maps 
\[
U^+\otimes U^-\rightarrow (U_q^{\otimes 2})^{\,\widetilde{}\,},
\]
where $U^+\otimes U^-$ is provided with the subspace topology of $U_q^{\otimes 2}$. We extend both maps to continuous maps 
$\overline{U^+\otimes U^-}\rightarrow (U_q^{\otimes 2})^{\,\widetilde{}\,}$, where
$\overline{U^+\otimes U^-}$ is the closure of $U^+\otimes U^-$ in $(U_q^{\otimes 2})^{\widetilde{}}$. In particular, we may act by $S\otimes\textup{id}$ and $\textup{id}\otimes S^{-1}$ on $\overline{\mathcal{R}}\in\overline{U_+\otimes U_-}$.

The explicit expression \eqref{inverseexp} for $\overline{\mathcal{R}}^{-1}$ then gives the identities
\[
(S\otimes \textup{id})(\mathcal{R})=\mathcal{R}^{-1}=(\textup{id}\otimes S^{-1})(\mathcal{R})
\]
in $\mathcal{U}^{(2)}$, where the action of $S\otimes \textup{id}$ and $\textup{id}\otimes S^{-1}$ on $\overline{U^+\otimes U^-}$ is formally extended to elements in $\kappa\,\overline{U^+\otimes U^-}$ using \eqref{asformallycompletedcartan}.
For example, $(S\otimes\textup{id})(\mathcal{R})\in\mathcal{U}^{(2)}$ stands for 
\[
((S\otimes \textup{id})(\mathcal{R}))_{M,N}(m\otimes n):=\sum_{\beta\in Q^+}q^{-\langle \mu,\nu-\beta\rangle}(S\otimes\textup{id})(\overline{\mathcal{R}}_\beta)(m\otimes n)\qquad
(m\in M[\mu],\, n\in N[\nu])
\]
for $M,N\in\cM$.

%%%%%%%%%%%%%%%%%%%%%%%%%%
\subsection{The ribbon element and twisting in category $\cM$}
\label{Section ribbon element}
%%%%%%%%%%%%%%%%%%%%%%%%%%%
By a classical result of Drinfeld \cite{Drinfeld-1990}, the braided monoidal category $\Rep$ is a ribbon category. So $\Rep$ admits left and right duality, which we recall in the next subsection, and a compatible twist, which
is defined in terms of the action of Drinfeld's \cite{Drinfeld-1990} ribbon element $\vartheta$ of $U_q$. The aim of this subsection is to provide a natural extension of the twist to $\cM$,
turning $\cM$ into a braided monoidal category with twist (see Definition \ref{scdef}).

Let \(m^{\textup{op}}: U_q\otimes U_q\rightarrow U_q\) be the opposite multiplication map defined by \(m^{\mr{op}}(a\otimes b) = ba\) for \(a, b\in U_q\), and extend it continuously to a map $m^{\textup{op}}: (U_q\otimes U_q)^{\widetilde{}}\,\rightarrow (U_q)^{\widetilde{}}$\,. Define the quasi ribbon element $\overline{\vartheta}\in (U_q)^{\widetilde{}}$\, by 
\begin{equation}\label{bartheta}
\overline{\vartheta}:=m^{\textup{op}}\bigl((\textup{id}\otimes S^{-1})(\overline{\mathcal{R}}^{-1})\bigr)=\sum_{\beta\in Q^+}\overline{\vartheta}_\beta
\end{equation}
where, by \eqref{inverseexp},
\begin{equation}\label{bartheta2}
\overline{\vartheta}_\beta:=q^\beta m^{\textup{op}}((S\otimes S^{-1})(\overline{\mathcal{R}}_\beta))\in U^-U^0(U^+[\beta]).
\end{equation}

Define $\gamma\in\mathcal{U}^{(1)}$ by 
\[
\gamma_M\vert_{M[\mu]}:=q^{\langle\mu,\mu+2\rho\rangle}\textup{id}_{M[\mu]}
\]
for $M\in\cM$ and $\mu\in\mathfrak{h}^*$. Formally, $\gamma_M$ is the action of $q^{2\rho+\sum_{i=1}^rx_i^2}$ on $M$.
Drinfeld's \cite{Drinfeld-1990} ribbon element $\vartheta$ in the present context is now defined by
\[
\vartheta:=\gamma\overline{\vartheta}=
q^{2\rho}m^{\textup{op}}\bigl((\textup{id}\otimes S^{-1})(\mathcal{R}^{-1})\bigr)\in\mathcal{U}^{(1)}.
\]
In the following proposition we give some basic properties of $\vartheta$. 
%%%%%%%%%%%%%%%%%%%%%%%%%%%%%%%%%%%%
\begin{proposition}
	\label{Drinfeldtheta}
	\hfill
	\begin{enumerate}
		\item\label{Drinfeldtheta_prop1} $\vartheta\in\mathcal{U}^{(1)}$ is invertible. Its inverse is $\vartheta^{-1}=q^{-2\rho}u$, with $u\in\mathcal{U}^{(1)}$ the Drinfeld element
		\[
		u:=m^{\textup{op}}((\textup{id}\otimes S)(\mathcal{R})).
		\]
		\item\label{Drinfeldtheta_prop2} For any $X\in (U_q)^{\widetilde{}}\,$ one has $\vartheta X=X\vartheta$, viewed as identity in the algebra $\mathcal{U}^{(1)}$.
		\item\label{Drinfeldtheta_prop3} One has the identity $\Delta(\vartheta)=(\vartheta\otimes\vartheta)\mathcal{R}^{21}\mathcal{R}$ in $\mathcal{U}^{(2)}$.
	\end{enumerate}
\end{proposition}
%%%%%%%%%%%%%%%%%%%%%%%%%%%%%%%%%%%
\begin{proof}
The algebraic computations leading to these properties, which can be performed in any quasi-triangular Hopf algebra $H$, are due to Drinfeld \cite{Drinfeld-1990}. The steps of the algebraic derivation can also be found in \cite[Prop. VIII.4.1\, \&\, Prop. VIII.4.5]{Kassel-1995}. The only additional check needed here is that at each step in this derivation the infinite sum manipulations are allowed 
in the present choice of completion
serving the representation category $\cM$. 
This is a direct check, which we leave to the reader.
\end{proof}
%%%%%%%%%%%%%%%%%%%%%%%%%%%%%%%%%%
As an alternative proof of part (\ref{Drinfeldtheta_prop1}) of Proposition \ref{Drinfeldtheta} one can use \cite[Prop. 6.1.7]{Lusztig-1994}, which shows that the
quasi ribbon element $\overline{\vartheta}$ is invertible in $(U_q)^{\widetilde{}}$\, with inverse
\[
\overline{\vartheta}^{-1}=\sum_{\beta\in Q^+}q^{-\langle\beta,\beta\rangle}q^{-2\beta}m^{\textup{op}}\bigl((\textup{id}\otimes S)(\overline{\mathcal{R}}_\beta)\bigr).
\]

By parts (\ref{Drinfeldtheta_prop1}) and (\ref{Drinfeldtheta_prop2}) of Proposition \ref{Drinfeldtheta}, $\vartheta\in\mathcal{U}^{(1)}$ defines an isomorphism
$\vartheta_M\in\textup{End}_{U_q}(M)$ for each $M\in\cM$, called the {\it quantum Casimir operator} on $M$ (cf. \cite[\S 6.1]{Lusztig-1994}).
By Proposition \ref{Drinfeldtheta}(\ref{Drinfeldtheta_prop3}) we conclude:
%%%%%%%%%%%%%%%%%%%%%%%%%%%%%%%%%
\begin{corollary}\label{extendedtwist}
	\label{cortheta}
	$(\vartheta_M)_{M\in\cM}$ is a twist for the braided monoidal category $\cM$.
\end{corollary}
%%%%%%%%%%%%%%%%%%%%%%%%%%%%%%%%%
We also recover the following familiar result.
%%%%%%%%%%%%%%%%%%%%%%%%%%%%%%%%%%%%%%%
\begin{proposition} \label{constanttheta}
Suppose that $M\in\cO$ is a highest weight module of highest weight $\lambda\in\mathfrak{h}^*$. Then
\begin{equation}
\label{c_mu def}
\vartheta_M=q^{\langle \lambda,\lambda+2\rho\rangle}\textup{id}_M.
\end{equation}
\end{proposition}
%%%%%%%%%%%%%%%%%%%%%%%%%%%%%%%%%%%%%%%
\begin{proof}
This follows from Corollary \ref{cortheta}, \eqref{bartheta}, \eqref{bartheta2} and the fact that
$\overline{\vartheta}_0=1$.
\end{proof}
%%%%%%%%%%%%%%%%%%%%%%%%%%%%%%%%%%%%%%%

%%%%%%%%%%%%%%%%%%%%%%%%%%%
\subsection{Dual representations}\label{dualsection}
%%%%%%%%%%%%%%%%%%%%%%%%%%%
In this subsection we introduce the left duality of $\Rep$, turning $\Rep$ into a ribbon category with twist $(\vartheta_V)_{V\in\Rep}$. We first define dual representations for $U_q$-modules $M$ from the representation category $\cM_{\textup{adm}}$.

The full linear dual $M^*$ of \(M\in\cM_{\textup{adm}}\) is a $U_q$-module with $U_q$-action defined by
\begin{equation}
\label{dual representation}
(\pi_{M^*}(X)f)(m) := f\left(\pi_M(S(X))m\right)
\end{equation}
for \(f\in M^*\), \(X\in U_q\) and \(m\in M\). For $\mu\in\mathfrak{h}^\ast$ and $f\in M^*$ we have $f\in M^*[-\mu]$ if and only if $f\vert_{M[\nu]}=0$ for all $\nu\not=\mu$. Hence
$M^*[-\mu]\simeq M[\mu]^*$, and the weight spaces of $M^*$ are finite-dimensional. 

The {\it restricted dual} \(M^\circ\in\textup{Mod}_{U_q}\) of \(M\in\mathcal{M}_{\textup{adm}}\) is the $U_q$-submodule of $M^*$ defined by
\[
M^\circ:=\bigoplus_{\mu\in\hh^\ast}M^*[\mu].
\]
The $U_q$-module $M^\circ$ is an $\mathfrak{h}$-semisimple $U_q$-module with finite-dimensional weight spaces. In the remainder of the paper we will identify the weight spaces $M^\circ[\mu]=M^*[\mu]$ with $M[-\mu]^*$. Note that in general $M^\circ\not\in\cM_{\textup{adm}}$, since the required condition on the weights (see Definition \ref{subcat}(\ref{weight condition})) does not have to hold true for $M^\circ$. For instance, $M_\lambda^\circ\not\in\cM_{\textup{adm}}$ for all $\lambda\in\mathfrak{h}^*$. 

Given a homogeneous basis \(\mathcal{B}_M\) of \(M\in\cM_{\textup{adm}}\) and a basis element \(b\in\mathcal{B}_M\), define the dual vector \(b^{\ast}\in M^\circ\)   by requiring that
\begin{equation}
\label{dual vectors def}
b^{\ast}(b^\prime) = \delta_{b,b^\prime}
\end{equation}
for all $b^\prime\in \mathcal{B}_M$. Clearly \(\{b^\ast\}_{b\in\mathcal{B}_M}\) 
is a homogeneous basis for \(M^\circ\). It is called the homogeneous basis of \(M^\circ\) dual to \(\mathcal{B}_M\). 
If \(M\in\cO\) is a highest weight module of highest weight \(\lambda\) with highest weight vector \(m\in M[\lambda]\), then we will write \(m^\ast\) for the unique linear functional in \(M^\circ[-\lambda]\) such that \(m^\ast(m)=1\). 

We now consider the above-mentioned notion of dual representation for finite-dimensional modules in $\cM$, 
which form the subcategory $\Rep$.
If $V$ is finite-dimensional, then $V^\circ$ is the full linear dual $V^*$ of $V$, and the dual representation $V^*$ lies in $\Rep$ again.
In this case the assignment \(V\rightarrow V^\ast\) gives rise to the left dual of \(V\), 
with the evaluation and injection morphisms 
the $U_q$-linear maps defined by 
\begin{align}
\label{evaluation morphism}
e_V:&\ V^\ast\otimes V\to \mathbb{1}: f\otimes v \mapsto f(v),\\
\label{injection morphism}
\iota_V:&\ \mathbb{1}\to V\otimes V^\ast: 1\mapsto \sum_{b\in \mathcal{B}_V} b\otimes b^\ast.
\end{align}
Note that the dual $A^*\in\Hom_{U_q}(V^\ast,U^\ast)$ of a morphism $A\in\Hom_{U_q}(U,V)$ is simply its dual as a $\CC$-linear map.

We have the following well-known theorem of Drinfeld \cite{Drinfeld-1990}.
%%%%%%%%%%%%%%%%%%%%%%%%%
\begin{theorem}\label{Drinfeldribbontheorem}
	The braided monoidal category \(\Rep=(\Rep,\otimes,\mathbb{1},a,\ell,r,c)\), endowed with the left duality \((V^\ast,e_V,\iota_V)\), becomes a ribbon category with 
	twist \((\vartheta_V)_{V\in\Rep}\). 
\end{theorem}
%%%%%%%%%%%%%%%%%%%%%%%%

A straightforward computation shows that the evaluation morphism \(\widetilde{e}_V\) and co-evaluation morphism \(\widetilde{\iota}_V\) of the right duality of the ribbon category $\Rep$ are explicitly given by
\begin{align}\label{rhoshift}
\widetilde{e}_V(v\otimes f)=f\left(q^{2\rho}v\right),\qquad
\widetilde{\iota}_V(1)=\sum_{b\in\mathcal{B}_V} b^\ast\otimes q^{-2\rho}b
\end{align}
for \(v\in V\) and \(f\in V^\ast\).

For later purposes we recall the definition of the partial quantum trace over $V\in\Rep$. 

%%%%%%%%%%%%%%%%%%%%%%%%%%%%%%%%%%%
\begin{definition}\label{pqt}
For \(V\in\Rep\) and \(M,M'\in\cM\) the partial quantum trace 
\[
\qTr_V^{M,M^\prime}: \textup{Hom}_{U_q}(M\otimes V,M^\prime\otimes V)\rightarrow \textup{Hom}_{U_q}(M,M^\prime)
\]
over $V$ is defined by 
	\[
	\qTr_V^{M,M^\prime}(\Psi):=(\textup{id}_{M'}\otimes\widetilde{e}_V)(\Psi\otimes\textup{id}_{V^*})(\textup{id}_{M}\otimes\iota_V)
	\]
for $\Psi\in\textup{Hom}_{U_q}(M\otimes V,M^\prime\otimes V)$, where we identify $M\otimes\mathbb{1}\simeq M$ and $M^\prime\otimes\mathbb{1}\simeq M^\prime$
using the right unit constraint maps.
\end{definition}
%%%%%%%%%%%%%%%%%%%%%%%%%%%%%%%%%%%
	In terms of a homogeneous basis $\mathcal{B}_V$ of $V$ the partial quantum trace of the intertwiner 
	$\Psi\in\textup{Hom}_{U_q}(M\otimes V,M^\prime\otimes V)$ over $V$ is explicitly given by the formula
	\[
	\qTr_V^{M,M^\prime}(\Psi)=\sum_{v\in \mathcal{B}_V}(\id_{M'}\otimes v^\ast)\Psi(\id_M\otimes \pi_V(q^{2\rho})v).
	\]

%%%%%%%%%%%%%%%%%%%%%%%%%%%%%%%%%%%%%%%
\section{Graphical calculus}\label{GcSection}
%%%%%%%%%%%%%%%%%%%%%%%%%%%%%%%%%%%%%%%

In this section we will outline the notations and basic principles of the graphical calculus for strict braided monoidal categories with twist and strict ribbon categories. 
The formulations and terminology in this section follow closely \cite{Reshetikhin&Turaev-1990} and \cite[Chpt. I]{Turaev-1994}. 

%%%%%%%%%%%%%%%%%%%%%%%%%%%%%%%%%
\subsection{Ribbon graphs and ribbon graph diagrams}
\label{tangles}
%%%%%%%%%%%%%%%%%%%%%%%%%%%%%%%%%

A central concept in graphical calculus is the notion of a {\it ribbon graph}. A ribbon graph consists of bands, annuli and coupons, which are defined as follows.

Let $I$ be the unit interval $[0,1]$ in $\mathbb{R}$ and $\mathbb{S}^1$ the unit circle in $\mathbb{R}^2$. We take the right-handed coordinate system in $\mathbb{R}^3$ with the $x$-axis drawn horizontally from left to right, and the $z$-axis vertically from bottom to top. The $y$-axis thus points away from the reader.

A {\it band} is a homeomorphic image of $I\times I$ in $\mathbb{R}^2\times I$. The
homeomorphic images of $I\times \{0\}$ and $I\times \{1\}$ are called the bases of the band, and the homeomorphic image of $\{\frac{1}{2}\}\times I$ the core of the band. 
An {\it annulus} is the homeomorphic image of $\mathbb{S}^1\times I$ in $\mathbb{R}^2\times I$. The homeomorphic image of $\mathbb{S}^1\times\{\frac{1}{2}\}$ is called its core. 
A {\it coupon} is a band with one base designated as the top base, and the opposite base as the bottom base.

For an oriented surface $\Omega$ in $\mathbb{R}^3$, we call the side of the surface with the surface normal sticking out the \emph{white side} of the surface, and the opposite side the \emph{shaded side}.
For $k,\ell\in\mathbb{Z}_+$ a {\it $(k,\ell)$-ribbon graph} $\Omega$ is an oriented surface in $\mathbb{R}^2\times [0,1]$ consisting of a union of a finite number of bands, annuli and coupons, such that
\begin{enumerate}
\item Coupons and annuli lie in $\mathbb{R}^2\times (0,1)$.
\item $\Omega$ meets the planes $\mathbb{R}^2\times\{0\}$ and $\mathbb{R}^2\times\{1\}$ in the segments $[i-\frac{1}{3},i+\frac{1}{3}]\times\{0\}\times\{0\}$ ($1\leq i\leq k$) and $[j-\frac{1}{3},j+\frac{1}{3}]\times\{0\}\times\{1\}$ ($1\leq j\leq \ell$). These segments are bases of bands. The white side is up at these bases.
The bases of bands which are not one of these segments, are lying on bases of coupons. 
\item The only nontrivial intersections between the bands, annuli and coupons are the intersections of the bases of bands with the bases of coupons.
\item The cores of the bands and annuli are oriented.
\end{enumerate}
{A $(k,\ell)$-ribbon graph without coupons is called a {\it $(k,\ell)$-ribbon tangle}.

We call the bases of bands lying on $\mathbb{R}^2\times\{0\}$ and $\mathbb{R}^2\times\{1\}$ the {\it bottom and top extremal bases}, respectively. More generally, we say that a base of a band of a ribbon graph $\Omega$ is a {\it bottom base} of $\Omega$ if it is either a bottom extremal base or if it lies on the top base of a coupon. Otherwise, we call it a {\it top base} (in that case, it is either a top extremal base or it lies on a bottom base of a coupon).  From now on we will call the core of a band in a ribbon graph its {\it strand}. The intersection of the bases of the band with its strand will be called the bases of the strand. 

We will be concerned with isotopy classes of $(k,\ell)$-ribbon graphs. Isotopy refers to ambient isotopy in $\mathbb{R}^2\times I$ fixing $\mathbb{R}^2\times \{0\}$ and $\mathbb{R}^2\times \{1\}$, preserving the orientation of the graph surface, preserving the splitting in bands, annuli and coupons, and preserving the orientations of the cores of the bands and annuli. 

A ribbon graph in standard position (see \cite[\S 2.1]{Turaev-1994}) has its coupons lying in the strip $\mathbb{R}\times\{0\}\times (0,1)$ with the white sides up and with the top base of each coupon lying above its bottom base. Furthermore, the surface normal of a ribbon graph in standard position is required to take values in the half-space $\mathbb{R}\times \mathbb{R}_{<0}\times\mathbb{R}$ (i.e., the white side is tilted towards the reader).
The projections onto $\mathbb{R}\times\{0\}\times I$ along the $y$-axis of the cores of its bands and annuli are only allowed to have transversal double crossings at interior points of the projected cores, and are only allowed to have a finite number of local maxima and local minima.  
The resulting diagrams, enriched with the over-and undercrossing information at the double crossings, are called {\it ribbon graph diagrams}. Note that
an extremal base of a subdiagram of a ribbon graph diagram is a top (respectively bottom) extremal base if and only if
it is a local maximum (respectively local minimum) of its strand.

In the $(k,\ell)$-ribbon graph diagrams, the $k$ bottom extremal bases will lie on $\mathbb{R}\times\{0\}\times\{0\}$ (the ``floor''), while the $\ell$ top extremal bases lie on $\mathbb{R}\times\{0\}\times\{1\}$ (the ``ceiling''). We omit the floor and the ceiling when drawing the ribbon graph diagrams, if no confusion can arise. 
Figure \ref{diagram1} is an example of a $(3,2)$-ribbon graph diagram with one coupon. 
\begin{figure}[H]
	\centering
	\includegraphics[scale = 0.7]{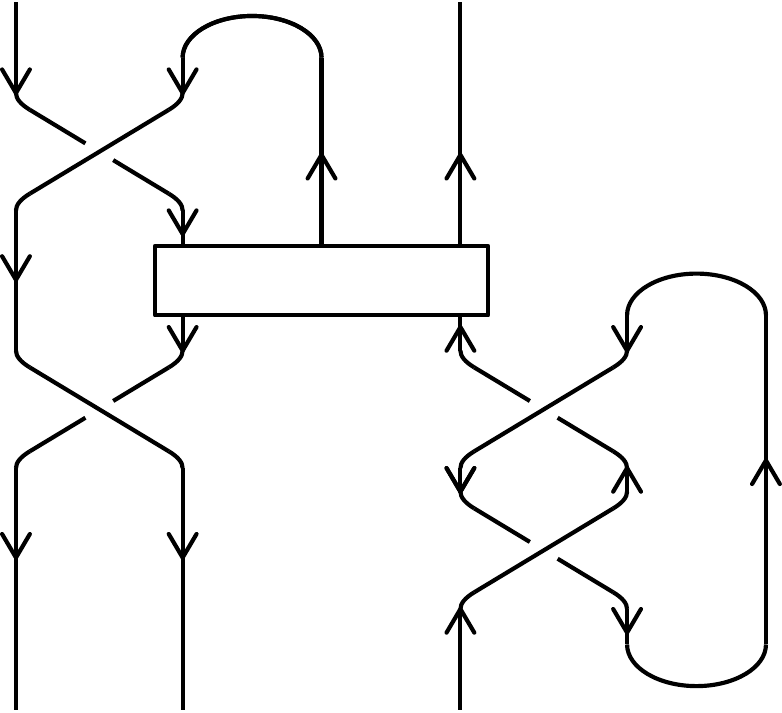}
	\caption{}
	\label{diagram1}
\end{figure}

The projections of the $(k,\ell)$-ribbon tangles are called {\it $(k,\ell)$-ribbon tangle diagrams}. Isotopy of $(k,\ell)$-ribbon tangles can then be described in terms of the associated ribbon graph diagrams by planar isotopies and a number of elementary local moves,
see $\textup{Rel}_1$--$\textup{Rel}_{10}$ in \cite[\S 5]{Reshetikhin&Turaev-1990} for the complete list.
It includes the framed Reidemeister moves depicted in Figures \ref{Yang-Baxter eq diagram}, \ref{universal R-matrix double crossing} and \ref{Redemeister-f}.
\begin{figure}[H]
\begin{minipage}{0.62\textwidth}
	\centering
	\includegraphics[scale = 0.7]{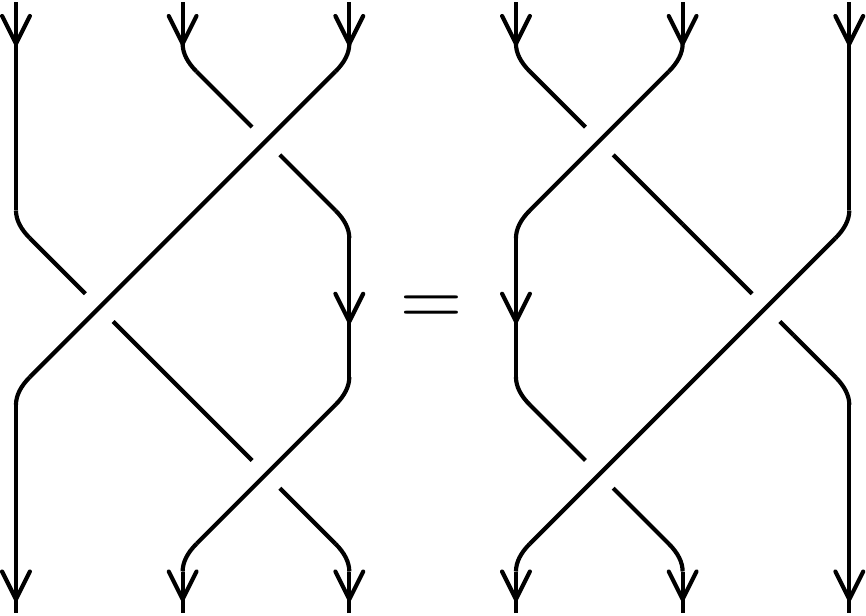}
	\captionof{figure}{ }
	\label{Yang-Baxter eq diagram}
\end{minipage}
\begin{minipage}{0.36\textwidth}
		\centering
		\includegraphics[scale = 0.7]{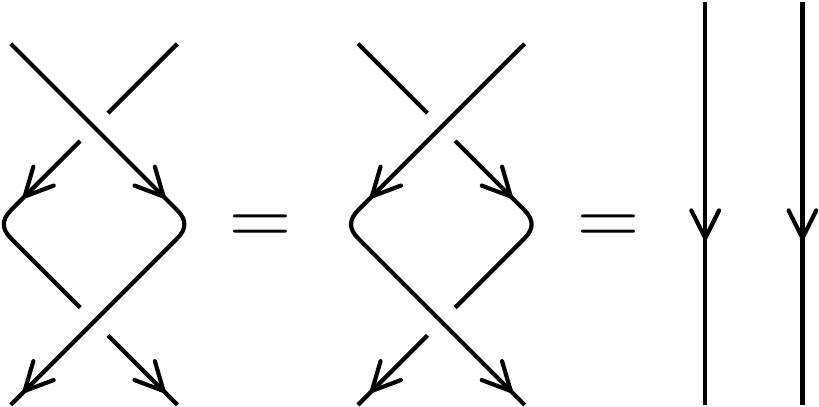}
		\captionof{figure}{ }
		\label{universal R-matrix double crossing}
\end{minipage}
\end{figure}

\begin{figure}[H]
	\begin{minipage}{0.47\textwidth}
		\centering
		\includegraphics[scale = 1]{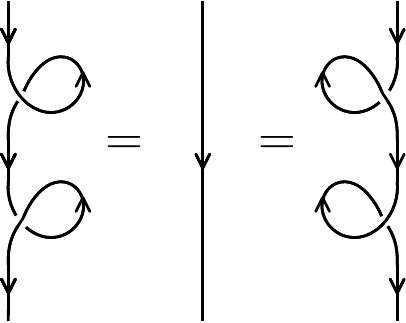}
	\end{minipage}\qquad
	\begin{minipage}{0.47\textwidth}
		\centering
		\includegraphics[scale = 1]{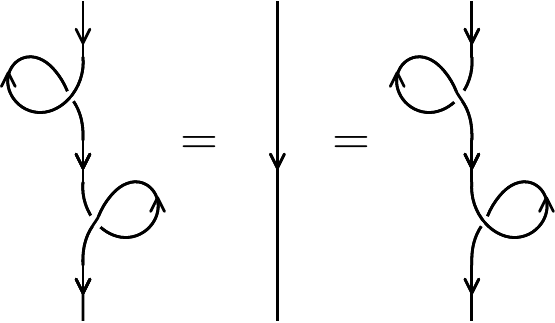}
	\end{minipage}
	\caption{}
	\label{Redemeister-f}
\end{figure}
In order to describe the isotopy of ribbon graphs in terms of the associated ribbon graph diagrams, one needs to add three elementary moves between the coupons and 
the projected cores of the bands and the annuli, see $\textup{Rel}_{11}$--$\textup{Rel}_{13}$ in \cite[\S 5]{Reshetikhin&Turaev-1990}.

%%%%%%%%%%%%%%%%%%%%%%%%%%%%%%%%%%%%%%%%%
\subsection{Ribbon-braid graphs and ribbon-braid graph diagrams}\label{topSection}
%%%%%%%%%%%%%%%%%%%%%%%%%%%%%%%%%%%%%%%%%

A {\it $(k,\ell)$-ribbon-braid graph} is a $(k,\ell)$-ribbon graph $L$ without annuli such that: 
\begin{enumerate}
\item Each band of $L$ has a top base and a bottom base. The orientation of the strands is from the top base towards the bottom base.
\item $L$ is isotopic to a ribbon graph 
whose strands intersect 
the plane $\mathbb{R}^2\times\{z\}$ in at most one point for all $z\in I$.
\end{enumerate}
Note that
$(k,\ell)$-ribbon-braid graphs without coupons only exist when $k=\ell$, in which case we call them {\it $k$-ribbon-braids}. 

We say that a ribbon-braid graph $L$ is {\it in standard position} if $L$ is in standard position as a ribbon graph and if each strand in the associated ribbon graph diagram moves down besides a finite number of the full twists depicted in Figure \ref{Redemeister-twist}.  
\begin{figure}[H]
\includegraphics[scale = 1]{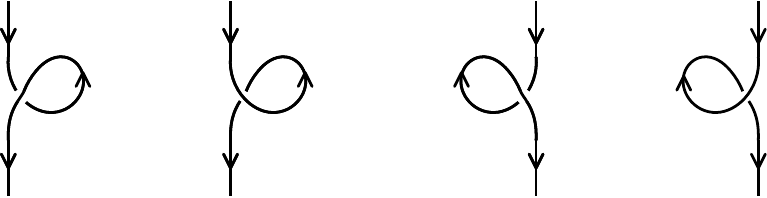}	
	\caption{}
	\label{Redemeister-twist}
\end{figure}
\noindent
The ribbon graph diagram of a ribbon-braid graph in standard position is called a {\it ribbon-braid graph diagram}. Note that in the realm of ribbon-braid graph diagrams, a full twist should be viewed as a single elementary diagram, which does not dissect into the smaller elementary diagrams for ribbon graph diagrams (the crossing, cup and cap).

If the ribbon-braid graph $L$ admits a ribbon-braid graph diagram without full twists, then we call $L$ a {\it braid graph}, and its associated diagram a {\it braid graph diagram}.

For ribbon-braid graphs the orientations of the strands are determined by the underlying unoriented ribbon-braid graph, so they might as well be ignored. However, since we will regularly consider ribbon-braid subgraphs of ribbon graphs, and at some point rotate the diagrams by 90 degrees, we will keep adding the orientation of the strands to the diagrams for the convenience of the reader. 

Isotopies between braid graphs are described in terms of the associated ribbon-braid graph diagrams by the elementary moves $\textup{Rel}_5$--$\textup{Rel}_7$ and $\textup{Rel}_{11}$--$\textup{Rel}_{12}$ from \cite[\S 5]{Reshetikhin&Turaev-1990}. 
The relations $\textup{Rel}_5$-$\textup{Rel}_7$ in \cite[\S 5]{Reshetikhin&Turaev-1990} are the second and third Reidemeister moves (see 
Figures \ref{Yang-Baxter eq diagram} and \ref{universal R-matrix double crossing}), which describe the elementary moves between the strands in the braid graph diagram.
$\textup{Rel}_{11}$--$\textup{Rel}_{12}$ from \cite[\S 5]{Reshetikhin&Turaev-1990}  describe the elementary moves between coupons and strands.

The description of the isotopies between ribbon-braid graphs in terms of the ribbon-braid graph diagrams involve the additional non-elementary move describing how a ribbon-braid graph diagram $L$ can be transported through full twists, which is depicted in Figure \ref{diagram1april}, 
\begin{figure}[H]
\includegraphics[scale = 0.7]{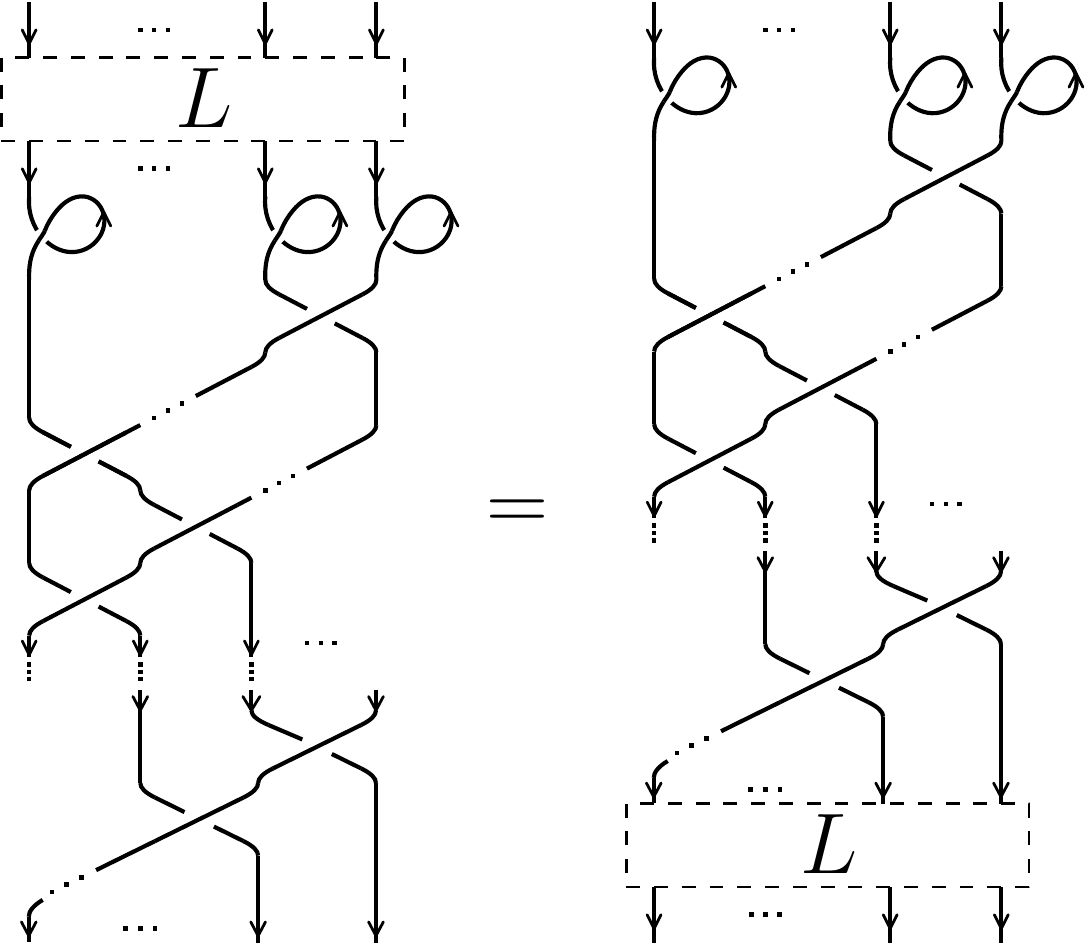}	
\caption{}
\label{diagram1april}
\end{figure}
\noindent
as well as the elementary move that full twists may be pulled over and under a crossing, see Figure \ref{diagram3april} for examples. 
\begin{figure}[H]
	\begin{minipage}{0.47\textwidth}
		\centering
		\includegraphics[scale = 0.8]{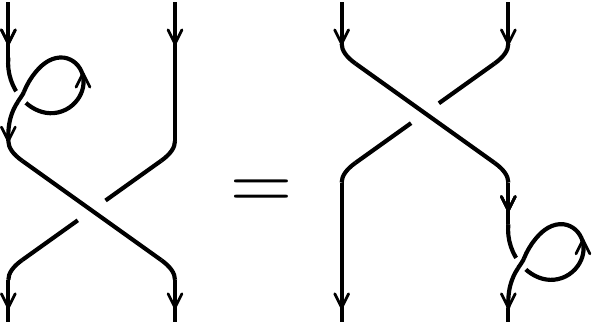}
	\end{minipage}\qquad
	\begin{minipage}{0.47\textwidth}
		\centering
		\includegraphics[scale = 0.8]{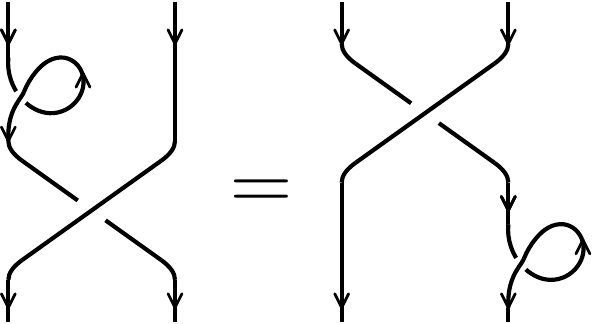}
	\end{minipage}
\caption{}
\label{diagram3april}
\end{figure}

%%%%%%%%%%%%%%%%%%%
\subsection{Strictifications}\label{StSection}
%%%%%%%%%%%%%%%%%%%%
A monoidal category $\cD=(\cD,\otimes,\mathbb{1},a,\ell,r)$ is said to be strict if the associator $a$ and unit constraints $\ell,r$ are trivial. 
$\cD$-coloring ribbon(-braid) graphs requires the strictification of $\cD$, which is a tensor equivalence assigning to $\cD$ a strict monoidal category $\cD^\str$ 
 (cf. \cite{MacLane-1971} and \cite[\S XI.5]{Kassel-1995}). 
 
The objects of $\cD^\str$ are the tuples $(V_1,\ldots,V_k)$ of objects $V_i\in\cD$ for some $k\in\mathbb{Z}_+$. By convention, there is a unique object of length zero, which we denote by 
$\emptyset$. The morphisms of $\cD^\str$ with source $S=(V_1,\ldots,V_k)$ and target $T=(W_1,\ldots,W_\ell)$ are
\begin{equation}\label{morSTR}
\textup{Hom}_{\cD^\str}\bigl(S,T):=\textup{Hom}_{\cD}\bigl(V_1\otimes(V_2\otimes(V_3\otimes\cdots)),
W_1\otimes (W_2\otimes(W_3\otimes\cdots))\bigr),
\end{equation}
with composition and identity inherited from $\cD$. When $k=0$ respectively $\ell=0$, the source respectively target in the right-hand side of \eqref{morSTR} 
should be read as $\mathbb{1}$. 

The category $\cD^\str$ is a strict monoidal category $(\cD^\str,\tens,\emptyset)$ with unit object $\emptyset$ and tensor product $\tens$ defined on objects by
\[
(V_1,\ldots,V_k)\tens (W_1,\ldots,W_\ell):=(V_1,\ldots,V_k,W_1,\ldots,W_\ell).
\]
On morphisms the tensor product $\tens$ is defined by taking the tensor product of the associated morphisms in $\cD$, (pre)composed with an appropriate composition of associators in order to place the brackets in the right order. 

The functor $\cF^\str_{\cD}: \cD^\str\rightarrow\cD$, defined on objects \(S = (V_1,\dots, V_k)\) by 
\[
\cF^\str_{\cD}(S):=V_1\otimes(V_2\otimes(V_3\otimes\cdots)),\qquad
\cF^\str_{\cD}(\emptyset):=\mathbb{1}
\]
and on morphisms $A\in \Hom_{\cD^\str}(S,T)$ by
\[
\cF^\str_{\cD}(A):=A,
\]
where $A$ is now viewed as morphism in $\Hom_{\cD}\bigl(\cF^\str_{\cD}(S), \cF^\str_{\cD}(T)\bigr)$,
defines an equivalence of monoidal categories.\footnote{Note, however, that unless $\cD$ is a strict monoidal category, the monoidal functor $\cF^\str$ is not strict.} A quasi-inverse of $\cF^\str_{\cD}$ is given by the functor $\cG^\str_{\cD}: \cD\rightarrow\cD^\str$ mapping $V\in\cD$ to the object $(V)\in\cD^\str$ of length $1$ and mapping \(A\in\Hom_{\cD}(V,W)\)
to \(A\), considered as element of \(\Hom_{\cD^\str}((V),(W))\). From now on we omit the sublabels $\cD$ in $\cF^\str_\cD$ and $\cG^\str_\cD$ if no confusion can arise.

We have $\cF^\str\circ\cG^\str=\textup{id}_{\cD}$ and 
\[
J: \id_{\cD^\str} \overset{\sim}{\longrightarrow} \cG^\str\circ\cF^\str,
\]
with $J=(J_S)_{S\in\cD^\str}$ consisting of functorial isomorphisms
\begin{equation}\label{jS}
J_S\in\textup{Hom}_{\cD^\str}(S,(\mathcal{F}^\str(S)))
\end{equation}
representing $\textup{id}_{\mathcal{F}^\str(S)}$.
Note that for \(S,T\in\cD^\str\) and \(A\in\textup{Hom}_{\cD^\str}(S,T)\) we have
\begin{equation}
\label{meaning of natural}
J_T\circ A\circ J_S^{-1}=\bigl(\mathcal{G}^\str\circ\cF^\str\bigr)(A)
\end{equation}
as morphisms in \(\textup{Hom}_{\cD^\str}\bigl((\cF^\str(S)), (\cF^\str(T))\bigr)\). We call the $J_S$ the \emph{fusion morphisms} of $\cD$.

%%%%%%%%%%%%%%%%%%%%%%%%%%%%%%%%%%%%%%%%
\begin{remark}
The fusion morphisms \(J_S\)
satisfy the \(2\)-cocycle property
\begin{equation}
\label{basic j}
J_{(\cF^\str(S))\tens T}\circ(J_S\tens \textup{id}_T)=J_{S\tens T}=J_{S\tens (\cF^\str(T))}\circ (\textup{id}_S\tens J_T)
\end{equation}
for \(S,T\in\cD^\str\). Here we have suppressed the $\cG^\str$-images of associators which map the left-hand side of the formula to \(\textup{Hom}_{\cD^\str}\bigl(S\tens T,(\cF^\str(S\tens T))\bigr)\). 
\end{remark}
%%%%%%%%%%%%%%%%%%%%%%%%%%%%%%%%%%%%%%%%
If the order in which to consider $k$-fold tensor products of objects $V_1,\ldots,V_k$ in $\cD$ is clear from the context, then we will leave out the brackets and denote the resulting object by the standard multi-tensor product notation $V_1\otimes\cdots\otimes V_k$. For our purposes, the preferred order will always be from right to left, i.e.\ as prescribed by the definition of the functor \(\cF^{\str}\). In case $\cD$ is a module category, we use the same convention for pure tensors in $\cF^\str((V_1,\ldots,V_k))$.

A left duality for the monoidal category $\cD$ extends to a left duality for $\cD^\str$.
The left dual $(S^*,e_S^\str,\iota_S^\str)$ of an object $S\in\cD^\str$ is defined by
\[
\emptyset^*:=\emptyset,\qquad\quad (V_1,\ldots,V_k)^*:=(V_k^*,\ldots,V_1^*),
\]
with evaluation morphism $e_{(V_1,\ldots,V_k)}^\str\in\textup{Hom}_{\cD^\str}((V_k^*,\ldots,V_1^*,V_1,\ldots,V_k),\emptyset)$ represented by the morphism
\begin{equation}\label{ecomp}
e_{V_k}(\textup{id}_{V_k^*}\otimes e_{V_{k-1}}\otimes\textup{id}_{V_k})\cdots (\textup{id}_{V_k^*\otimes\cdots\otimes V_2^*}\otimes e_{V_1}\otimes\textup{id}_{V_2\otimes\cdots\otimes V_k})
\end{equation}
in $\textup{Hom}_{\cD}(V_k^*\otimes\cdots\otimes V_1^*\otimes V_1\otimes\cdots\otimes V_k,\mathbb{1})$ and the injection morphism 
\[
\iota_{(V_1,\ldots,V_k)}^\str\in\textup{Hom}_{\cD^\str}(\emptyset, (V_1,\ldots,V_k,V_k^*,\ldots,V_1^*))
\]
represented by the morphism
\begin{equation}\label{icomp}
(\textup{id}_{V_1\otimes\cdots\otimes V_{k-1}}\otimes\iota_{V_k}\otimes\textup{id}_{V_{k-1}^*\otimes\cdots\otimes V_1^*})\cdots
(\textup{id}_{V_1}\otimes\iota_{V_2}\otimes\textup{id}_{V_1^*})\iota_{V_1}
\end{equation}
in $\textup{Hom}_{\cD}(\mathbb{1},V_1\otimes\cdots\otimes V_k\otimes V_k^*\otimes\cdots\otimes V_1^*)$. For the object of length zero we set 
$e_\emptyset^\str:=\textup{id}_\emptyset$ and 
$\iota_\emptyset^\str:=\textup{id}_\emptyset$.
A similar remark applies to right duality.

A commutativity constraint $c=(c_{V,W})_{V,W\in\cD}$ for the 
monoidal category $\cD$ also extends to $\cD^\str$. The associated commutativity constraint $c^\str=(c_{S,T}^\str)_{S,T\in\cD^\str}$ of $\cD^\str$ consists of 
the morphisms $c_{S,T}^{\str}\in\textup{Hom}_{\cD^{\str}}(S\tens T,T\tens S)$ representing
\[
c_{\cF^\str(S),\cF^\str(T)}\in\textup{Hom}_{\cD}(\cF^\str(S)\otimes\cF^\str(T),\cF^\str(T)\otimes\cF^\str(S))
\]
up to an appropriate (pre)composition with associators to place the brackets in the right order.
If the braided monoidal category $\cD$ has a twist $\theta=(\theta_V)_{V\in\cD}$, then the strict braided monoidal category $\cD^\str$ has a unique twist $\theta^\str=(\theta^\str_S)_{S\in\cD^\str}$ such that the automorphism
$\theta_{S}^\str\in\textup{End}_{\cD^\str}(S)$ is represented by $\theta_{\cF^\str(S)}\in\textup{End}_\cD(\cF^\str(S))$ for all $S\in\cD^\str$. 

In what follows, we will denote an element \((V)\) of length 1 in $\cD^\str$ simply by \(V\) if no confusion can arise.

%%%%%%%%%%%%%%%%%%%%%%%%%%%%%%%%%%%%%%%%%%%%%
\subsection{Graphical calculus for $\cD$-colored ribbon-braid graphs}\label{GcBraid}
%%%%%%%%%%%%%%%%%%%%%%%%%%%%%%%%%%%%%%%%%%%%%%

Assume that $\cD$ is a mo\-noi\-dal category. A {\it $\cD$-colored $(k,\ell)$-ribbon-braid graph} is a $(k,\ell)$-ribbon-braid graph with its bands colored by objects from $\cD$ and its coupons colored by
appropriate morphisms from $\cD^\str$. 
For a given coupon, we require the following compatibility between its color and the colors of the bases of the strands on the coupon (we use the convention that the bases of a band inherit the color of its band).
If the coupon has $k$ bases of bands lying on its bottom base, colored by $V_1,\ldots,V_k$ in counterclockwise order, and $\ell$ bases of bands lying on its top base, colored by $W_1,\ldots,W_\ell$ in clockwise order, then an admissible coloring of the coupon is a choice of a morphism $A\in\textup{Hom}_{\cD^\str}\bigl((V_1,\ldots,V_k),(W_1,\ldots,W_\ell)\bigr)$.
For coupons with no bands attached to the bottom base (respectively top base) the labeling is by morphisms from $\textup{Hom}_{\cD^\str}\bigl(\emptyset,(W_1,\ldots,W_\ell)\bigr)$
(respectively $\textup{Hom}_{\cD^\str}\bigl((V_1,\ldots,V_k),\emptyset\bigr)$). The coupons without any bands attached are labeled by morphisms from the commutative ring
$\textup{End}_{\cD^\str}(\emptyset)$. 

$\cD$-colored ribbon-braid graph diagrams are defined in a similar manner. An example is given in Figure \ref{diagram4}. Its coupon is colored by a morphism $A\in\textup{Hom}_{\cD^\str}\bigl((W_1,W_2),(V_3,V_1,V_4,V_6)\bigr)$.
\begin{figure}[H]
	\centering
	\includegraphics[scale = 1]{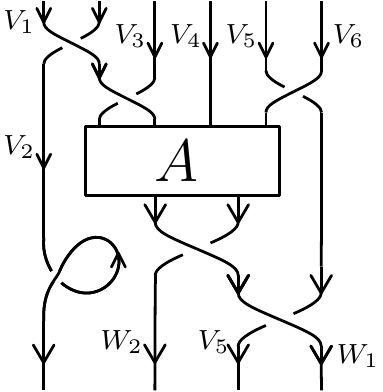}
	\caption{}
	\label{diagram4}
\end{figure}

Isotopies between $\cD$-colored ribbon-braid graphs are ribbon-braid graph isotopies preserving the colors of the bands and coupons. The description of isotopies of ribbon-braid graphs in terms of the associated ribbon-braid graph diagrams (see Subsection \ref{topSection}) extends to $\cD$-colored ribbon-braid graphs in the obvious manner.

We define the category $\cBrr_\cD$ of $\cD$-colored ribbon-braid graphs as follows. The class of objects of $\cBrr_\cD$ is the class of objects of $\cD^\str$.
For objects $S=(V_1,\ldots,V_k)$ and $T=(W_1,\ldots,W_\ell)$ in $\cBrr_{\cD}$, the class of morphisms $\textup{Hom}_{\cBrr_\cD}(S,T)$ consist of the isotopy classes of 
$\cD$-colored $(k,\ell)$-ribbon-braid graphs with the bottom extremal bases colored counterclockwise by $V_1,\ldots,V_k$ and the top extremal bases colored
clockwise by $W_1,\ldots,W_\ell$. The isotopy class of the $\cD$-colored ribbon-braid graph shown in Figure \ref{diagram4} thus represents a morphism $(V_2,W_2,V_5,W_1)\rightarrow (V_1,V_2,V_3,V_4,V_5, V_6)$
in $\cBrr_\cD$.

Composition is defined by vertical stacking of the $\cD$-colored ribbon-braid graphs.
The identity morphism $\textup{id}_{\emptyset}^{\mathbb{B}_{\cD}}$ is the empty graph, while
$\textup{id}_{(V_1,\ldots,V_k)}^{\mathbb{B}_{\cD}}$ for $k\geq 1$ is the $\cD$-colored $(k,k)$-ribbon-braid graph with $k$ parallel vertical bands colored from left to right by $V_1,\ldots,V_k$.

The category $\cBrr_{\cD}$ is a strict monoidal category with unit object $\emptyset$ and tensor product $\widetilde{\tens}_{\mathbb{B}_{\cD}}$ defined on objects by concatenation of the tuples of objects from $\cD^\str$, and on morphisms by placing the ribbon-braid graphs next to each other. We will omit the sublabel and simply write \(\widetilde{\tens}\) if no confusion arise. Note that the tensor products $\widetilde{\tens}$ and $\tens$ are identical on the objects of \(\cBrr_{\cD}\).
 
$\cBrr_{\cD}$ is braided with commutativity constraint $c^{\cBrr_\cD}=(c^{\cBrr_\cD}_{S,T})_{S,T\in\cBrr_\cD}$ given by Figure \ref{diagram5new}.
\begin{figure}[H]
	\centering
	\includegraphics[scale = 0.8]{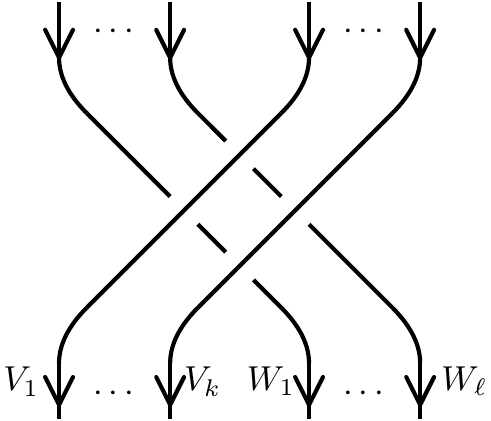}
	\caption{\(c^{\cBrr_{\cD}}_{S,T}\) with \(S = (V_1,\dots,V_k)\) and \(T = (W_1,\dots,W_\ell).\)}
	\label{diagram5new}
\end{figure}
If one of the two objects (or both) is $\emptyset$, then the commutativity constraint is the identity morphism. Note that the functoriality of $c^{\cBrr_\cD}$ is immediate for morphisms in $\cBrr_\cD$ without coupons. For morphisms involving coupons it is a consequence of the two elementary moves involving coupons ($\textup{Rel}_{11}$--$\textup{Rel}_{12}$ in \cite[\S 5]{Reshetikhin&Turaev-1990}). The definition of $c^{\cBrr_\cD}$ automatically guarantees that the hexagon identities are satisfied.

Finally, the strict braided monoidal category
$\cBrr_{\cD}$ has a twist $\theta^{\cBrr_{\cD}}=(\theta^{\cBrr_{\cD}}_S)_{S\in\cBrr_{\cD}}$ with the functorial
isomorphisms $\theta_S^{\cBrr_{\cD}}: S\overset{\sim}{\longrightarrow}S$ defined by $\theta_\emptyset^{\cBrr_{\cD}}=\textup{id}_\emptyset^{\mathbb{B}_{\cD}}$ and 
Figure \ref{diagram2april}. Note the distinction between the notations \(\theta\) for the twist in the category \(\cBrr_{\cD}\) and \(\vartheta\) for the ribbon element in \(\mathcal{U}^{(1)}\) defined in Subsection \ref{Section ribbon element}.
\begin{figure}[H]
	\centering
	\includegraphics[scale = 0.8]{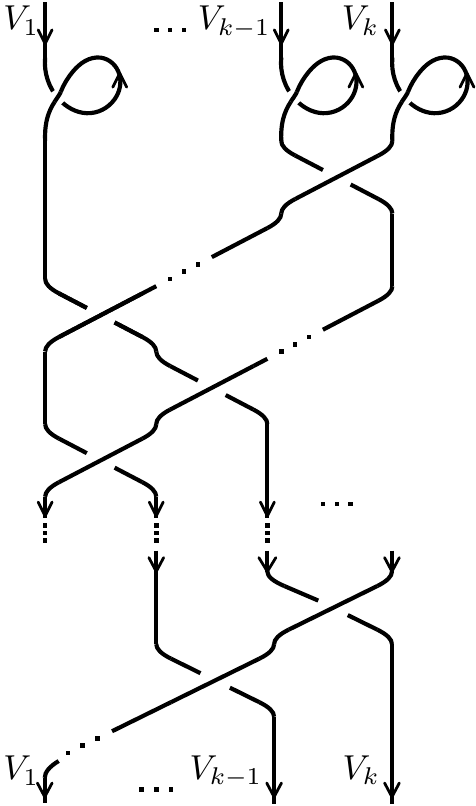}
	\caption{\(\theta_S^{\cBrr_\cD}$ with $S=(V_1,\ldots,V_k)\).}
	\label{diagram2april}
\end{figure}
Functoriality is immediate by Figure \ref{diagram1april}.
The identity
\[
\theta_{S\widetilde{\tens}\,T}^{\cBrr_{\cD}}=(\theta_S^{\cBrr_{\cD}}\widetilde{\tens}\,\theta_T^{\cBrr_{\cD}})c_{T,S}^{\cBrr_{\cD}}c_{S,T}^{\cBrr_{\cD}}
\]
follows from a straightforward computation in $\cBrr_{\cD}$.

%%%%%%%%%%%%%%%%%%%%%%%%
\begin{theorem}
	\label{theorem RT braided}
	For any braided monoidal category \(\cD=(\cD,\otimes,\mathbb{1},a,\ell,r,c,\theta)\) with twist, there exists a unique strict braided tensor functor
	\[
	\cF_\cD^{\mr{br}}: \cBrr_\cD \to \cD^\str
	\]
	satisfying the following properties:
	\begin{enumerate}
		\item For an object $S$ in $\cBrr_\cD$, $\cF_\cD^{\mr{br}}(S):=S$, with $S$ now viewed as object in $\cD^\str$.
		\item $\cF_\cD^{\mr{br}}$ maps $\theta_S^{\cBrr_\cD}$ to $\theta_S^\str$ for all $S\in\cBrr_\cD$.
		\item $\cF_\cD^{\mr{br}}$ maps a $\cD$-colored coupon to its color.
	\end{enumerate}
\end{theorem}
%%%%%%%%%%%%%%%%%%%%%%%
The fact that $\cF_\cD^{\textup{br}}$ is a strict braided tensor functor in particular implies that 
\[
\cF_\cD^{\mr{br}}\bigl(c_{S,T}^{\cBrr_{\cD}}\bigr)=c_{S,T}^\str
\]
for all objects $S,T\in\cBrr_\cD$.		

The proof of Theorem \ref{theorem RT braided} follows \cite{Reshetikhin&Turaev-1990}:
if the functor is well defined then it is clearly unique.
In order to show that it is well defined
one has to prove that it is constant on Reidemeister equivalence classes and respects the elementary moves involving coupons and the full twist, which is not difficult.

If $\cD$ is braided monoidal with twist and $L,L^\prime$ are two (isotopy classes of) $\cD$-colored ribbon-braid graphs in $\cBrr_\cD$, then we write 
\[
L\doteq L^\prime
\]
if $\cF_\cD^\mr{br}(L)=\cF_\cD^\mr{br}(L^\prime)$. Dot-equality allows to ``melt'' coupons in the way as described in Figure \ref{diagram6new}.
\begin{figure}[H]
	\begin{minipage}{0.49\textwidth}
		\centering
		\includegraphics[scale = 0.7]{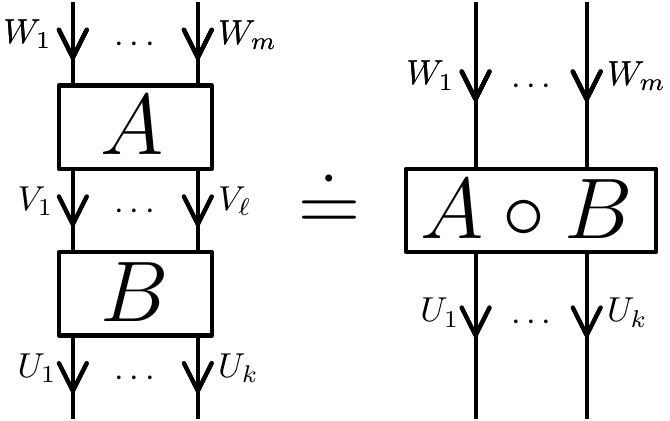}
	\end{minipage}
	\begin{minipage}{0.49\textwidth}
		\centering
		\includegraphics[scale = 0.7]{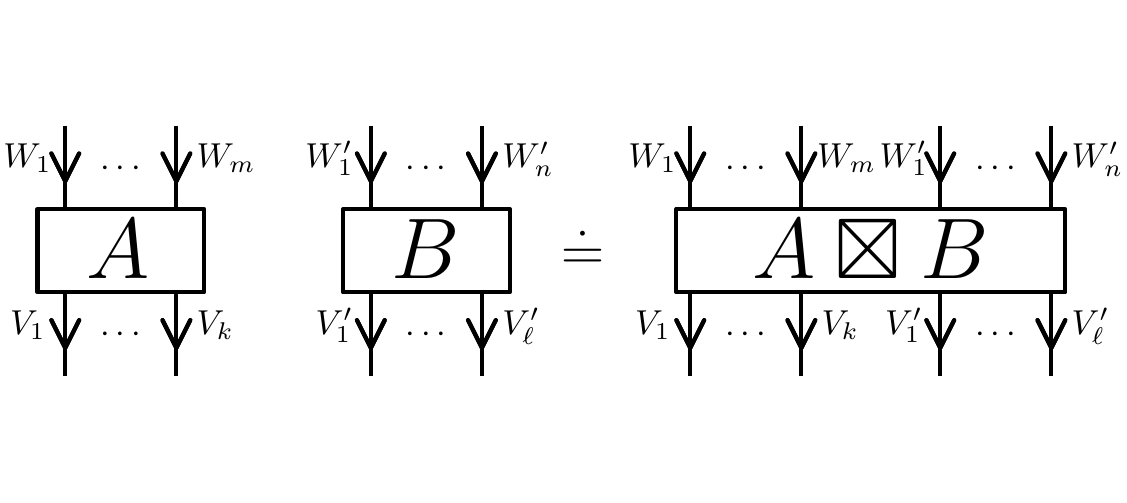}
	\end{minipage}
	\caption{}
	\label{diagram6new}
\end{figure}
Furthermore, any $\cD$-colored ribbon-braid graph is dot-equal to a coupon. We also write $f\doteq L$ for a morphism $f$ in $\cD^\str$ and a
$\cD$-colored ribbon-braid graph $L$ if $L_f\doteq L$, with $L_f$ the coupon colored by $f$.

An example of an identity between morphisms in $\cD^\str$ obtained from 
graphical identities in the category $\cBrr_\cD$ is the Yang-Baxter identity for the
commutativity constraint $c$,
\[
(\id_{V''}\otimes c_{V,V'})(c_{V,V''}\otimes\id_{V'})(\id_V\otimes c_{V',V''}) = (c_{V',V''}\otimes\id_{V})(\id_{V'}\otimes c_{V,V''})
(c_{V,V'}\otimes\id_{V''}).
\]
It is the $\cF_\cD^{\textup{br}}$-image of the Reidemeister move identity in $\cBrr_\cD$ for three bands labeled by $V$, $V'$ and $V''$ respectively, viewed as identity in $\cD$ (cf. Figures \ref{Yang-Baxter eq diagram} and \ref{diagram5new}).

In Section \ref{Section Generalized ABRR} we apply the graphical calculus of ribbon-braid graphs to the braided monoidal module category \(\cM\) with twist defined in Section \ref{Section 1.2}, which contains category $\mathcal{O}$ of $U_q(\mathfrak{g})$. 

%%%%%%%%%%%%%%%%%%%%%%%%%%%%%%%
\begin{remark}
In an analogous manner one can define for a braided monoidal category $(\cD,\otimes,\mathbb{1},a,\ell,r,c)$ the strict braided monoidal category $\cBr_\cD$ of $\cD$-colored braid graphs and show that there exists a unique strict braided tensor functor $\widehat{\cF}_\cD^{\textup{br}}: \cBr_\cD\to\cD^\str$ mapping $S\in\cBr_\cD$ to $S$, viewed as object in $\cD^\str$, and a $\cD$-colored coupon
to its color. If $\cD$ has a twist, then there exists a natural faithful strict tensor functor $\mathcal{I}^{\textup{br}}: \cBr_\cD\rightarrow\cBrr_\cD$ such that the graphical calculi match:
\[
\widehat{\cF}_\cD^{\textup{br}}=\cF_\cD^{\textup{br}}\circ\mathcal{I}^{\textup{br}}.
\]
\end{remark}
%%%%%%%%%%%%%%%%%%%%%%%%%%%%%%%%

%%%%%%%%%%%%%%%%%%%%%%%%%%%%%%%%%%
\subsection{The Reshetikhin-Turaev functor}
\label{Section RT-functor}
%%%%%%%%%%%%%%%%%%%%%%%%%%%%%%%%%%
For a ribbon category $\cD$, the Reshetikhin-Turaev functor from \cite{Reshetikhin&Turaev-1990} provides an extension of $\cF_\cD^{\textup{br}}$
to a strict braided tensor functor from the category of
$\cD$-colored ribbon graphs to $\cD^\str$. We will recall its definition in this subsection.

We start with a monoidal category $\cD=(\cD,\otimes,\mathbb{1},a,\ell,r)$ with left duality. 
A {\it $\cD$-colored ribbon graph} is a ribbon graph with the bands and annuli colored by objects from $\cD$ and coupons colored by appropriate morphisms from $\cD^\str$. 
Before we can describe the admissible colors of the coupons, we need to explain how the coloring of a band within a ribbon graph determines a coloring of its two bases by pairs $(V,\epsilon)$ ($V\in\cD$, $\epsilon\in\{\pm 1\}$).
$V$ is simply the color assigned to the band, and $\epsilon\in\{\pm 1\}$ is determined by the rule that if the base is a bottom (respectively top) base, then $\epsilon=+1$ if and only if the orientation of the core of the band is directed towards (respectively away from) the base. This rule is depicted by Figure \ref{diagram7new}, where we have added the horizontal line in order to indicate whether the base is a top or a bottom base.
\begin{figure}[H]
	\centering
	\includegraphics[scale = 0.8]{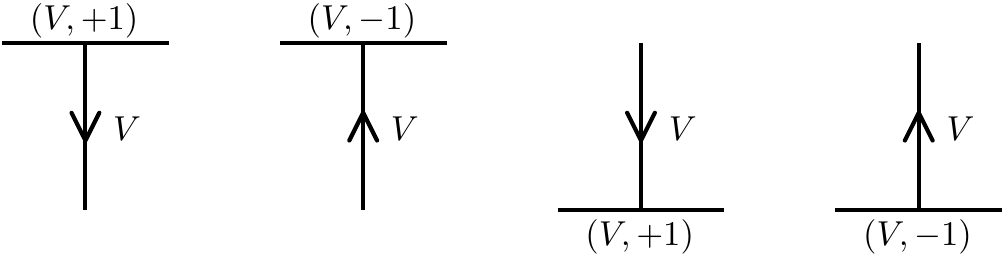}
	\caption{}
	\label{diagram7new}
\end{figure}
Note that with this rule, the orientation and the color of a band in a $\cD$-colored ribbon graph can be uniquely recovered from the pairs of colors of its bases (recall here that the unoriented ribbon graph already determines whether a base of a band is a bottom base or a top base). Furthermore, the two bases of a band have the same sign if and only if one base is a top base and the other a bottom base. 

If a coupon in a ribbon graph with $\cD$-colored bands and annuli has $k$ bases of bands lying on its bottom base, with colors $(V_1,\delta_1),\ldots,(V_k,\delta_k)$ enumerated in counterclockwise order,
and $\ell$ bases of bands lying on its top base, with colors $(W_1,\epsilon_1),\ldots,(W_\ell,\epsilon_\ell)$ enumerated in clockwise order, then we require that the coupon is colored by a morphism
\[
A\in\textup{Hom}_{\cD^\str}\bigl((V_1^{\delta_1},\ldots,V_k^{\delta_k}),(W_1^{\epsilon_1},\ldots,W_\ell^{\epsilon_\ell})\bigr)
\]
where
$V^{+1}:=V$ and $V^{-1}:=V^*$ for $V\in\cD$. Isotopies of $\cD$-colored ribbon graphs are isotopies of ribbon graphs preserving the colors of the bands, annuli and coupons.

In the same way we can now introduce (isotopies of) $\cD$-colored ribbon graph diagrams. The description of isotopies of ribbon graphs in terms of their diagrams extends
in the obvious way to $\cD$-colored ribbon graphs. The pictures of $\cD$-colored ribbon graphs that we will present, will always be pictures of the associated $\cD$-colored ribbon graph diagrams. 
An example of a $\cD$-colored ribbon graph diagram, obtained by $\cD$-coloring the $(3,2)$-ribbon graph diagram of Figure \ref{diagram1}, is given in Figure
\ref{diagram8}. In this diagram $A\in\textup{Hom}_{\cD^\str}\bigl((V_1,V_2^*),(W_1,W_2^*,W_3^*)\bigr)$.
\begin{figure}[H]
	\centering
	\includegraphics[scale = 0.7]{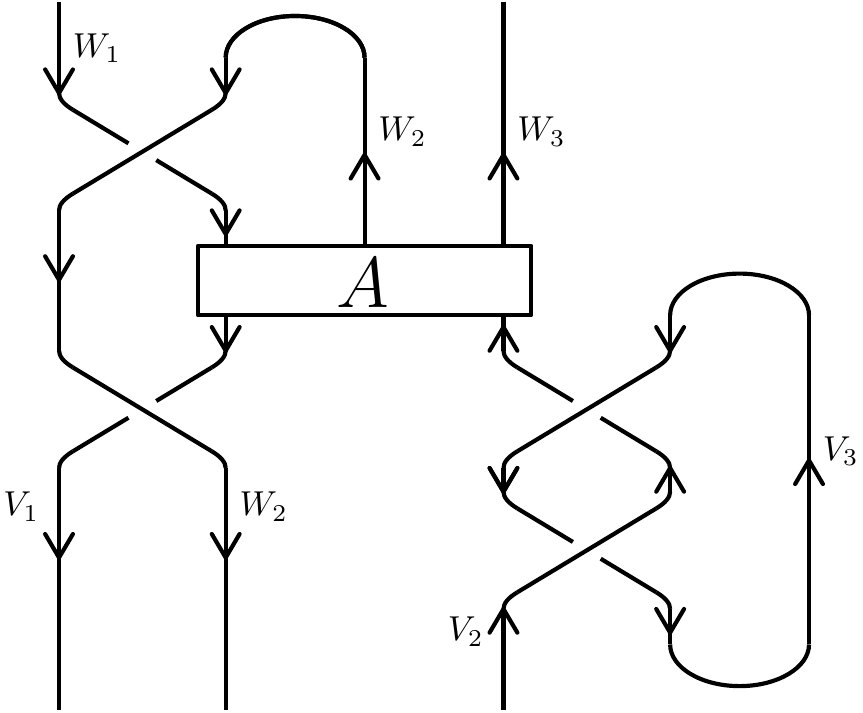}
	\caption{}
	\label{diagram8}
\end{figure}

The category $\textup{Rib}_{\cD}$ of $\cD$-colored ribbon graphs is now defined as follows. The objects are $\emptyset$ and $k$-tuples $((V_1,\delta_1),\ldots, (V_k,\delta_k))$
of pairs $(V_i,\delta_i)$ consisting of objects $V_i\in\cD$ and signs $\delta_i\in\{\pm 1\}$. We call such a $k$-tuple an object of length $k$, and we view $\emptyset$ as the single object of length $0$.
For objects $S=((V_1,\delta_1),\ldots,(V_k,\delta_k))$ and $T=((W_1,\epsilon_1),\ldots,(W_\ell,\epsilon_\ell))$ of length $k\in\mathbb{Z}_{>0}$ and $\ell\in\mathbb{Z}_{>0}$ respectively, the class of morphisms $\textup{Hom}_{\textup{Rib}_\cD}(S,T)$ consists of the isotopy classes of $\cD$-colored ribbon graphs with bottom extremal bases colored counterclockwise by 
$(V_1,\delta_1),\ldots, (V_k,\delta_k)$ and top extremal bases colored clockwise by $(W_1,\epsilon_1),\ldots,(W_\ell,\epsilon_\ell)$. For \(k=0\) or \(\ell = 0\) one should replace \(S\) respectively \(T\) by \(\emptyset\) and construct the corresponding class of morphisms in analogy to Subsection \ref{GcBraid}. For example, the source of the $\cD$-colored $(3,2)$-ribbon graph in Figure \ref{diagram8} is the tuple $((V_1,+1),(W_2,+1),(V_2,-1))$ and its target is 
$((W_1,+1),(W_3,-1))$.

Composition is defined by vertically stacking the $\cD$-colored ribbon graphs. Note that this is well defined since attaching a bottom extremal base of a band with color $(V,\delta)$ to a top extremal base of a band with the same color $(V,\delta)$ is guaranteeing compatibility of the color of the bands and of the orientation of their strands. The identity morphism 
$\textup{id}_{((V_1,\delta_1),\ldots,(V_k,\delta_k))}$ is represented by $k$ parallel vertical bands with the $k$ bottom extremal bases colored counterclockwise by $(V_1,\delta_1),\ldots,(V_k,\delta_k)$ (as a result, the $k$ top extremal bases are colored clockwise by $(V_1,\delta_1),\ldots,(V_k,\delta_k)$ and
 the orientation of the $i^{\textup{th}}$ strand is from top to bottom if and only if $\delta_i=+1$). 

The category $\textup{Rib}_{\cD}$ is a strict monoidal category with unit object $\emptyset$ and tensor product $\widetilde{\tens}_{\mr{Rib}_\cD}$ defined on objects by concatenation of the tuples of pairs
$(V,\delta)$, and on morphisms by placing the $\cD$-colored ribbon graphs next to each other. We will simply write \(\widetilde{\tens}\) if no confusion can arise. $\textup{Rib}_{\cD}$ is braided with commutativity constraint
$c^{\textup{Rib}_\cD}=(c^{\textup{Rib}_{\cD}}_{S,T})_{S,T\in\textup{Rib}_{\cD}}$ given by Figure \ref{diagram9new}.
\begin{figure}[H]
	\centering
	\includegraphics[scale = 0.9]{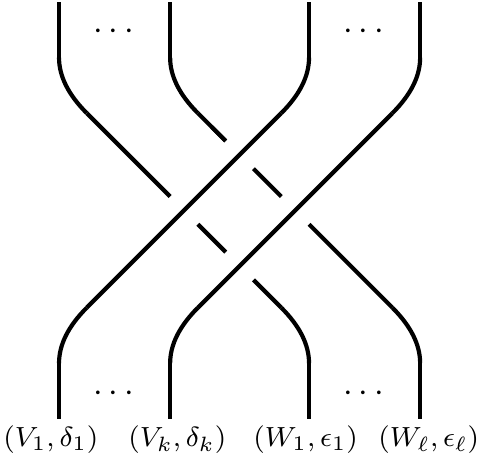}
	\caption{\(c^{\mr{Rib}_{\cD}}_{S,T}\) with \(S = ((V_1,\delta_1),\dots,(V_k,\delta_k))\) and \(T = ((W_1,\epsilon_1),\dots,(W_\ell,\epsilon_\ell))\).}
	\label{diagram9new}
\end{figure}
Note here that the color and the orientation of the bands, as well as the pairs of colors of the top extremal bases, are completely determined by the pairs of colors of the bottom extremal bases.
In fact, the pair attached to the top extremal base of a band is the same as the pair at its bottom extremal base.

The monoidal category $\textup{Rib}_{\cD}$ admits a left duality. The left dual $(S^*,e_S^{\textup{Rib}_{\cD}},\iota_S^{\textup{Rib}_{\cD}})$ to an object $S\in\mr{Rib}_{\cD}$
is defined by 
\[
\emptyset^*:=\emptyset,\qquad ((V_1,\delta_1),\ldots,(V_k,\delta_k))^*:=((V_k,-\delta_k),\ldots,(V_1,-\delta_1)),
\]
with the morphisms $e_S^{\textup{Rib}_{\cD}}$ and $\iota_S^{\textup{Rib}_{\cD}}$ defined by Figure \ref{diagram10new}.
\begin{figure}[H]
	\begin{minipage}{0.49\textwidth}
		\centering
		\includegraphics[scale = 1]{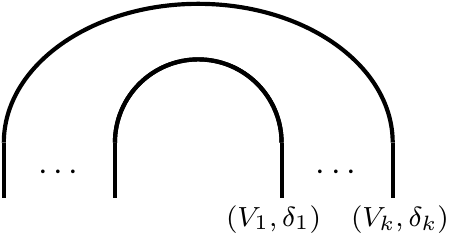}
	\end{minipage}
	\begin{minipage}{0.49\textwidth}
		\centering
		\includegraphics[scale = 1]{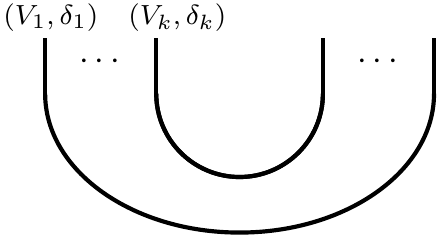}
	\end{minipage}
	\caption{\(e_S^{\mr{Rib}_\cD}\) and \(\iota_S^{\mr{Rib}_\cD}\) for \(S = ((V_1,\delta_1),\dots,(V_k,\delta_k))\)}
	\label{diagram10new}
\end{figure}

As before, the color and the orientation of the bands, as well as the pairs of colors of the extremal bases which are uncolored in Figure \ref{diagram10new}, are determined by the pairs of colors of the extremal bases that are colored in Figure \ref{diagram10new}. Note that the pair attached to an uncolored base in Figure \ref{diagram10new} now has opposite sign compared to the base on the opposite side of its band. 

If a morphism $S\rightarrow T$ in $\textup{Rib}_{\cD}$ is represented by a $\cD$-colored ribbon tangle diagram $L$, then its dual is the $\cD$-colored ribbon tangle diagram obtained from $L$ by rotating it by 180 degrees while preserving the orientations and the colors of its strands and annuli. This has the effect that top (respectively bottom) bases in $L$ become bottom (respectively top) bases in the rotated diagram. Furthermore, the sign is flipped in the pairs of colors assigned to the bases of the bands. This simple description of the dual of a morphism does not work for morphisms in $\textup{Rib}_{\cD}$ involving coupons, due to the rule that the bottom base of a coupon in a $\cD$-colored ribbon graph diagram should always be below its top base.

It is clear from the previous paragraph that $(S^*)^*=S$ for all $S\in\textup{Rib}_{\cD}$, and 
\[
\bigl(c_{S,T}^{\textup{Rib}_{\cD}}\bigr)^*=c_{S^*,T^*}^{\textup{Rib}_{\cD}},\qquad
\bigl(e_S^{\textup{Rib}_{\cD}}\bigr)^*=\iota_{S^*}^{\textup{Rib}_{\cD}},\qquad
\bigl(\iota_S^{\textup{Rib}_{\cD}}\bigr)^*=e_{S^*}^{\textup{Rib}_{\cD}}.
\]

The braided monoidal category
$\textup{Rib}_{\cD}$ with left duality is a ribbon category with twist $\theta^{\textup{Rib}_{\cD}}=(\theta^{\textup{Rib}_{\cD}}_S)_{S\in\textup{Rib}_{\cD}}$ the collection of functorial
isomorphisms $\theta_S^{\textup{Rib}_{\cD}}: S\overset{\sim}{\longrightarrow}S$, defined for objects of length $>0$ by Figure \ref{diagram11new}.
\begin{figure}[H]
	\centering
	\includegraphics[scale = 0.8]{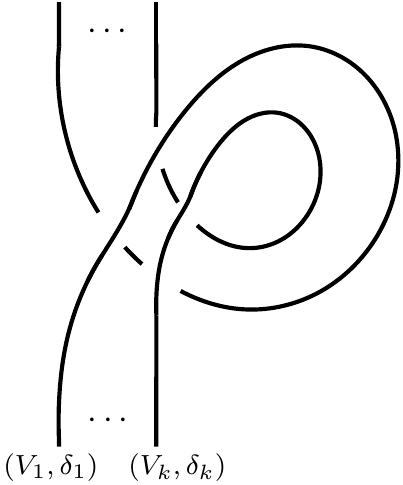}
	\caption{$\theta_S^{\textup{Rib}_{\cD}}$ for $S=((V_1,\delta_1),\ldots,(V_k,\delta_k))$.}
	\label{diagram11new}
\end{figure}
\noindent
Furthermore, $\theta_\emptyset^{\textup{Rib}_{\cD}}:=\textup{id}_\emptyset$.
Indeed, functoriality is immediate for morphisms representable by $\cD$-colored ribbon tangle diagrams. For morphisms involving coupons this is a consequence of the three elementary moves involving coupons, see $\textup{Rel}_{11}$--$\textup{Rel}_{13}$ in \cite[\S 5]{Reshetikhin&Turaev-1990}. The two identities
\[
\theta_{S\widetilde{\tens}\,T}^{\textup{Rib}_{\cD}}=(\theta_S^{\textup{Rib}_{\cD}}\widetilde{\tens}\,\theta_T^{\textup{Rib}_{\cD}})c_{T,S}^{\textup{Rib}_{\cD}}c_{S,T}^{\textup{Rib}_{\cD}},\qquad \bigl(\theta_S^{\textup{Rib}_{\cD}}\bigr)^*=\theta_{S^*}^{\textup{Rib}_{\cD}}
\]
follow from a straightforward computation in $\textup{Rib}_{\cD}$.

%%%%%%%%%%%%%%%%%%%%%%%%%%%%%%%%%%%%%%%%
\begin{theorem}[Reshetikhin--Turaev \cite{Reshetikhin&Turaev-1990}]
	\label{theorem Reshetikhin-Turaev}
	For any ribbon category \(\mathcal{D}\), there exists a unique strict braided tensor functor
	\[
	\mathcal{F}_{\mathcal{D}}^{\mathrm{RT}}: \mathrm{Rib}_{\mathcal{D}} \to \mathcal{D}^\str
	\]
	satisfying  the following properties:
	\begin{enumerate}
		\item  
		$\mathcal{F}_{\cD}^{\textup{RT}}(((V_1,\delta_1),\ldots,(V_k,\delta_k)))=(V_1^{\delta_1},\ldots, V_k^{\delta_k})$.
		\item For objects $S=((V_1,+1),\ldots,(V_k,+1))\in\textup{Rib}_{\cD}$ only involving $(+1)$-signs, 
		\begin{equation}\label{plusimage}
		\cF_\cD^{\mathrm{RT}}(e_{S}^{\textup{Rib}_{\cD}})=e_{\mathcal{F}_{\cD}^{\textup{RT}}(S)}^\str,\qquad
		\cF_\cD^{\textup{RT}}(\iota_{S}^{\textup{Rib}_{\cD}})=\iota_{\mathcal{F}_{\cD}^{\textup{RT}}(S)}^\str,\qquad 
		\cF_\cD^{\textup{RT}}(\theta_S^{\textup{Rib}_{\cD}})=\theta_{\mathcal{F}_{\cD}^{\textup{RT}}(S)}^\str.
		\end{equation}
		\item $\mathcal{F}_{\cD}^{\textup{RT}}$ maps a $\cD$-colored coupon to its color.
	\end{enumerate}
\end{theorem}
%%%%%%%%%%%%%%%%%%%%%%%%%%%%%%%%%%%%%%%%%%%% 
The fact that $\cF_\cD^{\textup{RT}}$ is a strict braided tensor functor in particular implies that 
\[
\cF_\cD^{\textup{RT}}\bigl(c_{S,T}^{\textup{Rib}_{\cD}}\bigr)=c_{\mathcal{F}_{\cD}^{\textup{RT}}(S),\mathcal{F}_{\cD}^{\textup{RT}}(T)}^\str
\]
for all objects $S,T\in\textup{Rib}_\cD$.

Recall that a ribbon category $\cD$ is automatically rigid (see Subsection \ref{bmsection}). The right dual of $\cD$ provides a right dual $(S^*,\widetilde{e}_S^{\,\str},\widetilde{\iota}_S^{\,\str})$ for each $S\in\cD^\str$, with evaluation morphism $\widetilde{e}_S^{\,\str}: S\tens S^*\rightarrow\emptyset$ and injection morphism $\widetilde{\iota}_S^{\,\str}: \emptyset\rightarrow S^*\tens S$ obtained recursively from the evaluation morphisms $\widetilde{e}_V: V\otimes V^*\rightarrow\mathbb{1}$ and injection 
morphisms $\widetilde{\iota}_V: \mathbb{1}\rightarrow V^*\otimes V$ for $V\in\cD$ (cf. Subsection \ref{StSection}). A straightforward computation in $\textup{Rib}_{\cD}$
together with \eqref{ecomp} implies that
\begin{equation}\label{firstminusdual}
\cF_\cD^{\textup{RT}}(e_{((V_k,-1),\ldots,(V_1,-1))}^{\textup{Rib}_{\cD}})=\widetilde{e}_{(V_1,\ldots,V_k)}^{\,\str},\qquad 
\cF_\cD^{\textup{RT}}(\iota_{((V_k,-1),\ldots,(V_1,-1))}^{\textup{Rib}_{\cD}})=\widetilde{\iota}_{(V_1,\ldots,V_k)}^{\,\str}.
\end{equation}
The images under $\cF_\cD^{\textup{RT}}$
of $e_{S}^{\textup{Rib}_{\cD}}$ and $\iota_{S}^{\textup{Rib}_{\cD}}$ can now easily be computed  for any object $S\in\textup{Rib}_\cD$ using \eqref{plusimage} and \eqref{firstminusdual} (note for instance that in $\textup{Rib}_{\cD}$ the evaluation and injection morphisms colored by objects of length $k$ can be written as compositions of elementary $\cD$-ribbon tangles involving evaluation and injection morphisms colored by objects of length $1$).

Note that $\cF_\cD^{\textup{RT}}(L^*)=\cF_\cD^{\textup{RT}}(L)^*$ if $L\in\textup{Hom}_{\textup{Rib}_\cD}(S,T)$ has source and target containing only $(+1)$-signs. In particular, 
\begin{equation}\label{firstminustwist}
\cF_\cD^{\textup{RT}}(\theta_{(V,-1)}^{\textup{Rib}_\cD})=(\theta_{(V)}^\str)^*.
\end{equation}
Then $\cF_\cD^{\mr{RT}}(\theta_{S}^{\textup{Rib}_\cD})$ can be computed for any object $S$ using \eqref{plusimage}, \eqref{firstminustwist} and the compatibility of the twist with the tensor product.

The following result is now immediate.
%%%%%%%%%%%%%%%%%%%%%%%%%%%%%%%%%%%%%%%%%%%
\begin{proposition}\label{RTbr}
Let $\cD$ be a ribbon category. There exists a unique strict braided tensor functor $\mathcal{I}_{\cD}: \cBrr_\cD\rightarrow\textup{Rib}_{\cD}$ satisfying
\begin{enumerate}
\item $\mathcal{I}_{\cD}((V_1,\ldots,V_k)):=((V_1,+1),\ldots,(V_k,+1))$,
\item For $S,T\in\cBrr_\cD$ and $L\in\textup{Hom}_{\cBrr_\cD}(S,T)$ we set $\mathcal{I}_{\cD}(L)$ to be $L$, viewed as morphism in 
$\textup{Hom}_{\textup{Rib}_{\cD}}(\mathcal{I}_{\cD}(S),\mathcal{I}_{\cD}(T))$ \textup{(}in other words, the $\cD$-colored ribbon-braid graph diagram representing $L$ is viewed as $\cD$-colored ribbon graph diagram in such a way that the extremal bases all have $(+1)$-signs\textup{)}.
\end{enumerate}
The functor $\mathcal{I}_\cD$ is faithful and compatible with the two ``graphical calculus'' functors $\cF_\cD^{\textup{RT}}$ and $\cF_\cD^{\textup{br}}$, in the sense that
\[
\cF_\cD^{\textup{br}}=\cF_\cD^{\textup{RT}}\circ\mathcal{I}_{\cD}.
\]
\end{proposition}
%%%%%%%%%%%%%%%%%%%%%%%%%%%%%%%%%%%%%%%%%%
If $\cD$ is a ribbon category and $L,L^\prime$ are two (isotopy classes of) $\cD$-colored ribbon graphs in $\textup{Rib}_{\cD}$, then we write $L\doteq L^\prime$ if
$\cF_\cD^{\textup{RT}}(L)=\cF_\cD^{\textup{RT}}(L^\prime)$. This is compatible with the dot-equality we defined earlier for $\cD$-colored ribbon-braid graphs. 

So besides the possibility to ``melt'' coupons (cf. Figure \ref{diagram6new}) and viewing $\cD$-colored ribbon graphs as coupons, one can now also move coupons through local maxima and minima at the cost of changing its color by its dual, see Figure \ref{diagram12new} for one of the cases.
\begin{figure}[H]
	\centering
	\includegraphics[scale = 0.9]{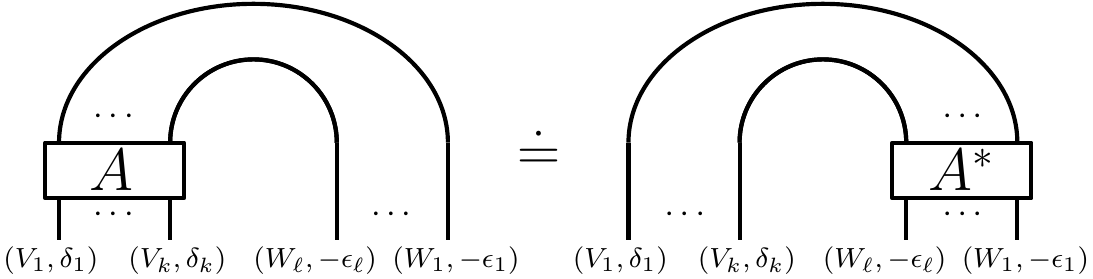}
	\caption{ \(A\in\Hom_{\cD^\str}(((V_1,\delta_1),\dots,(V_k,\delta_k)),((W_1,\epsilon_1),\dots,(W_\ell,\epsilon_\ell)))\).}
	\label{diagram12new}
\end{figure}

%%%%%%%%%%%%%%%%%%%%%%%%%%%%%%%%%%%%%%%%
\subsection{Colored ribbon graph subdiagrams in $\mathbb{B}_{\cD}$}\label{mixedsection}
%%%%%%%%%%%%%%%%%%%%%%%%%%%%%%%%%%%%%%%%
Consider the set-up with $\cD$ a braided mo\-noi\-dal category with twist, and $\cC$ a full braided mo\-noi\-dal subcategory of $\cD$ admitting a compatible left duality, turning it into a ribbon category. Examples relevant for this paper are $(\cC,\cD)=(\Rep,\cM)$ and $(\cC,\cD)=(\Rep,\cM_{\textup{adm}})$. In this case we will draw $\cD$-colored ribbon-braid graph diagrams containing $\cC$-colored ribbon graph diagrams, and interpret them as $\cD$-colored ribbon-braid graph diagrams by the following local rules:
\begin{enumerate}
\item\label{localrule1} A strand colored by $V\in\cC$ and oriented upwards (i.e., the bottom and top base are colored by $(V,-1)$), stands for the strand 
colored by $V^*$ and oriented downwards.
\item\label{localrule2} 
 The $\cC$-colored ribbon graph diagrams $e_{(V,+1)}^{\textup{Rib}_{\cC}}$,
$e_{(V,-1)}^{\textup{Rib}_{\cC}}$, $\iota_{(V,+1)}^{\textup{Rib}_{\cC}}$ and
$\iota_{(V,-1)}^{\textup{Rib}_{\cC}}$ stand for the $\cD$-colored ribbon-braid graph diagrams depicted in Figure \ref{diagramB1}.
\begin{figure}[H]
	\centering
	\includegraphics[scale = 0.75]{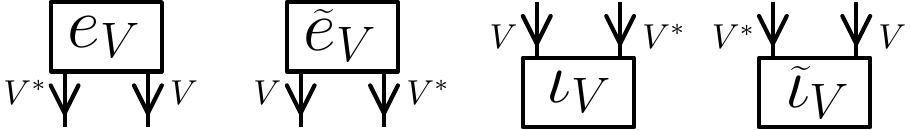}
	\caption{}
	\label{diagramB1}
\end{figure}
\item\label{localrule3} A $\cC$-colored coupon $((V_1,\delta_1),\ldots,(V_k,\delta_k))\rightarrow ((W_1,\epsilon_1),\ldots,(W_\ell,\epsilon_\ell))$ colored by the morphism
$A\in\textup{Hom}_{\cC^\str}((V_1^{\delta_1},\ldots,V_k^{\delta_k}),(W_1^{\epsilon_1},\ldots,W_\ell^{\epsilon_\ell}))$ stands for the coupon 
\[
(V_1^{\delta_1},\ldots,V_k^{\delta_k})\rightarrow (W_1^{\epsilon_1},\ldots,W_\ell^{\epsilon_\ell})
\]
colored by $A$, now viewed as morphism in $\cD^\str$.
\end{enumerate}
These rules provide an interpretation of any $\cD$-colored ribbon-braid graph diagram $L$ with $\cC$-colored ribbon graph subdiagrams as a $\cD$-colored ribbon-braid graph diagram. For instance, the crossings depicted in Figures \ref{diagramB3} and \ref{diagramB4} 
\begin{figure}[H]
	\begin{minipage}{0.48\textwidth}
		\centering
		\includegraphics[scale = 0.8]{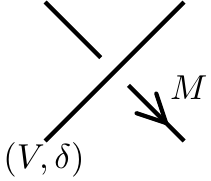}
		\captionof{figure}{}
		\label{diagramB3}
	\end{minipage}
	\begin{minipage}{0.48\textwidth}
		\centering
		\includegraphics[scale = 0.8]{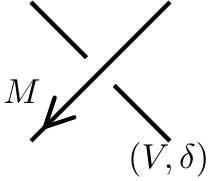}		
		\captionof{figure}{}
		\label{diagramB4}
	\end{minipage}
\end{figure}
\noindent
for $M\in\cD$, $V\in\cC$ and $\delta\in\{\pm 1\}$ stand for $c_{V^\delta,M}^{\textup{Rib}_{\cD}}$ and $c_{M,V^\delta}^{\textup{Rib}_{\cD}}$ respectively, while
$\theta_{(V,-1)}^{\textup{Rib}_{\cC}}$ stands for the $\cD$-colored ribbon-braid graph diagram given in Figure \ref{diagramB2}.
\begin{figure}[H]
	\centering
	\includegraphics[scale = 0.75]{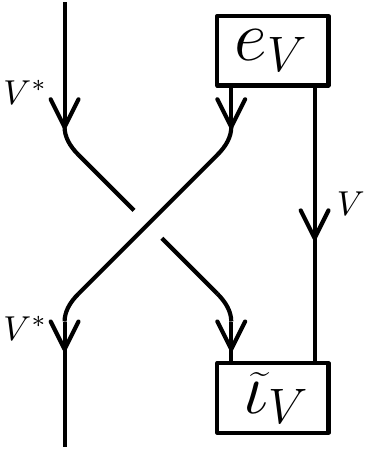}
	\caption{}
	\label{diagramB2}
\end{figure}
Applying an isotopy (of colored ribbon graph diagrams) to a $\cC$-colored ribbon graph subdiagram of a ribbon-braid graph diagram $L$ results in a $\cD$-colored ribbon-braid graph diagram $L^\prime$, which may not be isotopic to $L$ as $\cD$-colored ribbon-braid graph diagram, but still satisfies
\[
L\doteq L^\prime.
\]
Note that isotopies of $\cD$-colored ribbon-braid graph diagrams with entangled $\cC$-colored ribbon graph diagrams include the moves of pulling coupons, local maxima and local minima in $\cC$-colored strands through mixed crossings, like in Figure \ref{diagramB5} (which should be read as equality in $\mathbb{B}_{\cD}$).
\begin{figure}[H]
	\centering
	\includegraphics[scale = 0.75]{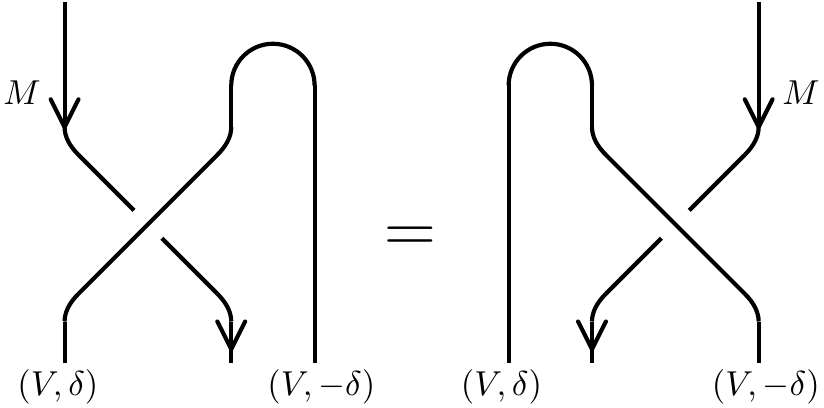}
	\caption{}
	\label{diagramB5}
\end{figure}
\noindent
It is an instructive exercise to check these identities in $\mathbb{B}_{\cD}$. For Figure \ref{diagramB5} for instance, this amounts to expressing the diagram in the left-hand side as a $\cD$-colored ribbon-braid graph diagram using the rules (\ref{localrule1})--(\ref{localrule3}), then applying the second Reidemeister move (cf.\ Figure \ref{universal R-matrix double crossing}) to the two parallel strands colored by $M$ and $V^{-\delta}$, then pulling the strand colored by $M$ over the coupon colored by the evaluation morphism, and finally rewriting it again as $\cD$-colored ribbon-braid graph diagram with a $\cC$-colored ribbon graph subdiagram.

%%%%%%%%%%%%%%%%%%%%%%%%%%%%%%%%%%%%%%%%%%%%%%%
\subsection{Fusion morphisms}
\label{Section 2 strictified}
%%%%%%%%%%%%%%%%%%%%%%%%%%%%%%%%%%%%%%%%%%%%%%%
Let $\cD$ be a braided monoidal category with twist.

For $S=(V_1,\ldots,V_k)\in\cD^\str$ we introduce Figure \ref{j double} (respectively Figure \ref{theta double}) for the coupon $\mathbb{J}_S^{\mathbb{B}_{\cD}}$ (respectively $\mathbb{I}_{S}^{\mathbb{B}_{\cD}}$) in $\mathbb{B}_{\cD}$ colored by the fusion morphism
 $J_S\in\textup{Hom}_{\cD^\str}(S,\cF^\str(S))$ (respectively $J_S^{-1}$). 
In other words, in their graphical notation we shrink the coupon to a vertex. 
The incoming directions of the strands at the vertex determine how their bases are connected to the bases of the coupon.
\begin{figure}[H]
	\centering
	\begin{minipage}{0.48\textwidth}
		\centering
		\includegraphics[scale = 0.8]{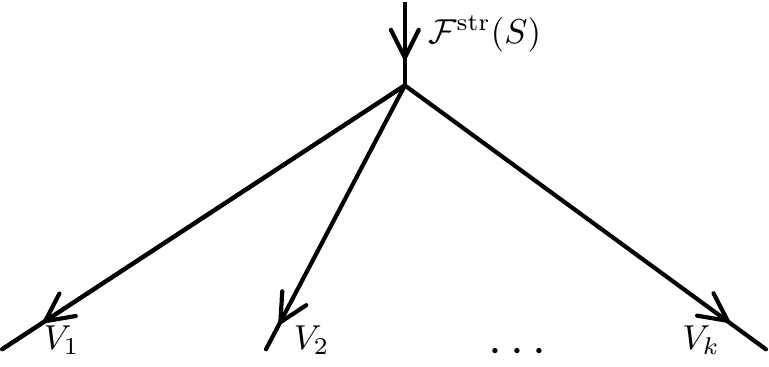}
		\captionof{figure}{ }
		\label{j double}
	\end{minipage}
	\begin{minipage}{0.48\textwidth}
		\centering
		\includegraphics[scale = 0.8]{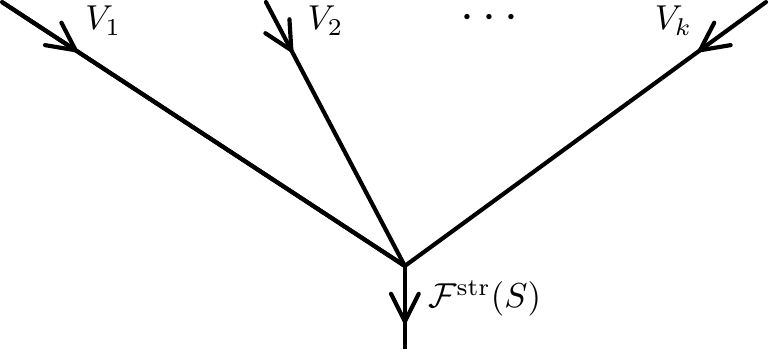}
		\caption{ }
		\label{theta double}
	\end{minipage}
\end{figure}
\noindent
The ribbon-braid graph diagram $\mathbb{J}_S^{\mathbb{B}_{\cD}}$ zips $k$ vertical parallel strands labeled by $V_1,\ldots,V_k$ to a single strand labeled by
$V_1\otimes\cdots\otimes V_k$. The $2$-cocycle identity \eqref{basic j} shows that Figure \ref{j double} is dot-equal to 
zipping the $k$ parallel strands labeled by $V_1,\ldots,V_k$ step by step, irrespective of the order in which the neighboring strands are zipped. A similar remark can be made about
$\mathbb{I}_S^{\mathbb{B}_{\cD}}$ and unzipping. Note that
\[
\mathbb{J}_S^{\mathbb{B}_{\cD}}\circ\mathbb{I}_S^{\mathbb{B}_{\cD}}\doteq\textup{id}_{(\cF^\str(S))}, \qquad\qquad
\mathbb{I}_S^{\mathbb{B}_{\cD}}\circ\mathbb{J}_S^{\mathbb{B}_{\cD}}\doteq\textup{id}_S.
\]

The naturality of the braiding and the twist allows to move the zip and unzip diagrams from Figures \ref{j double} and \ref{theta double} through a crossing and through a full twist in $\cBrr_\cD$. For instance, we have the equalities in $\cBrr_\cD$ depicted in Figures \ref{diagram14new} and \ref{diagram4april}.
\begin{figure}[H]
	\centering
	\begin{minipage}{0.48\textwidth}
		\centering
		\includegraphics[scale = 0.65]{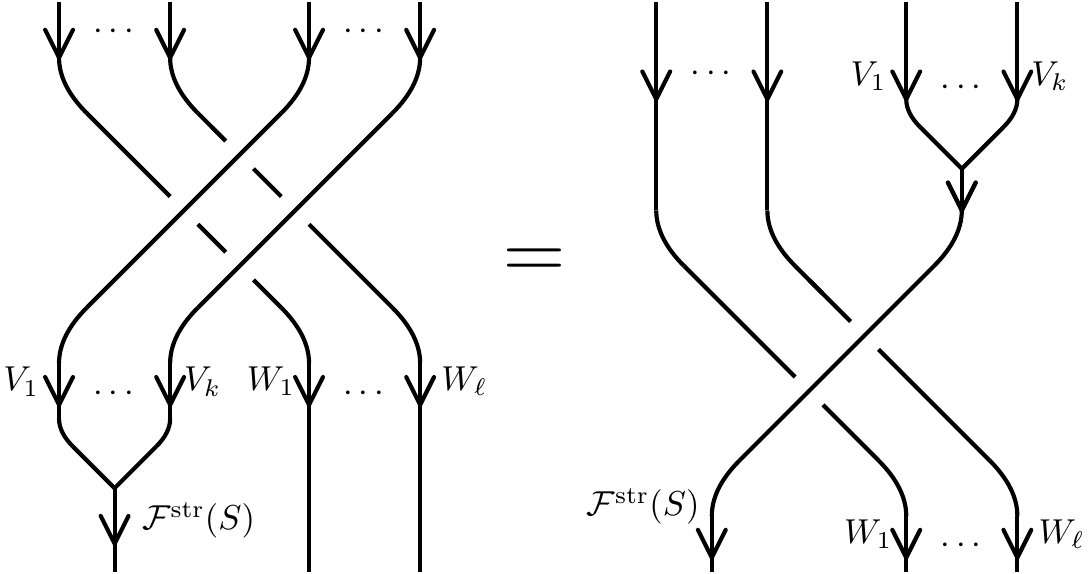}
		\captionof{figure}{ }
		\label{diagram14new}
	\end{minipage}\quad
	\begin{minipage}{0.48\textwidth}
		\centering
		\includegraphics[scale = 0.65]{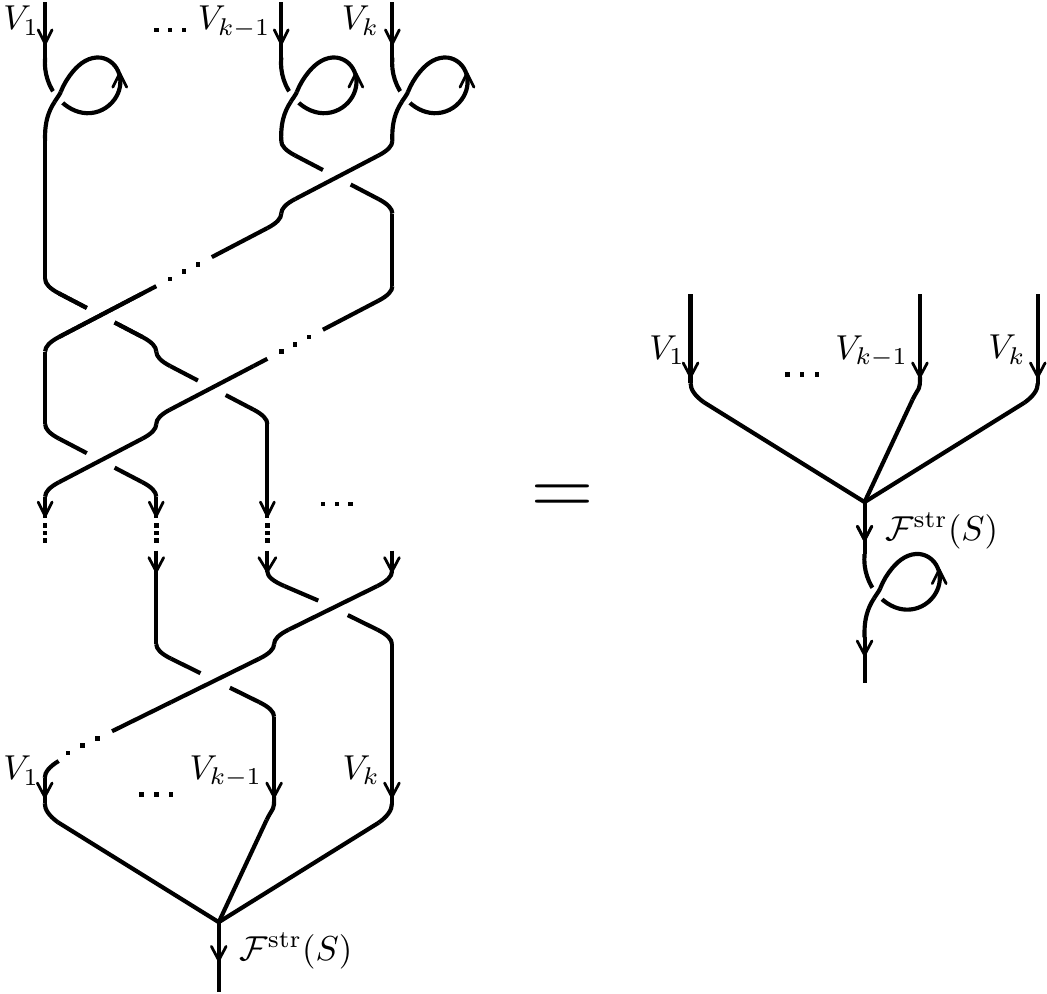}
		\caption{ }
		\label{diagram4april}
	\end{minipage}
\end{figure}

For $S,T\in\cD^\str$, the elementary diagrams $\textup{id}_S^{\mathbb{B}_{\cD}}$, $c_{S,T}^{\mathbb{B}_{\cD}}$ and $\theta_S^{\mathbb{B}_{\cD}}$ can be re-expressed in the graphical calculus as diagrams labeled by objects of length one by means of the following dot-equalities:
\begin{equation}\label{algebraicformulazip}
\begin{split}
\textup{id}_S^{\mathbb{B}_{\cD}}&\doteq\mathbb{I}_S^{\mathbb{B}_{\cD}}\circ\textup{id}_{\cF^\str_{\cD}(S)}^{\mathbb{B}_{\cD}}\circ \mathbb{J}_S^{\mathbb{B}_{\cD}},\\
c_{S,T}^{\mathbb{B}_{\cD}}&\doteq\bigl(\mathbb{I}_T^{\mathbb{B}_{\cD}}\widetilde{\tens}\,\mathbb{I}_S^{\mathbb{B}_{\cD}}\bigr)\circ c_{\cF^\str_{\cD}(S),\cF^\str_{\cD}(T)}^{\mathbb{B}_{\cD}}\circ
(\mathbb{J}_S^{\mathbb{B}_{\cD}}\widetilde{\tens}\mathbb{J}_T^{\mathbb{B}_{\cD}}),\\
\theta_S^{\mathbb{B}_{\cD}}&\doteq \mathbb{I}_S^{\mathbb{B}_{\cD}}\circ\theta_{\cF^\str_{\cD}(S)}^{\mathbb{B}_{\cD}}\circ\mathbb{J}_S^{\mathbb{B}_{\cD}}.
\end{split}
\end{equation}
The second identity in (\ref{algebraicformulazip}) is depicted in Figure \ref{doublezip}, the others have similar graphical presentations.

\begin{figure}[H]
	\centering
	\includegraphics[scale = 0.7]{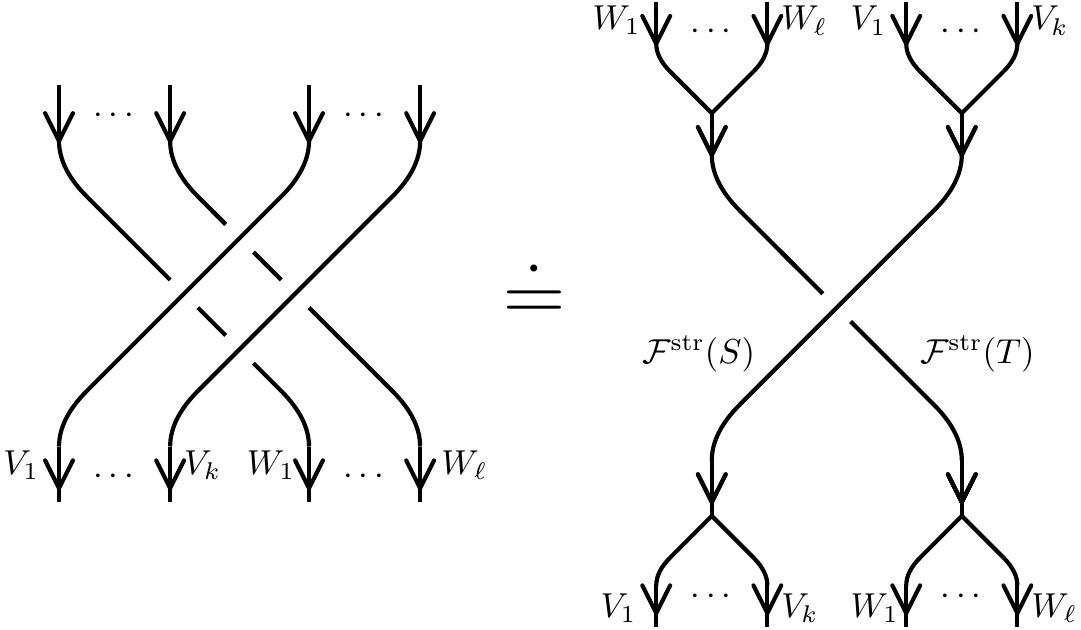}
	\caption{}
	\label{doublezip}
\end{figure}

Now let $S=(V_1,\ldots,V_k),T=(W_1,\ldots,W_\ell)\in\cD^\str$ and $A\in\textup{Hom}_{\cD^\str}(S,T)$. Zipping the parallel strands at the top and bottom base of the coupon labeled by $A$ is dot-equal to the coupon labeled by $(\cG^\str\circ\cF^\str)(A)$ by \eqref{meaning of natural}, as depicted in Figure \ref{diagram13new}.
\begin{figure}[H]
	\centering
	\includegraphics[scale = 0.8]{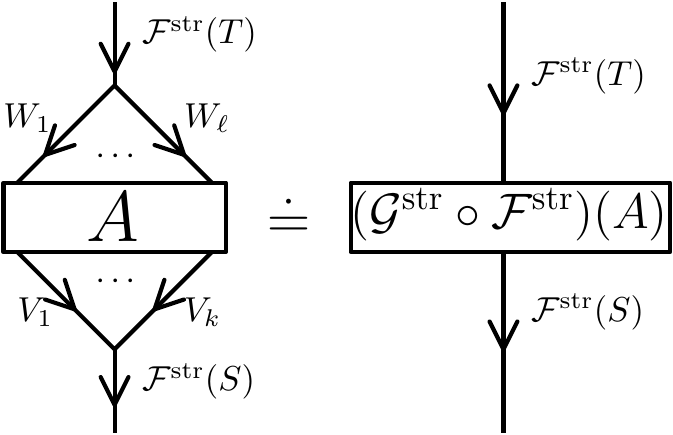}
	\caption{}
	\label{diagram13new}
\end{figure}

Suppose now that $\cD$ is a ribbon category, and consider the above identities in $\cBrr_\cD$ as identities in $\textup{Rib}_{\cD}$ via the functor $\mathcal{I}_{\cD}$. Then the $\cD$-colored ribbon-braid graph diagram in the left-hand side of Figure \ref{diagram4april} can alternatively be represented by the $\cD$-colored ribbon graph diagram depicted in Figure \ref{diagram5april},
\begin{figure}[H]
	\centering
	\includegraphics[scale = 0.7]{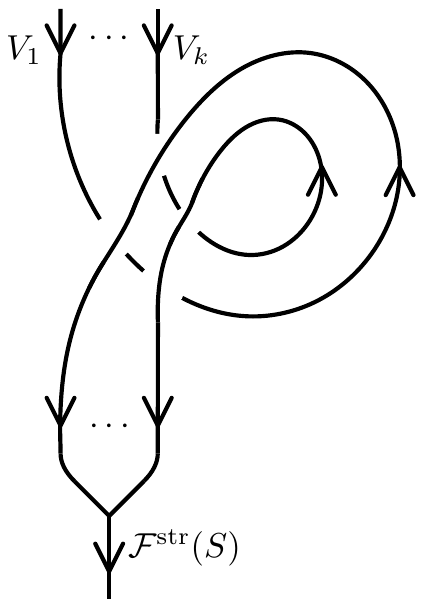}
	\caption{}
	\label{diagram5april}
\end{figure}
\noindent
viewed as morphism $(\cF^\str(S),+1)\rightarrow ((V_1,+1),\ldots,(V_k,+1))$ in $\textup{Rib}_\cD$.

Let \(S = (V_1,\dots,V_k)\in\cD^\str\). Besides the coupons
\[
\mathbb{J}_{((V_1,+1),\ldots,(V_k,+1))}^{\textup{Rib}_{\cD}}:=\mathcal{I}_{\cD}(\mathbb{J}_S^{\mathbb{B}_{\cD}}): ((V_1,+1),\ldots,(V_k,+1))\rightarrow 
(\cF^\str(S),+1)
\]
and $\mathbb{I}_{((V_1,+1),\ldots,(V_k,+1))}^{\textup{Rib}_{\cD}}:=\mathcal{I}_{\cD}(\mathbb{I}_S^{\mathbb{B}_{\cD}})$ in $\textup{Rib}_{\cD}$, still depicted by Figures \ref{j double} and \ref{theta double}, we have two additional zipping and unzipping coupons 
\begin{equation*}
\begin{split}
\mathbb{J}_{((V_k,-1),\ldots,(V_1,-1))}^{\textup{Rib}_{\cD}}:\,\,& ((V_k,-1),\ldots,(V_1,-1))\rightarrow 
(\cF^\str(S),-1),\\
\mathbb{I}_{((V_k,-1),\ldots,(V_1,-1)))}^{\textup{Rib}_{\cD}}:\,\,&  (\cF^\str(S),-1)
\rightarrow ((V_k,-1),\ldots,(V_1,-1)),
\end{split}
\end{equation*}
in $\textup{Rib}_{\cD}$,
where the coupon is colored by the appropriate morphism in $\cD^\str$ 
represented by the \(\cD\)-morphism $\textup{id}_{V_k^*\otimes\cdots\otimes V_1^*}^{\cD}$.
As before, we shrink these coupons to a single vertex and depict them by Figures \ref{transpose fusion} and \ref{transpose fusion inverse} respectively.
\begin{figure}[H]
	\begin{minipage}{0.48\textwidth}
		\centering
		\includegraphics[scale = 0.75]{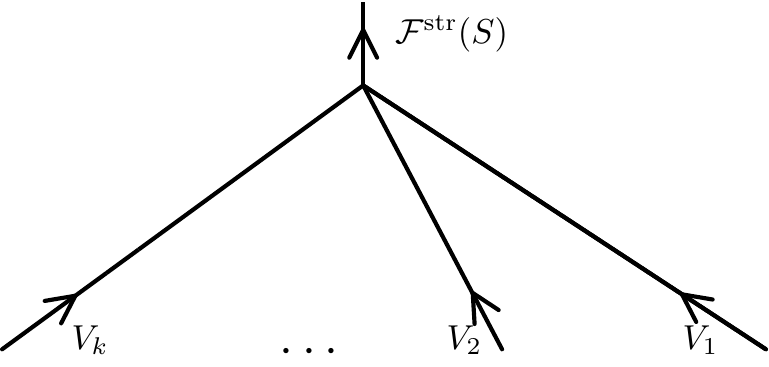}
		\captionof{figure}{}
		\label{transpose fusion}
	\end{minipage}\quad
	\begin{minipage}{0.48\textwidth}
		\centering
		\includegraphics[scale = 0.75]{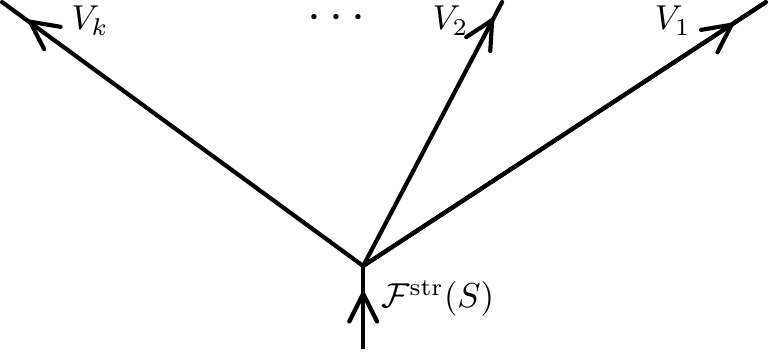}
		\captionof{figure}{}
		\label{transpose fusion inverse}
	\end{minipage}
\end{figure}
Note that 
\[
\mathbb{I}_{((V_1,+1),\ldots,(V_k,+1))}^{\textup{Rib}_{\cD}}\doteq\bigl(\mathbb{J}_{((V_k,-1),\ldots,(V_1,-1))}^{\textup{Rib}_{\cD}}\bigr)^*,\qquad
\mathbb{I}_{((V_k,-1),\ldots,(V_1,-1))}^{\textup{Rib}_{\cD}}\doteq\bigl(\mathbb{J}_{((V_1,+1),\ldots,(V_k,+1))}^{\textup{Rib}_{\cD}}\bigr)^*,
\]
which lead to dot-equalities allowing to (un)zip local maxima and minima in $\cD$-colored ribbon graph diagrams. 
For instance, as a special case of the dot-equality depicted in Figure \ref{diagram12new}, we then have 
\begin{figure}[H]
	\centering
	\includegraphics[scale = 0.75]{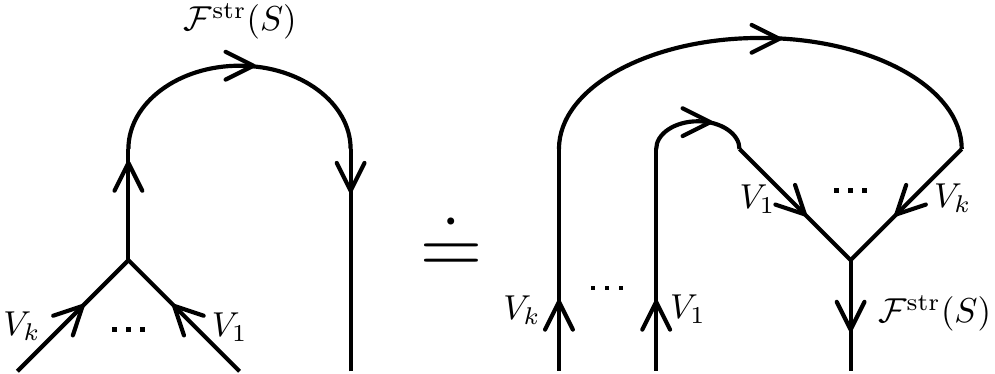}
	\caption{}
	\label{transpose identity}
\end{figure}

%%%%%%%%%%%%%%%%%%%%%%%%
\subsection{Bundling of strands}\label{Section bundling}
%%%%%%%%%%%%%%%%%%%%%%%%%
Let $\cD$ be a braided monoidal category with twist.
In the previous subsection we showed how the coupons colored by fusion morphisms can be used to zip $k$ parallel strands colored by objects $V_1,\ldots,V_k$ in $\mathbb{B}_{\cD}$ into a single strand colored by their tensor product $V_1\otimes\cdots\otimes V_k\in\cD$. It will also be convenient to use a graphical notation for the $k$ parallel strands colored by $V_1,\ldots,V_k$ as a single bundle of strands colored by the $k$-tuple $(V_1,\ldots,V_k)\in\cD^\str$. 

The formal construction and justification of this concept make use of the strict braided monoidal category $\mathbb{B}_{\cD^\str}$ with colors from the strictified braided monoidal category $\cD^\str$.
Its objects are $\emptyset$ and $k$-tuples $(S_1,\ldots,S_k)$ of objects $S_j\in\cD^\str$, while the morphisms
are isotopy classes of $\cD^\str$-colored ribbon-braid graphs. 
Consider the functor
\[
\mathcal{V}_{\cD}: \mathbb{B}_{\cD}\rightarrow \mathbb{B}_{\cD^\str}
\]
defined on objects by
\[
\mathcal{V}_{\cD}(\emptyset):=\emptyset,\qquad\quad \mathcal{V}_{\cD}((V_1,\ldots,V_k)):=((V_1),\ldots,(V_k)),
\]
and on morphisms by mapping a $\cD$-colored ribbon-braid graph diagram $L$ to the $\cD^\str$-colored ribbon-braid graph diagram obtained from $L$ by replacing for each strand its color $V\in\cD$ by $(V)\in\cD^\str$. Clearly $\mathcal{V}_{\cD}$ is a strict braided, twist preserving, tensor functor and an embedding of categories. We will identify objects and morphisms in $\mathbb{B}_{\cD}$ with their $\mathcal{V}_{\cD}$-images in $\mathbb{B}_{\cD^\str}$.

Consider the strict braided, twist preserving tensor functor
\[
\widetilde{\mathcal{F}}_{\cD}^{\textup{br}}:=\mathcal{F}_{\cD^\str}^{\str}\circ\mathcal{F}_{\cD^\str}^{\textup{br}}: \mathbb{B}_{\cD^\str}\rightarrow\cD^\str.
\]
It implements the graphical calculus for the braided monoidal category $\cD$ with twist in terms of $\cD^\str$-colored ribbon-braid graphs, and it is compatible with the graphical calculus introduced in Subsection \ref{GcBraid}, in the sense that
\[
\mathcal{F}_{\cD}^{\textup{br}}=\widetilde{\mathcal{F}}_{\cD}^{\textup{br}}\circ\mathcal{V}_{\cD}.
\]
We can thus view morphisms in $\mathbb{B}_{\cD}$ as morphisms in $\mathbb{B}_{\cD^\str}$ represented by $\cD^\str$-colored ribbon-braid graph diagrams with all its strands colored by objects in $\cD^\str$ of length $1$, and extend the notion of dot-equality to $\cD^\str$-colored ribbon-braid graph diagrams $L,L^\prime$
 by declaring $L\doteq L^\prime$ if and only if $\widetilde{\mathcal{F}}_{\cD}^{\textup{br}}(L)=\widetilde{\mathcal{F}}_{\cD}^{\textup{br}}(L^\prime)$.
 
Let $S=(V_1,\ldots,V_k)$ and $T=(W_1,\ldots,W_\ell)$ be two objects in $\cD^\str$, and denote by $(S)$ and $(T)$ the corresponding objects in $\mathbb{B}_{\cD^\str}$ of length $1$. The elementary morphisms $\textup{id}_{(S)}^{\mathbb{B}_{\cD^\str}}$, $c_{(S),(T)}^{\mathbb{B}_{\cD^\str}}$ and $\theta_{(S)}^{\mathbb{B}_{\cD^\str}}$ are thus represented by the $\cD^\str$-colored ribbon-braid graph diagrams depicted in 
Figures \ref{idS}, \ref{cST} and \ref{thetaS} respectively.
\begin{figure}[H]
\centering
	\begin{minipage}{0.32\textwidth}
		\centering
		\includegraphics[scale = 0.8]{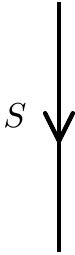}
		\captionof{figure}{}
		\label{idS}
	\end{minipage}
	\begin{minipage}{0.32\textwidth}
		\centering
		\includegraphics[scale = 0.6]{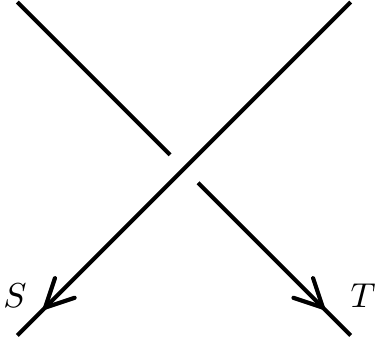}
		\caption{}
		\label{cST}
	\end{minipage}
	\begin{minipage}{0.32\textwidth}
		\centering
		\includegraphics[scale = 0.9]{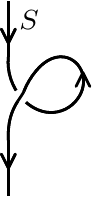}
		\caption{}
		\label{thetaS}
	\end{minipage}
\end{figure}
\noindent
%Note that they only lie in (the image under $\mathcal{V}_{\cD}$ of) $\mathbb{B}_{\cD}$ when $k=1=\ell$. 

In the graphical calculus for $\cD$ these elementary morphisms in $\mathbb{B}_{\cD^\str}$ are dot-equal to the identity, crossing and twist morphisms in $\mathbb{B}_{\cD}$ relative to two bundles of strands colored by $V_1,\ldots,V_k$ and $W_1,\ldots,W_\ell$ respectively, i.e.
\begin{equation}\label{bundlingeq}
\begin{split}
\textup{id}_{(S)}^{\mathbb{B}_{\cD^\str}}&\doteq\textup{id}_{\mathcal{V}_{\cD}(S)}^{\mathbb{B}_{\cD^\str}}=\textup{id}_S^{\mathbb{B}_{\cD}}=
\textup{id}_{V_1}^{\mathbb{B}_{\cD}}\widetilde{\tens}\cdots\widetilde{\tens}\textup{id}_{V_k}^{\mathbb{B}_{\cD}},\\
c_{(S),(T)}^{\mathbb{B}_{\cD^\str}}&\doteq c_{\mathcal{V}_{\cD}(S),\mathcal{V}_{\cD}(T)}^{\mathbb{B}_{\cD^\str}}=
 c_{S,T}^{\mathbb{B}_{\cD}},\\
 \theta_{(S)}^{\mathbb{B}_{\cD^\str}}&\doteq\theta_{\mathcal{V}_{\cD}(S)}^{\mathbb{B}_{\cD^\str}}=\theta_S^{\mathbb{B}_{\cD}}.
\end{split}
\end{equation}
This can be proven using \eqref{algebraicformulazip}, with the role of $\cD$ replaced by its strictified category $\cD^\str$. For instance, the first equality in \eqref{algebraicformulazip} then implies
 \[
  \mathbb{J}_{\mathcal{V}_{\cD}(S)}^{\mathbb{B}_{\cD^\str}}\circ\textup{id}_{S}^{\mathbb{B}_{\cD}}\circ\mathbb{I}_{\mathcal{V}_{\cD}(S)}^{\mathbb{B}_{\cD^\str}}=
 \mathbb{J}_{\mathcal{V}_{\cD}(S)}^{\mathbb{B}_{\cD^\str}}\circ\textup{id}_{\mathcal{V}_{\cD}(S)}^{\mathbb{B}_{\cD^\str}}\circ\mathbb{I}_{\mathcal{V}_{\cD}(S)}^{\mathbb{B}_{\cD^\str}}=\textup{id}_{((V_1)\tens\cdots\tens (V_k))}^{\mathbb{B}_{\cD^\str}}=\textup{id}_{(S)}^{\mathbb{B}_{\cD^\str}}
 \] 
as morphism in $\mathbb{B}_{\cD^\str}$, and hence we obtain 
the first line of \eqref{bundlingeq}:
\begin{figure}[H]
	\centering
	\includegraphics[scale = 0.85]{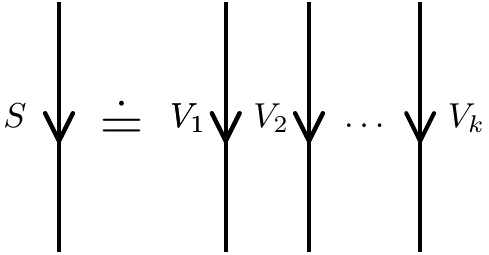}
	\caption{}
	\label{diagramC}
\end{figure}

Finally note that, since $\cD^\str$ is strict, a $\cD^\str$-colored ribbon-braid graph diagram $L$ is dot-equal to the diagram obtained from $L$ by removing strands colored by $\emptyset$. The following proposition formalizes the procedure of interpreting a $\cD^\str$-colored ribbon graph diagram as a $\cD$-colored ribbon graph diagram with its strands replaced by bundles of strands.

%%%%%%%%%%%%%%%%%%%%%%%%%%%%%%%%%%%%%%%%%%%%%%%%%%%%%%
\begin{proposition}\label{propUdef}
There exists a unique strict braided, twist preserving, tensor functor 
\[
\mathcal{U}_{\cD}: \mathbb{B}_{\cD^\str}\rightarrow\mathbb{B}_{\cD}
\]
satisfying
\begin{equation}\label{Uobj}
(S_1,\ldots,S_k)\mapsto S_1\tens\cdots\tens S_k\qquad\qquad (S_j\in\cD^\str)
\end{equation}
for $k>0$ and mapping a coupon $(S_1,\ldots,S_k)\rightarrow (T_1,\ldots,T_\ell)$ in $\mathbb{B}_{\cD^\str}$
colored by
\[
A\in\textup{Hom}_{(\cD^{\str})^{\str}}\bigl((S_1,\ldots,S_k),(T_1,\ldots,T_\ell)\bigr)
\]
to the coupon $S_1\tens\cdots\tens S_k\rightarrow T_1\tens\cdots\tens T_\ell$ in $\mathbb{B}_{\cD}$ colored by $A$, now viewed as morphism in $\textup{Hom}_{\cD^\str}(S_1\tens\cdots\tens S_k,
T_1\tens\cdots\tens T_\ell)$. Furthermore, 
$\mathcal{U}_{\cD}\circ\mathcal{V}_{\cD}=\textup{id}_{\mathbb{B}_{\cD}}$ and 
$\mathcal{F}_{\cD}^{\textup{br}}\circ \mathcal{U}_{\cD}=\widetilde{\mathcal{F}}_{\cD}^{\textup{br}}$.
\end{proposition}
%%%%%%%%%%%%%%%%%%%%%%%%%%%%%%%%%%%%%%%%%%%%%%%%%%%%%
\begin{proof}
The assignment \eqref{Uobj} means that 
one removes from the $k$-tuple $(S_1,\ldots,S_k)$ of tuples $S_j\in\cD^\str$ the $\emptyset$'s and views the remaining sequence as an ordered tuple with entries being objects from $\cD$ (i.e., one forgets the partitioning of the total sequence in sub-sequences). Note that in case all $S_j$ are equal to $\emptyset\in\cD^\str$, formula \eqref{Uobj} assigns the unit object $\emptyset\in\mathbb{B}_{\cD}$ to $(S_1,\ldots,S_k)$. Furthermore, the fact that $\mathcal{U}_{\cD}$ is a strict tensor functor forces $\emptyset\in\mathbb{B}_{\cD^\str}$ to be mapped to $\emptyset\in\mathbb{B}_{\cD}$. Hence on objects, $\mathcal{U}_\cD$ is uniquely determined, and it respects the tensor products of objects.

If the strict braided, twist preserving, functor $\mathcal{U}_{\cD}$ exists, then it is also clearly uniquely determined on morphisms.
Indeed, its action on coupons is already prescribed, and its action on colored crossings and twists is determined by the requirement that $\mathcal{U}_{\cD}$ preserves the braiding and twist. In particular,
\begin{equation}\label{actiononelem}
\mathcal{U}_{\cD}\bigl(\textup{id}_{(S)}^{\mathbb{B}_{\cD^\str}}\bigr)=\textup{id}_S^{\mathbb{B}_{\cD}},\qquad
\mathcal{U}_{\cD}\bigl(c_{(S),(T)}^{\mathbb{B}_{\cD^\str}}\bigr)=c_{S,T}^{\mathbb{B}_{\cD}},\qquad
\mathcal{U}_{\cD}\bigl(\theta_{(S)}^{\mathbb{B}_{\cD^\str}}\bigr)=\theta_S^{\mathbb{B}_{\cD}}
\end{equation}
for $S,T\in\cD^\str$. Existence of $\mathcal{U}_{\cD}$ now follows from the fact that these assignments on morphisms respect
 the local moves capturing isotopies of colored ribbon-braid graph diagrams. 
 It is immediate that $\mathcal{U}_{\cD}\circ\mathcal{V}_{\cD}=\textup{id}_{\mathbb{B}_{\cD}}$, while the fact that \(\mathcal{F}_{\cD}^{\textup{br}}\circ \mathcal{U}_{\cD}\) equals \(\widetilde{\mathcal{F}}_{\cD}^{\textup{br}}\) on colored crossings and twists follows from \eqref{bundlingeq} and \eqref{actiononelem} (the check that they are equal on coupons is trivial).
\end{proof}
%%%%%%%%%%%%%%%%%%%%%%%%%%%%%%%%%%%%%%%
\begin{remark}
Note that $\mathcal{U}_{\cD}$ is not a quasi-inverse of $\mathcal{V}_{\cD}$. Indeed, consider the isomorphisms
\[
\mathbb{I}_{\mathcal{V}_{\cD}(S_1)}^{\mathbb{B}_{\cD^\str}}\widetilde{\tens}\cdots\widetilde{\tens}\,\mathbb{I}_{\mathcal{V}_{\cD}(S_k)}^{\mathbb{B}_{\cD_\str}}:\,\,
(S_1,\ldots,S_k)\mapsto (\mathcal{V}_{\cD}\circ\mathcal{U}_{\cD})(S_1,\ldots,S_k)
\]
for $(S_1,\ldots,S_k)\in\mathbb{B}_{\cD^\str}$ \textup{(}$S_i\in\cD^\str$\textup{)}. Then it follows from the definitions that a colored ribbon-braid graph diagram $L: (S_1,\ldots,S_k)\rightarrow (T_1,\ldots,T_k)$ only satisfies
\[
\bigl(\mathbb{I}_{\mathcal{V}_{\cD}(T_1)}^{\mathbb{B}_{\cD^\str}}\widetilde{\tens}\cdots\widetilde{\tens}\,\mathbb{I}_{\mathcal{V}_{\cD}(T_k)}^{\mathbb{B}_{\cD_\str}}\bigr)\circ L=
(\mathcal{V}_{\cD}\circ\mathcal{U}_{\cD})(L)\circ \bigl(\mathbb{I}_{\mathcal{V}_{\cD}(S_1)}^{\mathbb{B}_{\cD^\str}}\widetilde{\tens}\cdots\widetilde{\tens}\,\mathbb{I}_{\mathcal{V}_{\cD}(S_k)}^{\mathbb{B}_{\cD_\str}}\bigr)
\]
if it doesn't contain any coupons.
\end{remark}
%%%%%%%%%%%%%%%%%%%%%%%%%%%%%%%%%%%%%%
Assume now that $\cD$ is a ribbon category. The 
embedding $\mathcal{V}_{\cD}$ of $\mathbb{B}_{\cD}$ into $\mathbb{B}_{\cD^\str}$ generalizes in the obvious manner to $\textup{Rib}_{\cD}$. It provides a strict ribbon functor
$\mathcal{V}_{\cD}: \textup{Rib}_{\cD}\rightarrow\textup{Rib}_{\cD^\str}$ mapping $((V_1,\delta_1),\ldots,(V_k,\delta_k))$ to $(((V_1),\delta_1),\ldots,((V_k),\delta_k))$. It is an embedding of categories, and we identify objects and morphisms from $\textup{Rib}_{\cD}$ with their $\mathcal{V}_{\cD}$-images in $\textup{Rib}_{\cD^\str}$.

The strict ribbon functor $\mathcal{U}_{\cD}$ also generalizes in the obvious manner to a strict ribbon functor $\mathcal{U}_{\cD}: \textup{Rib}_{\cD^\str}\rightarrow\textup{Rib}_{\cD}$. Note that it maps the object $(S,\delta)$ of length $1$ in $\textup{Rib}_{\cD^\str}$, with \(S = (V_1,\dots,V_k)\in\cD^\str\), to the object $((V_1,\delta),\ldots,(V_k,\delta))$ of length $k$ in $\textup{Rib}_{\cD}$, 
and $\mathcal{U}_{\cD}\circ\mathcal{V}_{\cD}=\textup{id}_{\textup{Rib}_{\cD}}$. The strict tensor functor
\[
\widetilde{\mathcal{F}}_{\cD}^{\textup{RT}}:=\mathcal{F}_{\cD^\str}^\str\circ\mathcal{F}_{\cD^\str}^{\textup{RT}}: \textup{Rib}_{\cD^\str}\rightarrow\cD^\str
\]
is the version of the Reshetikhin-Turaev functor from \cite[Thm. I.2.5]{Turaev-1994}. Then 
\begin{equation}\label{doteqbundle}
\mathcal{F}_{\cD}^{\textup{RT}}\circ\mathcal{U}_{\cD}=\widetilde{\mathcal{F}}_{\cD}^{\textup{RT}},
\end{equation}
which now signifies that a strand in $\textup{Rib}_{\cD^\str}$ with one of its two bases colored by $(S,\delta)$, with \(S = (V_1,\dots, V_k)\),
is dot-equal to the corresponding bundle of $k$ parallel strands, with the corresponding $k$ bases colored by $(V_1,\delta),\ldots,(V_k,\delta)$.
For instance, the evaluation morphism $e_{\widetilde{S}}^{\textup{Rib}_{\cD}}$ and the injection morphism $\iota_{\widetilde{S}}^{\textup{Rib}_{\cD}}$ for
$\widetilde{S}=((V_1,+1),\ldots,(V_k,+1))$ (see Figure \ref{diagram10new}) are dot-equal to 
$e_{(S,+1)}^{\textup{Rib}_{\cD^\str}}$ and $\iota_{(S,+1)}^{\textup{Rib}_{\cD^\str}}$.

The proof of \eqref{doteqbundle} now uses the four zip and unzip morphisms 
\begin{equation*}
\begin{split}
\mathbb{J}_{(((V_1),+1),\ldots,((V_k),+1))}^{\textup{Rib}_{\cD^\str}}:\,\,&
(((V_1),+1),\ldots,((V_k),+1))\rightarrow (S,+1),\\
\mathbb{J}_{(((V_k),-1),\ldots,((V_1),-1))}^{\textup{Rib}_{\cD^\str}}:\,\,& (((V_k),-1),\ldots,((V_1),-1))\rightarrow (S,-1),\\
\mathbb{I}_{(((V_1),+1),\ldots,((V_k),+1))}^{\textup{Rib}_{\cD^\str}}:\,\,&  (S,+1)\rightarrow (((V_1),+1),\ldots,((V_k),+1)),\\
\mathbb{I}_{(((V_k),-1),\ldots,((V_1),-1))}^{\textup{Rib}_{\cD^\str}}:\,\,&  (S,-1)\rightarrow (((V_k),-1),\ldots,((V_1),-1))
\end{split}
\end{equation*}
in $\textup{Rib}_{\cD^\str}$, cf. Subsection \ref{Section 2 strictified}.

As a final step, the bundling of strands can be performed in the mixed graphical calculus associated to a braided monoidal category $\mathcal{D}$ with twist and a full ribbon subcategory $\mathcal{C}$ (see Subsection \ref{mixedsection}). The reason why this works boils down to the fact that the bundling procedures for $\textup{Rib}_{\mathcal{C}}$ and $\mathbb{B}_{\cD}$, as discussed in this subsection,
are compatible with the canonical embeddings $\mathcal{C}\hookrightarrow\mathcal{D}$ and $\mathcal{C}^\str\hookrightarrow\mathcal{D}^\str$ of categories. We leave the details to the reader. The mixed graphical calculus in this paper will only involve bundling of $\cC$-colored strands.

%%%%%%%%%%%%%%%%%
\section{\(q\)-KZ equations for \(k\)-point dynamical fusion operators}
\label{Section Generalized ABRR}
%%%%%%%%%%%%%%%

%%%%%%%%%%%%%%%%
\subsection{$k$-point quantum vertex operators and their graphical representation}
\label{Section 2 C^+}
%%%%%%%%%%%%%%%%%
In this subsection we introduce $k$-point quantum vertex operators, which are built from intertwiners $M_\lambda\rightarrow M_\mu\otimes V$ with $\lambda,\mu\in\mathfrak{h}^*$ and $V\in\Rep$. We first recall the parametrization of such spaces of intertwiners, following \cite{Etingof&Varchenko-1999,Etingof&Varchenko-2000}.

Let $\cN$ be the symmetric monoidal category of $\mathfrak{h}^*$-graded vector spaces. The tensor product of two $\mathfrak{h}^*$-graded vector spaces $M=\bigoplus_{\mu\in\mathfrak{h}^*}M[\mu]$ and $N=\bigoplus_{\mu\in\mathfrak{h}^*}N[\mu]$ is 
\[
M\otimes N:=\bigoplus_{\mu\in\hh^\ast}(M\otimes N)[\mu],\qquad\quad (M\otimes N)[\mu]:=\bigoplus_{\nu\in\mathfrak{h}^*}M[\mu-\nu]\otimes N[\nu].
\]
The unit object of $\cN$ is taken to be $\mathbb{C}=\mathbb{C}[0]$.
The commutativity constraint is $(P_{M,N})_{M,N\in\cN}$. 
Let $\cN_{\textup{adm}}$ be the full subcategory of $\mathfrak{h}^*$-graded vector spaces with finite-dimensional graded components, and let $\cN_{\textup{fd}}$ be the full subcategory of finite-dimensional $\mathfrak{h}^*$-graded vector spaces. Both $\cN_{\textup{adm}}$ and $\cN_{\textup{fd}}$ are symmetric monoidal subcategories of $\cN$. In fact, $\cN_{\textup{fd}}$ is a symmetric tensor category, with as evaluation and injection morphisms the standard ones from the symmetric tensor category of finite-dimensional vector spaces. We will come back to this point in Subsection \ref{sectionqKZ}. 

One can think of $\cN$ as the symmetric monoidal category of semisimple $\mathfrak{h}$-modules by viewing a semisimple $\mathfrak{h}$-module $M$ as the $\mathfrak{h}^*$-graded space with graded components
\[
M[\mu]:=\{m\in M \,\, | \,\, h\cdot m=\mu(h)m\quad \forall\, h\in\mathfrak{h}\}.
\]
This viewpoint provides a forgetful functor 
\[
\cF^{\textup{frgt}}: \cM\rightarrow\cN
\]
mapping $M\in\mathcal{M}$ to its underlying $\mathfrak{h}^*$-graded vector space $\underline{M}:=\bigoplus_{\mu\in\mathfrak{h}^*}M[\mu]$ viewed as semisimple $\mathfrak{h}$-module. $\cF^{\textup{frgt}}$ maps the morphism $\phi\in\textup{Hom}_{U_q}(M,N)$ to $\phi$ viewed as morphism $\underline{M}\rightarrow\underline{N}$ of $\mathfrak{h}^*$-graded vector spaces.  The functor $\cF^{\textup{frgt}}$ restricts to functors
$\cM_{\textup{adm}}\rightarrow\cN_{\textup{adm}}$ and $\Rep\rightarrow\cN_{\textup{fd}}$. Note that $\cF^{\textup{frgt}}$ 
is a strict tensor functor, but it does not respect the braiding (see Subsection \ref{Section expectation value identity} for a further discussion of the properties of $\cF^{\textup{frgt}}$).

Fix $\lambda\in\mathfrak{h}^*$ and $V\in\Rep$. Consider the $\mathfrak{h}^*$-graded vector space 
\[
\textup{Int}_{\lambda,V}:=\bigoplus_{\mu\in\mathfrak{h}^*}\textup{Hom}_{U_q}(M_\lambda,M_{\lambda-\mu}\otimes V).
\]
So the graded component $\textup{Int}_{\lambda,V}[\mu]$ is the morphism space $\textup{Hom}_{U_q}(M_\lambda,M_{\lambda-\mu}\otimes V)$ in the representation category $\cM_{\textup{adm}}$. We call $M_{\lambda}, M_{\lambda-\mu}$ the auxiliary spaces and $V$ the spin space of $\phi\in\textup{Int}_{\lambda,V}[\mu]$.

%%%%%%%%%%%%%%%%%%%
\begin{definition}\label{expmapdef}
The morphism
\[
\langle\cdot\rangle_{\lambda,V}: \textup{Int}_{\lambda,V}\rightarrow \underline{V}
\]
of $\mathfrak{h}^*$-graded vector spaces defined by
\[
\langle (\phi^{(\mu)})_{\mu\in\mathfrak{h}^*}\rangle_{\lambda,V}:=\sum_{\mu\in\mathfrak{h}^*} (\mathbf{m}_{\lambda-\mu}^{\ast}\otimes\id_V)(\phi^{(\mu)}(\mathbf{m}_{\lambda}))
\]
is called the {\it expectation value map} of weight $\lambda$ relative to $V$. 
\end{definition}
%%%%%%%%%%%%%%%%%%%
For $V,W\in\Rep$ and $A\in\textup{Hom}_{U_q}(V,W)$ define the morphism $A^{\textup{spin}}\in\textup{Hom}_{\cN}\bigl(\textup{Int}_{\lambda,V},\textup{Int}_{\lambda,W}\bigr)$ by
\[
A^{\textup{spin}}\bigl((\phi^{(\mu)})_{\mu\in\mathfrak{h}^*}\bigr):=\bigl((\textup{id}_{M_{\lambda-\mu}}\otimes A)\phi^{(\mu)}\bigr)_{\mu\in\mathfrak{h}^*}.
\]
Then $\langle \cdot\rangle_{\lambda,V}$ is functorial in $V\in\Rep$, in the sense that
\begin{equation}\label{intFunct}
\langle A^{\textup{spin}}\phi\rangle_{\lambda,W}=A\bigl(\langle\phi\rangle_{\lambda,V}\bigr)
\end{equation}
for $A\in\textup{Hom}_{U_q}(V,W)$ and $\phi\in\textup{Int}_{\lambda,V}$.

We recall the following well-known fact (see e.g.\ \cite{Etingof&Varchenko-1999,Etingof&Latour-2005} for further details).

%%%%%%%%%%%%%%%%%%%%%%%%%%
\begin{proposition}
	\label{proposition linear isomorphism}
	For \(\lambda\in\hh^\ast_{\textup{reg}}\) and $V\in\Rep$ the expectation value map $\langle \cdot\rangle_{\lambda,V}: \textup{Int}_{\lambda,V}\rightarrow \underline{V}$ is an isomorphism of $\mathfrak{h}^*$-graded vector spaces. In particular, \(\Hom_{U_q}(M_\lambda,M_{\lambda-\mu}\otimes V)=\{0\}\) unless \(\mu\in\textup{wts}(V)\subset\Lambda\), and $\textup{Int}_{\lambda,V}\in\cN_{\textup{fd}}$. 
\end{proposition}
%%%%%%%%%%%%%%%%%%%%%%%%%%

Let $\lambda\in\mathfrak{h}_{\textup{reg}}^*$ and $V\in\Rep$. 
We denote the inverse of the isomorphism $\langle\cdot\rangle_{\lambda,V}$ by 
\begin{equation}\label{inverseEV}
\underline{V}\rightarrow \textup{Int}_{\lambda,V},\qquad v\mapsto \phi_\lambda^v.
\end{equation}
If $v=(v^{(\mu)})_{\mu\in\mathfrak{h}^*}$ is the decomposition of $v\in V$ in $\mathfrak{h}^*$-graded components, then $\phi_\lambda^v=(\phi_\lambda^{v^{(\mu)}})_{\mu\in\mathfrak{h}^*}$ with $\phi_\lambda^{v^{(\mu)}}\in\textup{Int}_{\lambda,V}[\mu]$ the unique $U_q$-intertwiner $M_\lambda\rightarrow M_{\lambda-\mu}\otimes V$ satisfying 
\[
\phi_\lambda^{v^{(\mu)}}(\mathbf{m}_\lambda)=\mathbf{m}_{\lambda-\mu}\otimes v^{(\mu)}+\textup{ l.o.t.}
\]
Here lower order terms refers to terms in $\bigoplus_{\nu<\lambda-\mu}M_{\lambda-\mu}[\nu]\otimes V$ with $\leq$ the dominance order on $\mathfrak{h}^*$. Then
\eqref{intFunct} implies that
\begin{equation}\label{intFunctinverse}
A^\textup{spin}\phi_\lambda^v=\phi_\lambda^{A(v)}
\end{equation}
for $A\in\textup{Hom}_{U_q}(V,W)$ and $v\in V$.

We now consider the case that $V=\cF^\str(S)$ with $S=(V_1,\ldots,V_k)\in\Rep^\str$. 
In this case the above considerations provide for $\lambda\in\mathfrak{h}_{\textup{reg}}^*$ a multilinear map 
\[
V_1\times\cdots\times V_k\rightarrow \textup{Int}_{\lambda,\cF^\str(S)},\qquad  (v_1,\ldots,v_k)\mapsto\phi_\lambda^{v_1\otimes\cdots\otimes v_k}
\]
which descends to an isomorphism
\[
\cF^\str(S)\overset{\sim}{\longrightarrow}\textup{Int}_{\lambda,\cF^\str(S)}
\]
of $\mathfrak{h}^*$-graded vector spaces. Composition of intertwiners provides an alternative way of parametrizing $\textup{Int}_{\lambda,\cF^\str(S)}$. The key construction from \cite{Etingof&Varchenko-2000} is as follows.

%%%%%%%%%%%%%%%%%%%%%%%%%%%%%%%
\begin{corollary}
Let $\lambda\in\mathfrak{h}_{\textup{reg}}^*$ and $S=(V_1,\ldots,V_k)\in\Rep^\str$. 
There exists a unique multilinear map
\begin{equation}\label{mapforward}
V_1\times\cdots\times V_k\rightarrow\textup{Int}_{\lambda,\cF^\str(S)},\qquad (v_1,\ldots,v_k)\mapsto 
\phi_\lambda^{v_1,\ldots,v_k} 
\end{equation}
such that for weight vectors $v_\ell\in V_\ell[\nu_\ell]$, the $k$-point quantum vertex operator $\phi_\lambda^{v_1,\ldots,v_k}\in\textup{Int}_{\lambda,\cF^\str(S)}[{\scriptstyle{\sum}}_\ell\nu_\ell]$ is defined by
\begin{equation}\label{kEV}
\phi_\lambda^{v_1,\ldots,v_k}:=
(\phi_{\lambda_1}^{v_1}\otimes\textup{id}_{V_2\otimes\cdots\otimes V_k})\cdots (\phi_{\lambda_{k-1}}^{v_{k-1}}\otimes\textup{id}_{V_k})\phi_{\lambda_k}^{v_k},
\end{equation}
with
$\lambda_k:=\lambda$ and 
$\lambda_j:=\lambda-\nu_{j+1}-\cdots-\nu_k$ for $0\leq j<k$.
\end{corollary}
%%%%%%%%%%%%%%%%%%%%%%%%%%%%%%%%%%

For \(S=(V_1,\ldots,V_k)\in\Rep^\str\) let us write $\underline{S}:=(\underline{V_1},\ldots,\underline{V_k})\in\cN_{\textup{fd}}^\str$. Note that the multilinear map \eqref{mapforward} descends to a morphism
\begin{equation}\label{fusiondefbasic}
\cF^\str(\underline{S})\rightarrow\textup{Int}_{\lambda,\cF^\str(S)},\qquad v_1\otimes\cdots\otimes v_k\mapsto \phi_\lambda^{v_1,\ldots,v_k}
\end{equation}
of finite-dimensional $\mathfrak{h}^*$-graded vector spaces.

Consider the endomorphism $j_S(\lambda)$ of $\cF^\str(\underline{S})$, defined by
\begin{equation}\label{defjSlambda}
j_S(\lambda)(v_1\otimes\cdots\otimes v_k):=\langle\phi_\lambda^{v_1,\ldots,v_k}\rangle_{\lambda,\cF^\str(S)}
\end{equation}
(we set $j_\emptyset(\lambda):=\textup{id}_{\mathbb{1}}$). Note that $j_{(V)}(\lambda)=\textup{id}_V$ for all $V\in\Rep$, and
we have the identity
\begin{equation}\label{kto1ground}
\phi_\lambda^{v_1,\ldots,v_k}=\phi_\lambda^{j_S(\lambda)(v_1\otimes\cdots\otimes v_k)}
\end{equation}
expressing a $k$-point quantum vertex operator as a $1$-point quantum vertex operator.
We call $j_S(\lambda)$ the {\it $k$-point dynamical fusion operator} of weight $\lambda$ relative to $S$.

The following result can be derived from \cite[\S 2.5]{Etingof&Varchenko-1999}.
%%%%%%%%%%%%%%%%%%%%%%%%%%%%%%%%
\begin{proposition}\label{corEV}
Let $\lambda\in\mathfrak{h}_{\textup{reg}}^*$ and let $S=(V_1,\ldots,V_k)\in\Rep^\str$ be an object of length $k>0$. Then $j_S(\lambda)\in\textup{End}_{\cN}(\cF^\str(\underline{S}))$ is an automorphism.
\end{proposition}
%%%%%%%%%%%%%%%%%%%%%%%%%%%%%%%%%
We will now introduce lifts of the $k$-point quantum vertex operator $\phi_\lambda^{v_1,\ldots,v_k}$ 
(see \eqref{kEV}) 
to $\cM^\str$ and of the $k$-point dynamical fusion operator $j_S(\lambda)$ to $\cN_{\textup{fd}}^\str$. The choice of lift determines
how $\phi_\lambda^{v_1,\ldots,v_k}$ and $j_S(\lambda)$ will be interpreted as coupons in the graphical calculus. 

For $S\in\Rep^\str$ and $\lambda\in\mathfrak{h}^*$, consider the $\mathfrak{h}^*$-graded vector space
\[
\textup{Int}_{\lambda,S}:=\bigoplus_{\mu\in\mathfrak{h}^*}\textup{Hom}_{\cM^\str_{\textup{adm}}}(M_\lambda,M_{\lambda-\mu}\tens S).
\]
Then $\textup{Int}_{\lambda,S}\overset{\sim}{\longrightarrow}\textup{Int}_{\lambda,\cF^\str(S)}$, with the isomorphism 
mapping $\Phi=(\Phi^{(\mu)})_{\mu\in\mathfrak{h}^*}$ to $(\cF^\str_{\cM}(\Phi^{(\mu)}))_{\mu\in\mathfrak{h}^*}$.

For $\lambda\in\mathfrak{h}_{\textup{reg}}^*$, $S=(V_1,\ldots,V_k)\in\Rep^\str$ and $v_\ell\in V_\ell$ ($1\leq \ell\leq k$) we write $\Phi_\lambda^{v_1,\ldots,v_k}\in\textup{Int}_{\lambda,S}$ for the pre-image of $\phi_\lambda^{v_1,\ldots,v_k}\in\textup{Int}_{\lambda,\cF^\str(S)}$ under this isomorphism. 
Note that for weight vectors $v_\ell\in V_\ell[\nu_\ell]$, the morphisms
\[
\Phi_\lambda^{v_1,\ldots,v_k}\in
\textup{Int}_{\lambda,S}[{\scriptstyle{\sum}}_\ell\nu_\ell]=\textup{Hom}_{\cM_{\textup{adm}}^\str}(M_\lambda,M_{\lambda-\sum_\ell\nu_\ell}\tens S)
\]
decompose as
\begin{equation}\label{kEVstrict}
\Phi_\lambda^{v_1,\ldots,v_k}:=
(\Phi_{\lambda_1}^{v_1}\tens\textup{id}_{(V_2,\cdots,V_k)})\cdots (\Phi_{\lambda_{k-1}}^{v_{k-1}}\tens\textup{id}_{V_k})\Phi_{\lambda_k}^{v_k},
\end{equation}
and they span $\textup{Int}_{\lambda,S}$ by Proposition \ref{proposition linear isomorphism}, Proposition \ref{corEV} and \eqref{kto1ground}. We also refer to
$\Phi_\lambda^{v_1,\ldots,v_k}$ as $k$-point quantum vertex operators of weight $\lambda$.

For $\lambda\in\mathfrak{h}^*_{\textup{reg}}$ and $S=(V_1,\ldots,V_k)\in\Rep^\str$, let us denote by 
\[
\ol{J_S}(\lambda)\in\textup{Hom}_{\cN^\str}(\underline{S},\cF^\str(\underline{S}))
\]
the isomorphism representing $j_S(\lambda)\in\textup{End}_{\cN}(\cF^\str(\underline{S}))$, which we will also refer to as
the $k$-point dynamical fusion operator of weight $\lambda$ relative to $S$.
We also view $\ol{J_S}(\lambda)$ as a map $V_1\times\cdots \times V_k\rightarrow V_1\otimes\cdots\otimes V_k$, defined by 
\begin{equation}\label{JSlambda}
\ol{J_S}(\lambda)(v_1,\ldots,v_k):=j_S(\lambda)(v_1\otimes\cdots\otimes v_k).
\end{equation}
The strictified version of the identity \eqref{kto1ground}  then becomes the identity
\begin{equation}\label{kto1}
J_S^{\textup{spin}}\bigl(\Phi_\lambda^{v_1,\ldots,v_k}\bigr)=\Phi_\lambda^{\ol{J_S}(\lambda)(v_1,\ldots,v_k)}
\end{equation}
in $\textup{Hom}_{\cM^\str_{\textup{adm}}}(M_\lambda,M_{\lambda-\sum_\ell\nu_\ell}\tens\cF^\str(S))$. 

Formula \eqref{kto1} generalizes as follows. 
Let $\lambda\in\mathfrak{h}_{\textup{reg}}^*$, $S=(V_1,\ldots,V_k)\in\Rep^\str$ and $k>0$. 
Consider the decomposition
\[
S=S_1\tens\cdots\tens S_\ell,\qquad S_j:=(V_{m_{j-1}+1},V_{m_{j-1}+2},\ldots,V_{m_{j}})
\]
in $\Rep^\str$ with $0=m_0<m_1<\cdots<m_{\ell}=k$. Then 
\[
(J_{S_1}\tens\cdots\tens J_{S_\ell})^{\textup{spin}}\bigl(\Phi_\lambda^{v_1,\ldots,v_k}\bigr)=
\Phi_\lambda^{w_1,\ldots,w_\ell}
\]
with $w_j:=\ol{J_S}(\lambda-\nu_{m_j+1}-\cdots-\nu_k)(v_{m_{j-1}+1},\ldots,v_{m_j})\in \cF^\str(S_j)$. 

Let $\lambda\in\mathfrak{h}_{\textup{reg}}^*$ and $S=(V_1,\ldots,V_k)\in\Rep^\str$. From \eqref{kEV} and \eqref{kto1ground} it follows that 
\begin{equation}\label{dynamical cocycle}
\begin{split}
\ol{J_S}(\lambda)&=\ol{J_{(V_1,\cF^\str(V_2,\ldots,V_k))}}(\lambda)
(\id_{\underline{V_1}}\tens \ol{J_{(V_2,\ldots,V_k)}}(\lambda)),\\
\ol{J_S}(\lambda)&=\ol{J_{(\cF^\str(V_1,\ldots,V_{k-1}),V_k)}}(\lambda)\ol{J_{(V_1,\ldots,V_{k-1})}}(\lambda-\mh_k)
\end{split}
\end{equation}
where $\ol{J_{(V_1,\ldots,V_{k-1})}}(\lambda-\mh_k)\in\textup{Hom}_{\cN^\str}\bigl(\underline{S},(\cF^\str(\underline{V_1},\ldots,\underline{V_{k-1}}),\underline{V_k})\bigr)$ is representing the endomorphism $j_{(V_1,\ldots,V_{k-1})}(\lambda-\mh_k)\in\textup{End}_{\cN}(\cF^\str(\underline{S}))$ defined by
\[
j_{(V_1,\ldots,V_{k-1})}(\lambda-\mh_k)(v_1\otimes\ldots\otimes v_k):=j_{(V_1,\ldots,V_{k-1})}(\lambda-\nu_k)(v_1\otimes\ldots\otimes v_{k-1})\otimes v_k
\]
for $v_\ell\in V_\ell$ ($1\leq \ell<k$) and $v_k\in V_k[\nu_k]$ (cf. e.g.\ \cite[(2.11)]{Etingof&Varchenko-2000}). 
It follows from \eqref{dynamical cocycle} that the $k$-point dynamical fusion operator can be written as composition of $2$-point dynamical fusion operators. Combining both formulas in \eqref{dynamical cocycle} for $k=3$ leads to the $2$-cocycle condition for the $2$-point dynamical fusion operator,
\[
\ol{J_{(V_1,V_2\otimes V_3)}}(\lambda)
(\id_{\underline{V_1}}\tens \ol{J_{(V_2,V_3)}}(\lambda))=
\ol{J_{(V_1\otimes V_{2},V_3)}}(\lambda)\ol{J_{(V_1,V_2)}}(\lambda-\mh_3),
\]
reflecting the fact that the order in which the neighboring $\Phi_{\lambda_i}^{v_i}$ in \eqref{kEVstrict} are fused pairwise, is irrelevant.

Dynamical $2$-point fusion operators were introduced and studied in e.g.\ \cite{Arnaudon&Buffenoir&Ragoucy&Roche-1998, Etingof&Varchenko-1999}. 
They arise as the action of a universal fusion element in a suitable completion of $U_q^{\otimes 2}$. The latter fact follows from a linear equation for the $2$-point dynamical fusion operator which is due to Arnaudon, Buffenoir, Ragoucy and Roche \cite{Arnaudon&Buffenoir&Ragoucy&Roche-1998}. We re-derive it graphically in Subsection \ref{Section expectation value identity}. 

We will now introduce in $\mathbb{B}_{\cM}$ graphical notations for coupons colored by quantum vertex operators. Since ribbon-braid graph diagrams involving quantum vertex operators tend to become quite long in the vertical direction, we will rotate them counterclockwise over 90 degrees. Vertical strands pointing downwards thus become horizontal strands oriented from left to right.

In the graphical notations for quantum vertex operators, the coupon colored with 
the $0$-point quantum vertex operator 
$\textup{id}_{(M_\lambda)}\in\textup{Int}_{\lambda,\emptyset}$ will be denoted by 
Figure \ref{identity on Verma}. This leads to Figure \ref{general intertwiner} as the coupon in $\mathbb{B}_{\cM^\str}$ colored by \(\Phi\in\textup{Int}_{\lambda,S}[\lambda-\mu]=\Hom_{\cM^{\str}_{\textup{adm}}}(M_\lambda,M_\mu\tens S)\) with \(S=(V_1,\ldots,V_k)\in\Rep^\str\),
\begin{figure}[H]
	\begin{minipage}{0.24\textwidth}
		\centering
		\includegraphics[scale = 0.9]{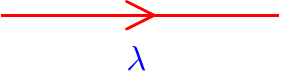}
		\captionof{figure}{}
		\label{identity on Verma}
	\end{minipage}
	\qquad\qquad\quad
	\begin{minipage}{0.24\textwidth}
		\centering
		\includegraphics[scale=0.9]{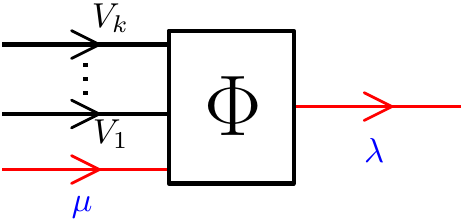}
		\captionof{figure}{  }
		\label{general intertwiner}
	\end{minipage}
\end{figure}
\noindent
while the coupon in $\mathbb{B}_{\cM^\str}$ colored by \(\Phi\) is
depicted by Figure \ref{general intertwinernew}.
\begin{figure}[H]
	\centering
	\includegraphics[scale = 0.85]{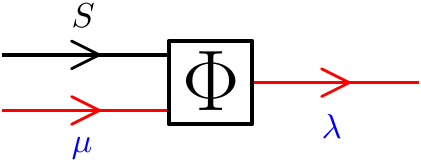}
	\caption{}
	\label{general intertwinernew}
\end{figure}

For a $1$-point quantum vertex operator $\Phi_\lambda^v\in\textup{Hom}_{\cM_{\textup{adm}}^\str}(M_\lambda,M_\mu\tens V)$ with expectation value $v\in V[\lambda-\mu]$
 we denote the coupon colored by $\Phi_\lambda^v$ as in Figure \ref{general intertwiner 2[new]}.

\begin{figure}[H]
	\centering
	\includegraphics[scale = 0.85]{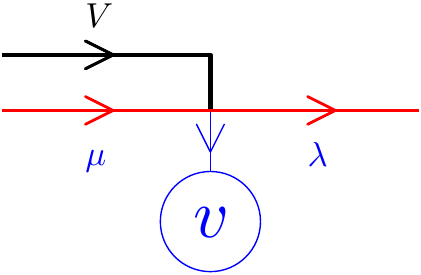}
	\caption{}
	\label{general intertwiner 2[new]}
\end{figure}

By \eqref{kEVstrict} the coupon in $\mathbb{B}_{\cM}$ colored by the $k$-point quantum vertex operator $\Phi_\lambda^{v_1,\ldots,v_k}$ ($v_\ell\in V_\ell[\nu_\ell]$)
equals the ribbon-braid graph diagram depicted by Figure \ref{general intertwiner2}.
\begin{figure}[h]
	\centering
	\includegraphics[scale = 0.85]{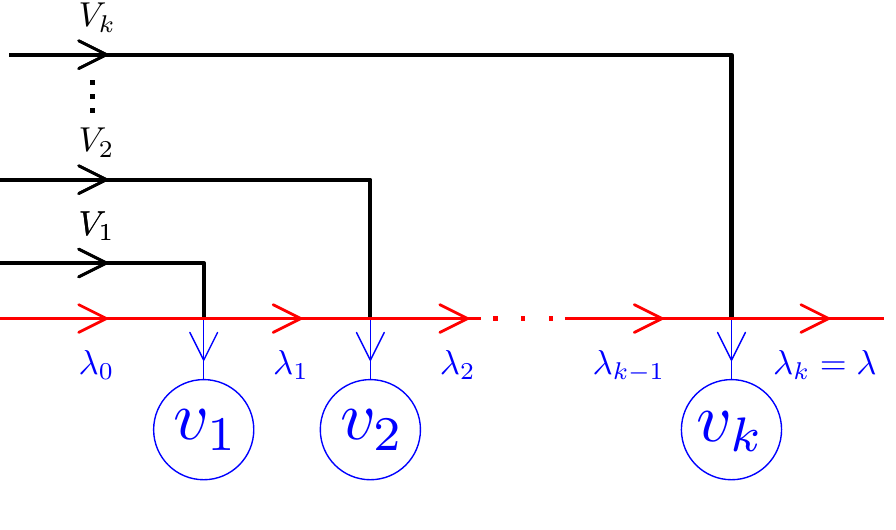}
	\caption{}
	\label{general intertwiner2}
\end{figure}

%%%%%%%%%%%%%%%%%%%%%%%%%%%%%%
\subsection{The topological operator \(q\)-KZ equations}\label{sectionqKZ}
%%%%%%%%%%%%%%%%%%%%%%%%%%%%%

The topological operator \(q\)-KZ equations are consistency equations for $k$-point quantum vertex operators, which we derive using the graphical calculus for the strict braided monoidal category $\cM^{\textup{str}}$ (Theorem \ref{theorem RT braided}). 
%%%%%%%%%%%%%%%%%%%%%
\begin{proposition}
	\label{prop asymptotic operator qKZ}
	For any \(S\in\Rep^\str\), \(\lambda,\mu\in\hh^\ast_{\textup{reg}}\) and \(\Phi\in\Hom_{\cM_{\mr{adm}}^\str}(M_\lambda, M_\mu\tens S)\) we have 
	\begin{figure}[H]
		\centering
		\includegraphics[scale = 0.75]{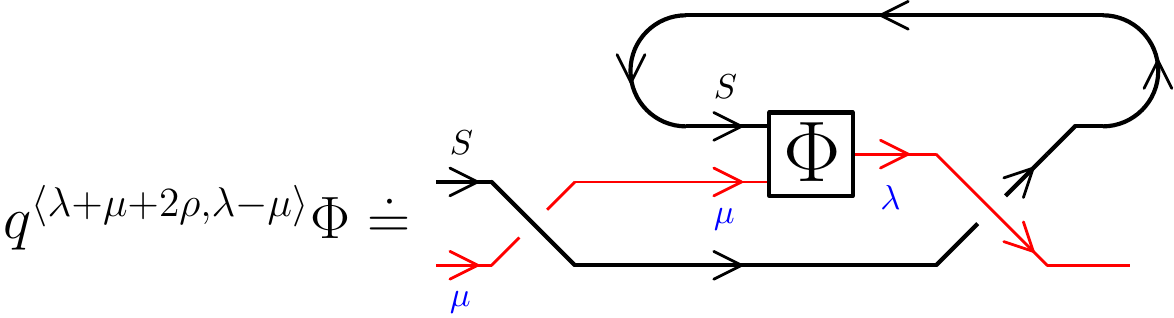}
		\caption{}
		\label{asymptotic qKZ}
	\end{figure}
\end{proposition}
%%%%%%%%%%%%%%%%%%%
\begin{proof}
Recall the twist $(\vartheta_M)_{M\in\cM}$ of $\cM$, defined in terms of the ribbon element $\vartheta$ of $U_q$ (see Subsection \ref{Section ribbon element}).
By Propositions \ref{Drinfeldtheta} and \ref{constanttheta} we have
\begin{center}
	\includegraphics[scale = 0.75]{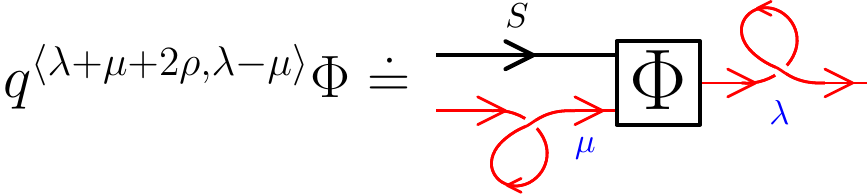}
\end{center}
By Figure \ref{diagram1april}, the right-hand side of the identity above, viewed as element in $\mathbb{B}_{\cM^\str}$, equals 
\begin{center}
	\includegraphics[scale = 0.75]{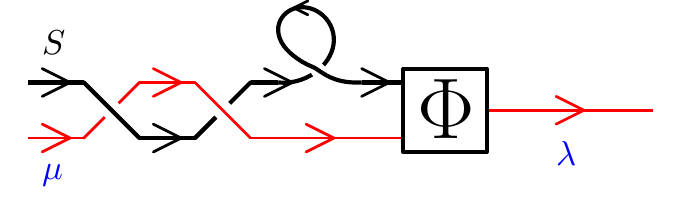}
\end{center}
The statement now follows from the identity in $\mathbb{B}_{\cM^\str}$ depicted by Figure \ref{asymptotic qKZ id}, which follows from the elementary move of pushing the coupon colored by $\Phi$ over the $S$-colored strand.
\begin{figure}[H]
		\centering
		\includegraphics[scale = 0.75]{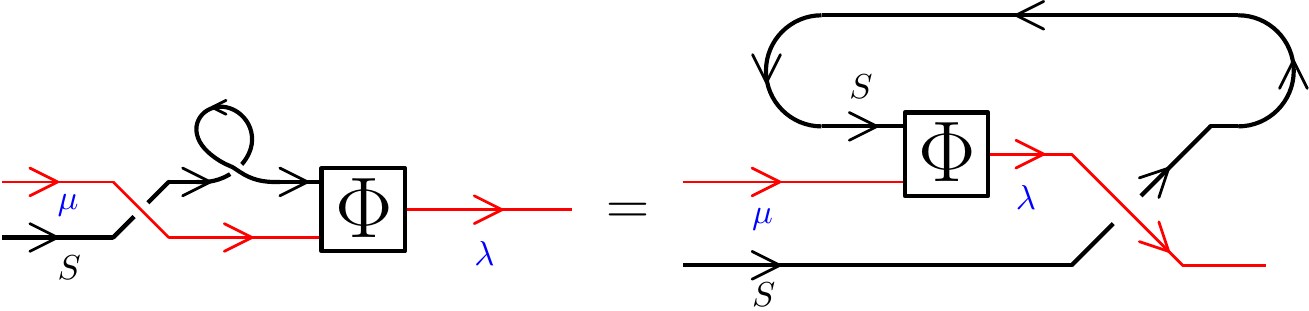}
		\caption{}
		\label{asymptotic qKZ id}
\end{figure}
\end{proof}
%%%%%%%%%%%%%%%%
As a corollary we have for $k=1$ the following topological analog of the operator \(q\)-KZ equation from \cite[Thm. 5.1]{Frenkel&Reshetikhin-1992}. Recall the partial quantum trace defined in Definition \ref{pqt}. 
%%%%%%%%%%%%%%%%
\begin{corollary}[Topological operator \(q\)-KZ equation]
Let $\lambda\in\mathfrak{h}_{\textup{reg}}^*$, $V\in\Rep$ and $v\in V[\nu]$. Then
\[
q^{\langle 2(\lambda+\rho)-\nu,\nu\rangle}\phi_\lambda^v=c_{V,M_{\lambda-\nu}}\,\textup{qTr}_V^{M_\lambda,V\otimes M_{\lambda-\nu}}\Bigl((\textup{id}_V\otimes\phi_\lambda^v)c_{M_\lambda,V}\Bigr)
\]
in $\textup{Hom}_{U_q}(M_\lambda,M_{\lambda-\nu}\otimes V)$.
\end{corollary}
%%%%%%%%%%%%%%%%%%
The following proposition will be instrumental in deriving the topological operator \(q\)-KZ equations for arbitrary $k$-point quantum vertex operators.

%%%%%%%%%%%%%%%%%%%%%%%%%%%%%%%
\begin{proposition}
	\label{prop bulk ABRR}
	For \(\lambda_i\in\hh^\ast_{\textup{reg}}\), \(S_i\in\Rep^\str\) and \(\Phi_i\in\Hom_{\cM_{\mr{adm}}^\str}(M_{\lambda_i},M_{\lambda_{i-1}}\tens S_i)\), let us write \(S:= S_1\tens S_2\tens S_3\) and set
	\begin{equation}
	\label{Big Phi def}
	\Phi:=(\Phi_1\tens\id_{S_2\tens S_3})(\Phi_2\tens\id_{S_3})\Phi_3\in\Hom_{\cM_{\mr{adm}}^\str}(M_{\lambda_3},M_{\lambda_0}\tens S).
	\end{equation}
	Then we have 
	\begin{figure}[H]
		\centering
		\includegraphics[scale = 0.75]{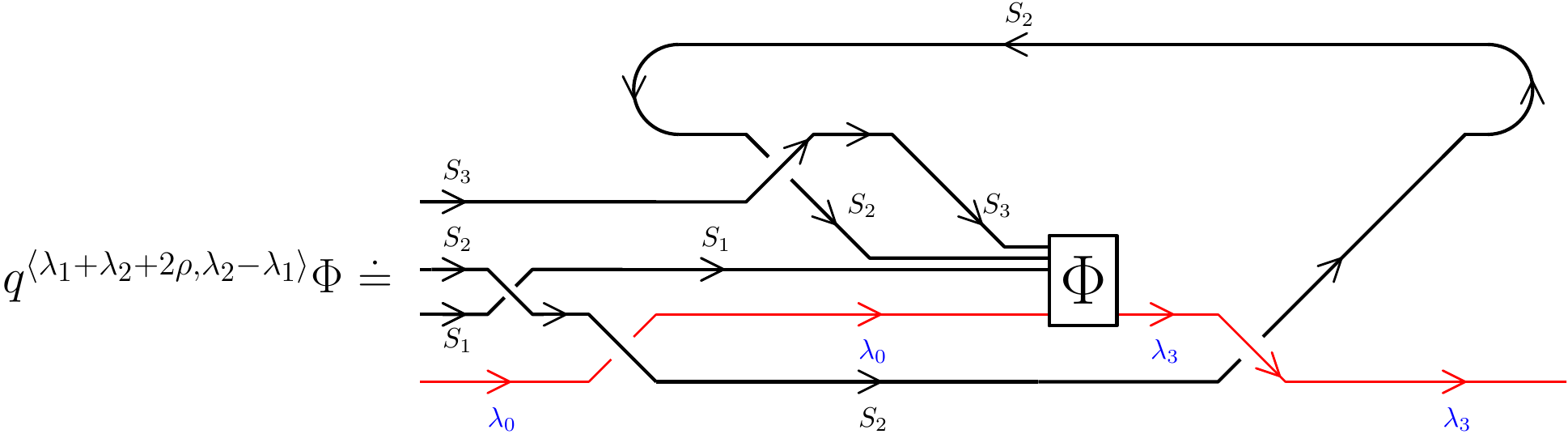}
		\caption{}
		\label{bulk}
	\end{figure}
\noindent in $\textup{Hom}_{\cM_{\mr{adm}}^{\textup{str}}}(M_{\lambda_3},M_{\lambda_0}\tens S)$.
\end{proposition}
%%%%%%%%%%%%%%%%%%%%%%%%%%%%%%%
\begin{proof}
Assume first that all objects $S_j$ are of length $>0$.
By the definition of \(\Phi\), it suffices to show that
	\begin{figure}[H]
		\centering
		\includegraphics[scale = 0.75]{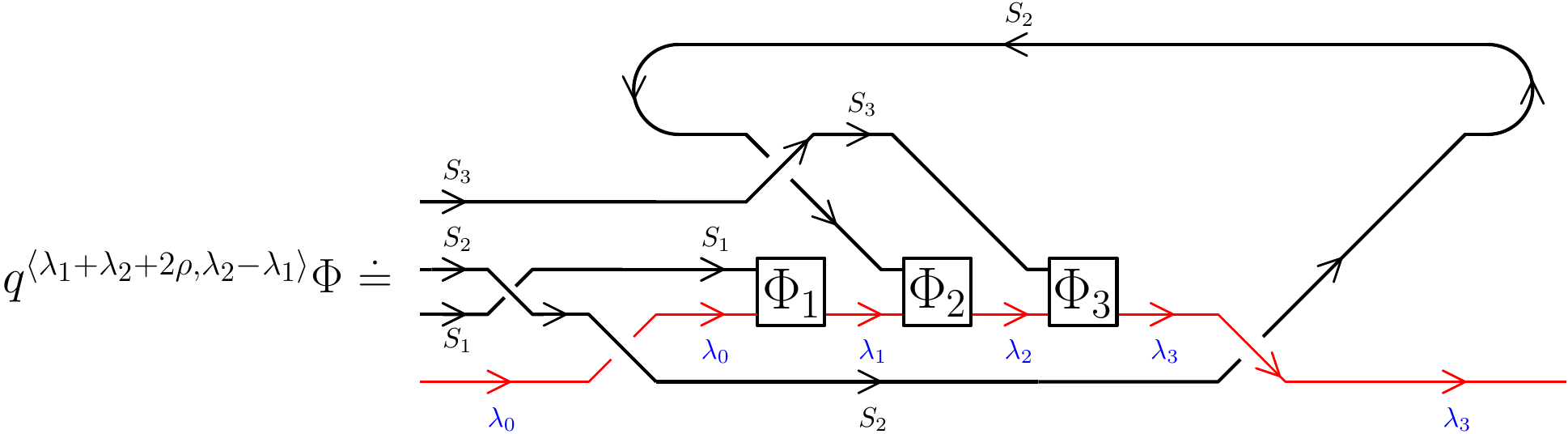}
		\caption{}
		\label{bulk A}
	\end{figure}
\noindent
By (\ref{Big Phi def}) and Proposition \ref{prop asymptotic operator qKZ} the left-hand side of Figure \ref{bulk A} is dot-equal to
	\begin{center}
	\includegraphics[scale = 0.75]{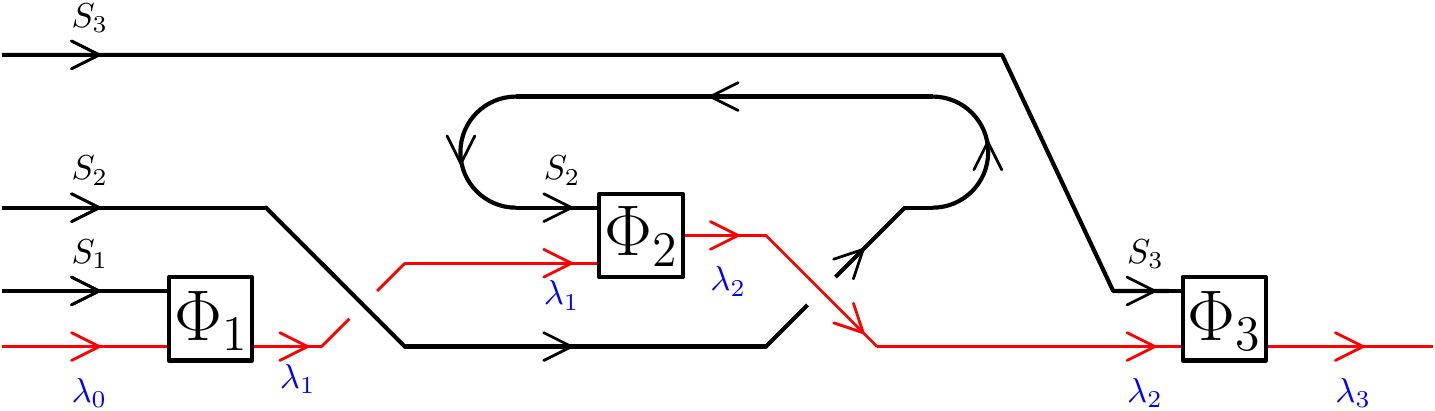}
	\end{center}
	By the second Reidemeister move (see Figure \ref{universal R-matrix double crossing}) on the strands colored by $S_2$ and $S_3$, this equals
	\begin{center}
	\includegraphics[scale = 0.75]{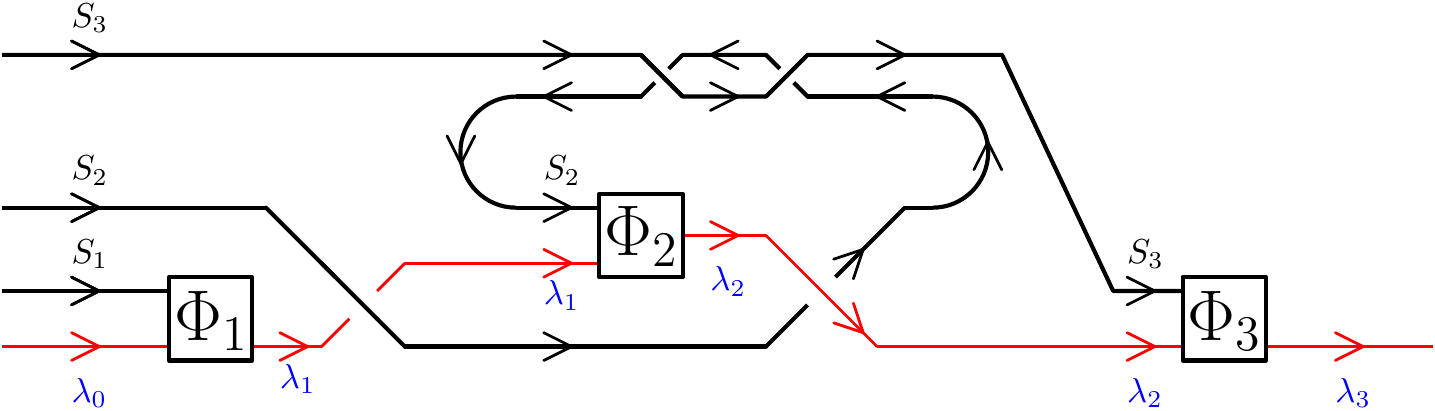}
	\end{center}
	\noindent
	in $\mathbb{B}_{\cM^\str}$.
	Pulling the cups and caps colored by $S_2$ through the crossing, this equals
	\begin{center}
	\includegraphics[scale = 0.75]{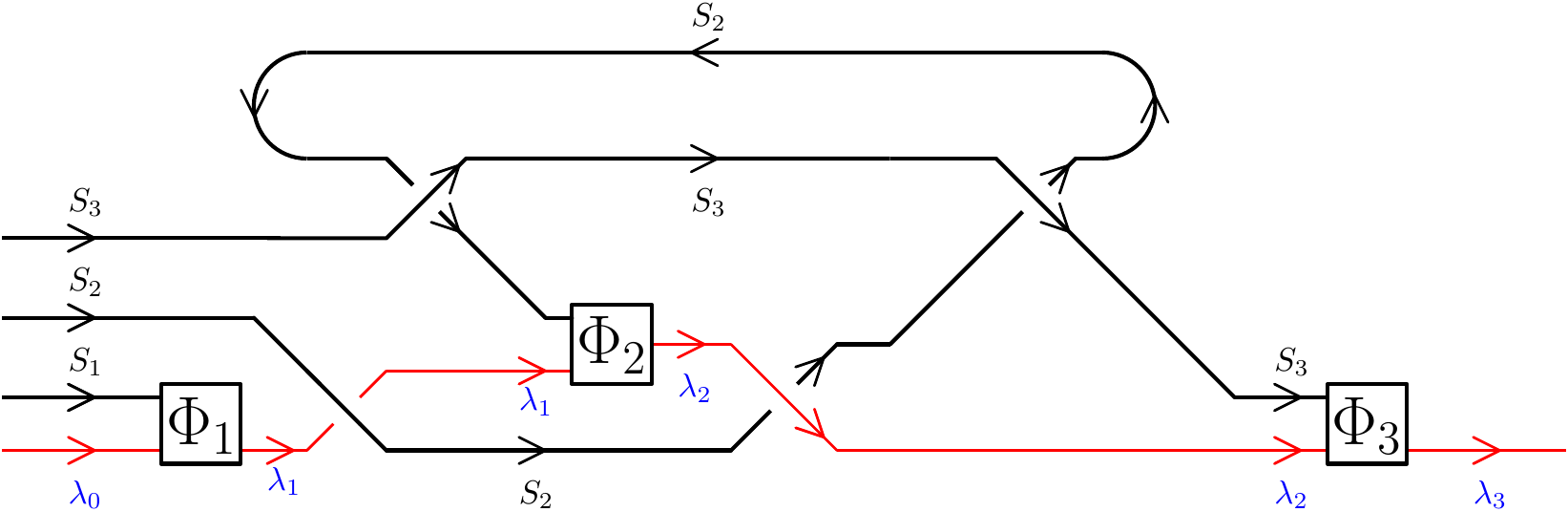}
	\end{center}	
	\noindent
	in $\mathbb{B}_{\cM^\str}$.
	Finally, upon pulling the coupons colored by $\Phi_1$ and $\Phi_3$ underneath respectively over the $S_2$-colored strand, this equals 
	the right-hand side of Figure \ref{bulk A} in $\mathbb{B}_{\cM^\str}$. When $S_1=\emptyset$ and/or $S_3=\emptyset$ and $S_2$ is of length $>0$, the proposition follows from a  straightforward adjustment of the above proof. Finally, when $S_2=\emptyset$ the proposition is trivial.    
\end{proof}
%%%%%%%%%%%%%%%%%%%%%%%%%%%%%%%%%%%%%%%%%%%%%%%%%%%%%%%%%
Now let us precompose the identity in Proposition \ref{prop bulk ABRR} with $c_{S_2,S_1}^{-1}$. On the right-hand side of Figure \ref{bulk}, this amounts to pulling the $S_2$-colored strand over the $S_1$-colored strand on the far left side of the diagram. This leads to the dot-equality
\begin{figure}[H]
	\centering
	\includegraphics[scale = 0.7]{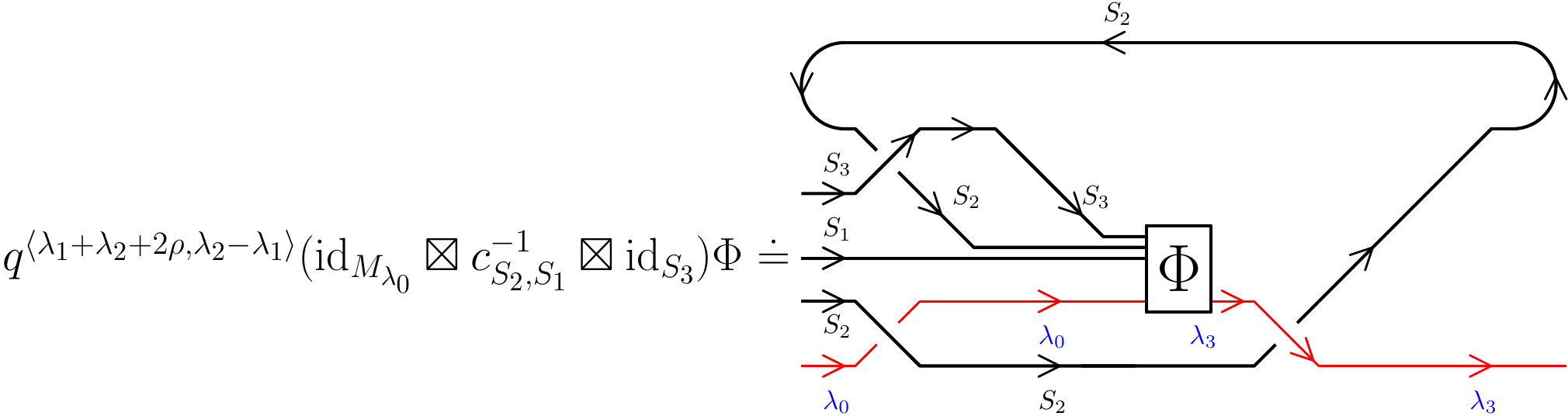}
	\caption{}
	\label{altform}
\end{figure}
\noindent in $\textup{Hom}_{\cM_{\mr{adm}}^\str}(M_{\lambda_3},M_{\lambda_0}\tens S_2\tens S_1\tens S_3)$. We will now use this identity as the starting point for deriving the topological operator $q$-KZ equations for quantum vertex operators. We formulate and prove these as identities in $\cM_{\mr{adm}}^\str$, to emphasize their topological nature.

From now on we omit identity morphisms in tensor products of morphisms if it does not cause confusion. For $S,T\in\cM_{\mr{adm}}^\str$ and
$V\in\cM_{\textup{fd}}$ we write 
\[
\textup{qTr}^{S,T}_V: \textup{Hom}_{\cM_{\mr{adm}}^\str}(S\tens V,T\tens V)\rightarrow\textup{Hom}_{\cM_{\mr{adm}}^\str}(S,T)
\]
for the unique linear map that sends $\Psi\in \textup{Hom}_{\cM_{\mr{adm}}^\str}(S\tens V,T\tens V)$ to the morphism in $\textup{Hom}_{\cM_{\mr{adm}}^\str}(S,T)$ representing
$\textup{qTr}_V^{\cF^\str(S),\cF^\str(T)}(\Psi)\in\textup{Hom}_{U_q}(\cF^\str(S),\cF^\str(T))$.
%%%%%%%%%%%%%%%%%%%%%%%%%%%%%%%%%%%%%%%%%%%%%%%%%%%%%%%%
\begin{theorem}[Topological operator $q$-KZ equations for quantum vertex operators]\label{operatorqKZthm}\hfill

\noindent
Let $\lambda\in\mathfrak{h}_{\textup{reg}}^*$, $S=(V_1,\ldots,V_k)\in\Rep^{\textup{str}}$ and $v_i\in V_i[\nu_i]$. Set $\lambda_j:=\lambda-\nu_{j+1}-\cdots-\nu_k$ for $0\leq j<k$, with the convention that $\lambda_k=\lambda$. 

The $k$-point quantum vertex operator 
$\Phi_\lambda^{v_1,\ldots,v_k}\in\textup{Hom}_{\cM_{\textup{adm}}^\str}(M_\lambda,M_{\lambda_0}\tens S)$ satisfies the linear equations
\begin{equation}\label{operqKZi}
\begin{split}
&q^{\langle\lambda_i+\lambda_{i-1}+2\rho,\lambda_i-\lambda_{i-1}\rangle}
c_{V_i,V_1}^{-1}\cdots c_{V_i,V_{i-1}}^{-1}\Phi_\lambda^{v_1,\ldots,v_k}\\
=\ &c_{V_i,M_{\lambda_0}}\,\textup{qTr}^{M_\lambda,V_i\tens M_{\lambda_0}\tens S^{(i)}}_{V_i}
\Bigl(
\bigl(\textup{id}_{V_i}\tens (c_{V_k,V_i}^{-1}\cdots c_{V_{i+1},V_i}^{-1}\Phi_\lambda^{v_1,\ldots,v_k})\bigr)c_{M_{\lambda},V_i}\Bigr)
\end{split}
\end{equation}
in $\textup{Hom}_{\cM_{\textup{adm}}^\str}\bigl(M_{\lambda},M_{\lambda_0}\tens V_i\tens S^{(i)}\bigr)$ for $i=1,\ldots,k$,
where
$S^{(i)}:=(V_1,\ldots,V_{i-1},V_{i+1},\ldots,V_k)$. Here $c_{V_i,V_1}^{-1}\cdots c_{V_i,V_{i-1}}^{-1}$ \textup{(}resp. $c_{V_k,V_i}^{-1}\cdots c_{V_{i+1},V_i}^{-1}$\textup{)} should be read as the identity when $i=1$ \textup{(}resp. $i=k$\textup{)}.
\end{theorem}
%%%%%%%%%%%%%%%%%%%%%%%
\begin{proof}
For fixed $1\leq i\leq k$ we now apply Figure \ref{altform} to the quantum vertex operators
\begin{equation}\label{threeint}
\begin{split}
\Phi_1&:=\Phi_{\lambda_{i-1}}^{v_1,\ldots,v_{i-1}}\in\textup{Hom}_{\cM_{\mr{adm}}^{\textup{str}}}(M_{\lambda_{i-1}},(M_{\lambda_0},V_1,\ldots,V_{i-1})),\\
\Phi_2&:=\Phi_{\lambda_i}^{v_i}\in\textup{Hom}_{\cM_{\mr{adm}}^{\textup{str}}}(M_{\lambda_i},(M_{\lambda_{i-1}},V_i)),\\
\Phi_3&:=\Phi_{\lambda}^{v_{i+1},\ldots,v_k}\in\textup{Hom}_{\cM_{\mr{adm}}^{\textup{str}}}(M_{\lambda},(M_{\lambda_i},V_{i+1},\ldots,V_k)),
\end{split}
\end{equation}
with the convention that $\Phi_1=\textup{id}_{M_{\lambda_0}}$ (resp. $\Phi_3=\textup{id}_{M_\lambda}$) when $i=1$ (resp. $i=k$). The corresponding decomposition $S=S_1\tens S_2\tens S_3$ of $S$ is
\[
S_1=(V_1,\ldots,V_{i-1}),\qquad S_2=(V_i),\qquad S_3=(V_{i+1},\ldots,V_k)
\]
with $S_1=\emptyset$ (resp. $S_3=\emptyset$) when $i=1$ (resp. $i=k$). Note that the resulting morphism 
\[
\Phi:=(\Phi_1\tens\id_{S_2\tens S_3})(\Phi_2\tens\id_{S_3})\Phi_3
\]
is the $k$-point quantum vertex operator $\Phi_\lambda^{v_1,\ldots,v_k}$, in view of \eqref{kEVstrict}. Finally, the set of shifted weights $(\lambda_0,\lambda_1,\lambda_2,\lambda_3)$ in Figure \ref{altform} is $(\lambda_0,\lambda_{i-1},\lambda_i,\lambda)$. The right-hand side of Figure \ref{altform} then becomes the $\cM^\str$-colored ribbon-braid graph diagram depicted by 
\begin{center}
	\includegraphics[scale = 0.67]{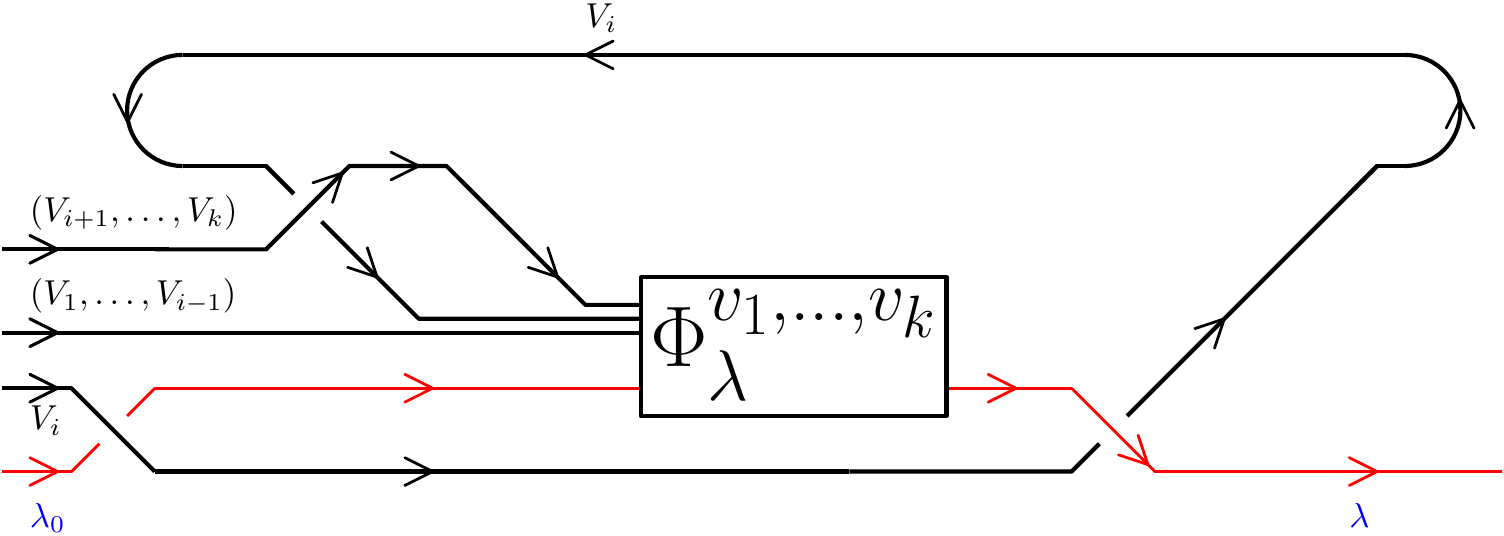}
\end{center}
and Figure \ref{altform} then reduces to the identity
\begin{equation*}
\begin{split}
&q^{\langle\lambda_i+\lambda_{i-1}+2\rho,\lambda_i-\lambda_{i-1}\rangle}
c_{V_i,(V_1,\ldots,V_{i-1})}^{-1}\Phi_\lambda^{v_1,\ldots,v_k}\\
=\ &c_{V_i,M_{\lambda_0}}\textup{qTr}_{V_i}^{M_\lambda,V_i\tens M_{\lambda_0}\tens S^{(i)}}\Bigl(
\bigl(\textup{id}_{V_i}\tens(c_{(V_{i+1},\ldots,V_k),V_i}^{-1}\Phi_\lambda^{v_1,\ldots,v_k})\bigr)c_{M_\lambda,V_i}\Bigr)
\end{split}
\end{equation*}
in $\textup{Hom}_{\cM_{\mr{adm}}^{\textup{str}}}\bigl(M_\lambda,M_{\lambda_0}\tens V_i\tens S^{(i)}\bigr)$. 
This implies
\eqref{operqKZi} since 
\[
c_{V_i,(V_1,\ldots,V_{i-1})}=c_{V_i,V_{i-1}}c_{V_i,V_{i-2}}\cdots c_{V_i,V_1},\qquad
c_{(V_{i+1},\ldots,V_k),V_i}=c_{V_{i+1},V_i}c_{V_{i+2},V_i}\cdots c_{V_k,V_i}
\]
by the hexagon identities \eqref{hi}.
\end{proof}
%%%%%%%%%%%%%%%%%%%%%%%%%%%%%%%%%%%%%%%%%%%%%
The equations \eqref{operqKZi} are called the {\it topological operator $q$-KZ equations} for $k$-point quantum vertex operators. Operator $q$-KZ equations were derived algebraically
in \cite[Thm. 5.2]{Frenkel&Reshetikhin-1992} in the context of $q$-analogs of WZW conformal blocks. The semiclassical limit of Theorem \ref{operatorqKZthm} is discussed in \cite[Cor. 6.2]{Reshetikhin&Stokman-2020-A} and \cite[\S 2.1]{Stokman-2021}.
 
%%%%%%%%%%%%%%%%%%%%%%%%%%%%%%%%%
\subsection{The topological $q$-KZ equations for the $k$-point dynamical fusion operator}
\label{Section expectation value identity}
%%%%%%%%%%%%%%%%%%%%%%%%%%%%%%%%
Taking the highest weight to highest weight component in the topological operator $q$-KZ equations \eqref{operqKZi} leads to topological $q$-KZ equations for the $k$-point dynamical fusion operator.
We derive these using the graphical calculus for the symmetric monoidal category $\cN$ of $\mathfrak{h}^*$-graded vector spaces and its symmetric tensor subcategory $\cN_{\textup{fd}}$ of finite-dimensional $\mathfrak{h}^*$-graded vector spaces.

Recall that the braiding $(P_{M,N})_{M,N\in\cN}$ of the symmetric monoidal category $\cN$ consists of the flip operators $P_{M,N}\in\textup{Hom}_{\cN}(M\otimes N,N\otimes M)$. The braiding $(P_{S,T})_{S,T\in\cN^\str}$ of the symmetric monoidal category $\cN^\str$ thus consists of the isomorphisms $P_{S,T}\in\textup{Hom}_{\cN^\str}(S\boxtimes T,T\boxtimes S)$ representing the flip operators $P_{\cF^\str_\cN(S),\cF^\str_\cN(T)}$.
We will depict the coupon in $\mathbb{B}_{\cN^\str}$ colored by the braiding $P_{S,T}$ by Figure \ref{transposition Vect}.  
\begin{figure}[H]
		\centering
		\includegraphics[scale = 0.75]{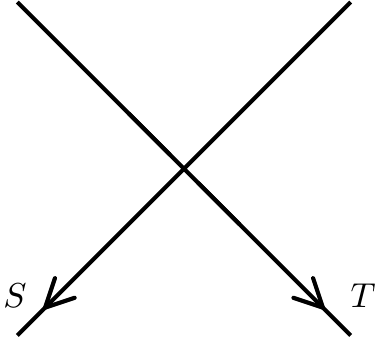}
		\captionof{figure}{Coupon in $\mathbb{B}_{\cN^\str}$ colored by $P_{S,T}$}
		\label{transposition Vect}
\end{figure}
\noindent Note that it is dot-equal to both $c_{S,T}^{\mathbb{B}_{\cN^\str}}$ and $(c_{T,S}^{\mathbb{B}_{\cN^\str}})^{-1}$, since
$\cN^\str$ is symmetric.

The dual $V^*$ of $V\in\cN_{\textup{fd}}$ is the linear dual of $V$, with $\mu$-graded component $V^*[\mu]$ the linear functionals on $V$ vanishing at $V[\nu]$ ($\nu\not=-\mu$). 
The evaluation and injection morphisms $e_V: V^*\otimes V\rightarrow\mathbb{C}$ and $\iota_V: \mathbb{C}\rightarrow V\otimes V^*$ 
of $\cN_{\textup{fd}}$ are the standard ones, see \eqref{evaluation morphism} and \eqref{injection morphism}. The forgetful functor $\cF^{\textup{frgt}}: \Rep\rightarrow\cN_{\textup{fd}}$ thus preserves left duality.

The twist $\theta=(\theta_V)_{V\in\cN_{\textup{fd}}}$ of $\cN_{\textup{fd}}$ is trivial: $\theta_V=\textup{id}_V$ for all $V\in\cN_{\textup{fd}}$. As a consequence, for an $\cN^{\textup{str}}_{\textup{fd}}$-colored ribbon tangle graph $D$, the associated morphism $\mathcal{F}_{\cN^{\textup{str}}_{\textup{fd}}}^{\textup{RT}}(D)$
only depends on the isotopy class of $D$ as $\cN^\str_{\textup{fd}}$-colored tangle graph.
The evaluation and injection morphisms associated to the right duality of the ribbon category $\cN_{\textup{fd}}$ are 
\begin{align}
\label{(co-) evaluation with hats}
\begin{split}
\widehat{e}_V:&\ V\otimes V^{\ast} \to \CC: v\otimes f \mapsto f(v)\\
\widehat{\iota}_V:&\ \CC\to V^{\ast}\otimes V: 1\mapsto \sum_{v\in\mathcal{B}_V} v^\ast\otimes v,
\end{split}
\end{align}
where \(\mathcal{B}_V\) is a basis for \(V\) and $\{v^*\,\, | \,\, v\in\mathcal{B}_V\}$ is the corresponding dual basis of $V^*$.

The strict tensor functor $\cF^{\textup{frgt}}: \cM^\str\rightarrow\cN^\str$ satisfies
\begin{equation}\label{relcR}
\cF^{\textup{frgt}}(c_{S,T})=P_{\underline{S},\underline{T}}\mathcal{R}_{\underline{S},\underline{T}}\in\textup{Hom}_{\cN^\str}(\underline{S}\tens\underline{T},\underline{T}\tens\underline{S}),
\end{equation}
where $\mathcal{R}_{\underline{S},\underline{T}}\in\textup{End}_{\cN^\str}(\underline{S}\tens\underline{T})$ is the morphism representing
\[
\cF^{\textup{frgt}}\bigl(\mathcal{R}_{\cF^\str(S),\cF^\str(T)}\bigr)\in\textup{End}_{\cN}(\cF^\str(\underline{S})\otimes\cF^\str(\underline{T})).
\]
For the twist $(\vartheta_S)_{S\in\cM^\str}$ of $\cM^\str$ (see Corollary \ref{cortheta}), we have
\[
\cF^{\textup{frgt}}(\vartheta_S)=\vartheta_{\underline{S}},
\]
where $\vartheta_{\underline{S}}\in\textup{End}_{\cN^\str}(\underline{S})$ is the endomorphism representing
$\cF^{\textup{frgt}}(\vartheta_{\cF^\str(S)})\in\textup{End}_{\cN}(\cF^\str(\underline{S}))$. The restriction of $\cF^{\textup{frgt}}$ to $\cM_{\textup{fd}}^\str$
preserves left duality, while for $S\in\Rep^\str$ we have by \eqref{rhoshift},
\[
\cF^{\textup{frgt}}(\widetilde{e}_S)=\widehat{e}_{\underline{S}}\bigl((q^{2\rho})_{\underline{S}}\tens\textup{id}_{\underline{S}^*}\bigr),\qquad
\cF^{\textup{frgt}}(\widetilde{\iota}_S)=\bigl(\textup{id}_{\underline{S}^*}\tens (q^{-2\rho})_{\underline{S}}\bigr)\widehat{\iota}_{\underline{S}},
\]
where $\widehat{e}_{\underline{S}}$ and $\widehat{\iota}_{\underline{S}}$ are the opposite evaluation and injection morphisms of $\cN^\str_{\textup{fd}}$ and
the endomorphism $(q^{\pm 2\rho})_{\underline{S}}\in\textup{End}_{\cN_{\textup{fd}}^\str}(\underline{S})$ is representing the action of $q^{\pm 2\rho}$ on $\cF^\str(\underline{S})$ (which clearly preserves the $\mathfrak{h}^*$-grading).

It follows that in the graphical calculus for $\cN^\str$, the morphisms
$\cF^{\textup{frgt}}(c_{S,T})$ ($S,T\in\cM^\str$) in $\cN^\str$ are dot-equal to Figure \ref{braiding in Vect}, where we have omitted the sublabels of $\mathcal{R}_{\underline{S},\underline{T}}$ in the coloring of the coupon, as well as the underlining in the colors of the strands.

\begin{figure}[H]
		\centering
		\includegraphics[scale = 0.55]{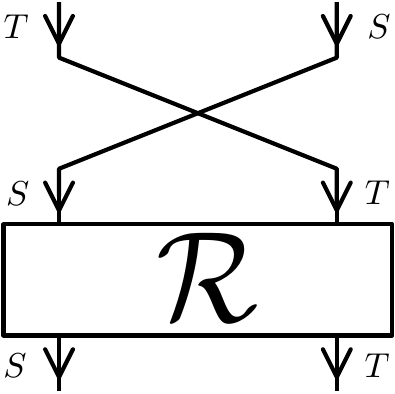}
		\captionof{figure}{}
		\label{braiding in Vect}
\end{figure}
\noindent
Similarly, for $S\in\Rep^\str$ the morphisms $\cF^{\textup{frgt}}(\widetilde{e}_S)$ and $\cF^{\textup{frgt}}(\widetilde{\iota}_S)$ in $\cN_{\textup{fd}}^\str$ are dot-equal to Figures \ref{evaluation right in Vect} and \ref{injection right in Vect} respectively.
\begin{figure}[H]
	\begin{minipage}{0.48\textwidth}
		\centering
		\includegraphics[scale = 1]{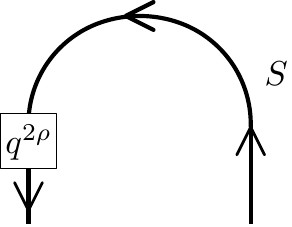}
		\captionof{figure}{}
		\label{evaluation right in Vect}
	\end{minipage}\quad
	\begin{minipage}{0.48\textwidth}
		\centering
		\includegraphics[scale = 1]{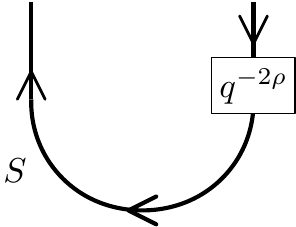}
		\captionof{figure}{}
		\label{injection right in Vect}
	\end{minipage}
\end{figure}
\noindent

To incorporate the evaluation map into the graphical calculus, we introduce the following notations. Let $M\in\cN$.
Recall the 1-dimensional $U_q(\mathfrak{b})$-module $\mathbb{C}_\nu=\mathbb{C}1_\nu$, which we now view as $\mathfrak{h}^*$-graded vector space with only one nontrivial graded component $\mathbb{C}_\nu[\nu]$, of degree $\nu\in\mathfrak{h}^*$. A homogeneous vector $m\in M[\nu]$ defines a morphism 
\[
\alpha_m\in\textup{Hom}_{\cN}(\CC_\nu,M),\qquad \alpha_m(1_\nu):=m.
\] 
For $M\in\cN$ and $f\in M[\nu]^*$ let
\[
\beta_f\in\textup{Hom}_{\cN}(M,\mathbb{C}_{\nu})
\]
be the unique morphism satisfying $\beta_f(m)=f(m)1_\nu$ for $m\in M[\nu]$. Note that $\beta_f\vert_{M[\mu]}=0$ for $\mu\not=\nu$, and we set
$M^*[-\nu]:=\{\beta_f\}_{f\in M[\nu]^*}$. 
The {\it restricted dual }of $M\in\cN$ is defined to be 
\[
M^\circ:=\bigoplus_{\mu\in\mathfrak{h}^\ast}M^*[\mu]\in\cN.
\]
This is to be compared with the construction of the restricted dual for $U_q$-modules discussed in Subsection \ref{dualsection}. 

In the $\cN$-graphical calculus, {\it boundary right \textup{(}resp. left\textup{)} spin coupons} in $\cN$-colored ribbon-braid graphs are coupons colored by morphisms $\alpha_m$ 
(resp. $\beta_f$) for $M\in\cN_{\textup{fd}}$, appearing respectively on the far left-- and far right-hand side of the diagram.  
We depict the boundary spin coupons by omitting the strand colored by the $1$-dimensional $\mathfrak{h}^*$-graded vector space $\mathbb{C}_\nu$, replacing the rectangle by a circle, and denoting the color of the coupon by the associated homogeneous vector $m$ (resp. $f$), see Figures \ref{endpoint A} and \ref{endpoint B}.
\begin{figure}[H]
	\begin{minipage}{0.48\textwidth}
		\centering
		\includegraphics[scale = 0.7]{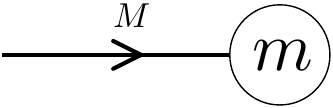}
		\captionof{figure}{$\alpha_{m}$}
		\label{endpoint A}
	\end{minipage}\quad
	\begin{minipage}{0.48\textwidth}
		\centering
		\includegraphics[scale = 0.7]{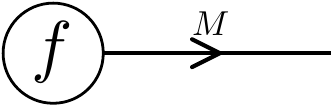}
		\captionof{figure}{$\beta_{f}$}
		\label{endpoint B}
	\end{minipage}
\end{figure}
\noindent
Recall the notation \(\mathbf{m}_\lambda\) for the highest weight vector of the Verma module \(M_\lambda\) defined in Subsection \ref{Section 1.2}, and \(\mathbf{m}_\lambda^\ast\) for its unique dual vector (see Subsection \ref{dualsection}). We denote the {\it boundary right and left coupons} colored by $\alpha_{\mathbf{m}_\lambda}$
and $\beta_{\mathbf{m}_\lambda^*}$, where $\mathbf{m}_\lambda$ and $\mathbf{m}_\lambda^*$ are now viewed as homogeneous vectors in $\underline{M_\lambda}\in\cN$ and $\underline{M_\lambda^\circ}\in\cN$ respectively,
 as in Figures \ref{highest weight vector Verma} and \ref{dual of highest weight vector Verma}.
\begin{figure}[H]
	\begin{minipage}{0.48\textwidth}
		\centering
		\includegraphics[scale = 0.75]{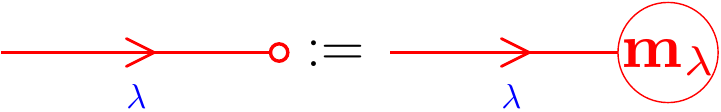}
		\captionof{figure}{$\alpha_{\mathbf{m}_\lambda}$}
		\label{highest weight vector Verma}
	\end{minipage}\quad
	\begin{minipage}{0.48\textwidth}
		\centering
		\includegraphics[scale = 0.75]{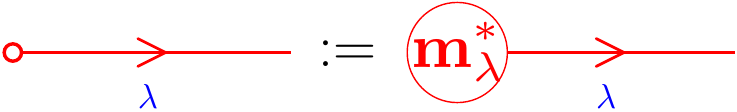}
		\captionof{figure}{$\beta_{\mathbf{m}_\lambda^\ast}$}
		\label{dual of highest weight vector Verma}
	\end{minipage}
\end{figure}

%%%%%%%%%%%%%%%%%%%%%%%%%%%%%
\begin{definition}
	\label{expectation value map def}
	For any \(\lambda,\mu\in\hh^\ast_{\textup{reg}}\) and \(S\in\Rep^\str\), we define the map
	\[
	\langle\cdot\rangle: \Hom_{\cM^{\str}_{\textup{adm}}}(M_\lambda,M_\mu\tens S)\to \textup{Hom}_{\cN^\str_{\textup{fd}}}(\mathbb{C}_{\lambda},\mathbb{C}_\mu\tens\underline{S})
	\]
	by
	\[
	\langle\Phi\rangle:=(\beta_{\mathbf{m}_\mu^*}\tens\textup{id}_{\underline{S}})\circ\Phi\circ\alpha_{\mathbf{m}_\lambda}.
	\]
\end{definition}
%%%%%%%%%%%%%%%%%%%%%%%%%%%%
This is consistent with the definition of the expectation value as given in Definition \ref{expmapdef}.
Indeed, 
using the natural identification 
\[
\textup{Hom}_{\cN}(\mathbb{C}_{\lambda},\mathbb{C}_\mu\otimes\mathcal{F}^\str(\underline{S}))\simeq
\textup{Hom}_{\cN}(\mathbb{C}_{\lambda-\mu},\cF^\str(\underline{S}))
\] 
the morphism underlying $\langle\Phi\rangle\in\textup{Hom}_{\cN^\str_{\textup{fd}}}(\mathbb{C}_{\lambda},\mathbb{C}_\mu\tens\underline{S})$ is 
$\alpha_m\in\textup{Hom}_{\cN}(\mathbb{C}_{\lambda-\mu},\cF^\str(\underline{S}))$ with 
the vector $m\in\mathcal{F}^\str(\underline{S})[\lambda-\mu]$ equal to the expectation value of the $U_q$-linear intertwiner 
representing $\Phi\in\Hom_{\cM^{\str}_{\textup{adm}}}(M_\lambda,M_\mu\tens S)$. Note that the coupon in $\mathbb{B}_{\cN^\str}$ colored by the morphism $\langle\Phi\rangle\in\textup{Hom}_{\cN^\str}(\mathbb{C}_\lambda,\mathbb{C}_\mu\tens\underline{S})$ is dot-equal to the $\cN^\str$-colored ribbon-braid graph diagram depicted in Figure \ref{ExpValpicture}, where we use the convention to denote the color $\underline{S}$ of the spin-strand by $S$, if no confusion can arise.
\begin{figure}[H]
\centering
\includegraphics[scale = 0.8]{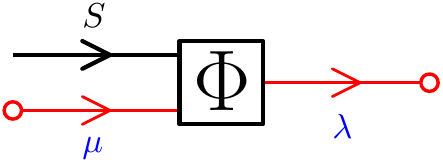}
\captionof{figure}{  }
\label{ExpValpicture}
\end{figure}

%%%%%%%%%%%%%%%%%%%%%%%%%%%%%%%%%%%%%%%%%%
\begin{lemma}
	\label{lemma braiding highest weight}
	For any \(\lambda\in\hh^\ast_{\textup{reg}}\) and \(S\in\Rep^\str\), one has the dot-equalities in $\mathbb{B}_{\cN^\str}$ represented by
	Figures \ref{braiding 1n} and \ref{braiding 3n},
	\begin{figure}[H]
		\begin{minipage}{0.48\textwidth}
			\centering
			\includegraphics[scale = 0.75]{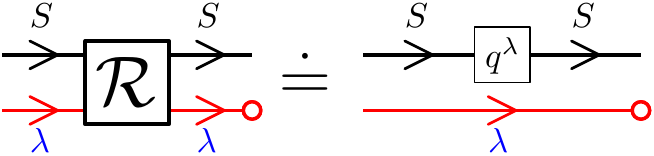}
			\captionof{figure}{  }
			\label{braiding 1n}
		\end{minipage}
		\begin{minipage}{0.48\textwidth}
			\centering
			\includegraphics[scale = 0.75]{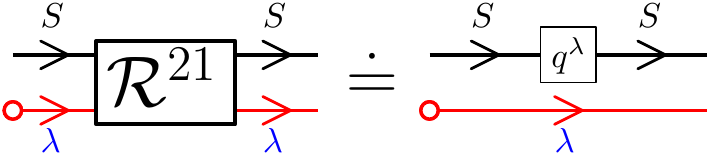}
			\captionof{figure}{  }
			\label{braiding 3n}
		\end{minipage}
	\end{figure}
\noindent
where $\mathcal{R}^{21}$ stands for the endomorphism
$(\mathcal{R}^{21})_{\underline{M_\lambda},\underline{S}}$ in $\textup{End}_{\cN^\str}(\underline{M_\lambda}\tens\underline{S})$  representing the $\mathfrak{h}$-linear endomorphism
$(\mathcal{R}^{21})_{M_\lambda,\cF^\str(S)}$ of $M_\lambda\otimes\cF^\str(S)$.
\end{lemma}
%%%%%%%%%%%%%%%%%%%%%%%%%%%%%%%%%%%%%%%%%%%%%
\begin{proof}
	The dot-equalities in Figures \ref{braiding 1n} and \ref{braiding 3n} are implied by the identities 
	\begin{equation*}
	\begin{split}
	\mathcal{R}_{M_\lambda,\cF^\str(S)}(\alpha_{\mathbf{m}_{\lambda}}\otimes \id_{\cF^\str(S)})&=\alpha_{\mathbf{m}_\lambda}\otimes\pi_{\cF^\str(S)}(q^{\lambda}),\\
	(\beta_{\mathbf{m}_{\lambda}^{\ast}}\otimes\id_{\cF^\str(S)})(\mathcal{R}^{21})_{M_\lambda,\cF^\str(S)}&=
	\beta_{\mathbf{m}_{\lambda}^{\ast}}\otimes\pi_{\cF^\str(S)}(q^{\lambda}),
	\end{split}
	\end{equation*}
	which follow from the expressions (\ref{quasi R def}) and \eqref{UniversalRmatrix} for $\mathcal{R}$, together with the fact that $\overline{\mathcal{R}}_\beta\in
	U^+[\beta]\otimes U^-[-\beta]$ ($\beta\in Q^+$) and $\overline{\mathcal{R}}_0=1\otimes 1$.
\end{proof}
%%%%%%%%%%%%%%%%%%%%%%%%%%

We are now in the position to derive the topological \(q\)-KZ equations for the dynamical fusion operator $\ol{J_S}(\lambda)$ from Subsection \ref{Section 2 C^+} using the $\cN^\str$-graphical calculus and the topological operator $q$-KZ equations for
$k$-point quantum vertex operators from Theorem \ref{operatorqKZthm}.

The starting point is the following identity, obtained by applying the expectation value map
$\langle\cdot\rangle: \textup{Hom}_{\cM^\str_{\textup{adm}}}(M_{\lambda_3},M_{\lambda_0}\tens S)\rightarrow\textup{Hom}_{\cN^\str_{\textup{fd}}}(\mathbb{C}_{\lambda_3},\mathbb{C}_{\lambda_0}\tens\underline{S})$
from Definition \ref{expectation value map def} to the identity in Proposition \ref{prop bulk ABRR}.

%%%%%%%%%%%%%%%%%%%%%%%%%%%%%%%
\begin{proposition}
	\label{theorem ABRR}
For \(\lambda_i\in\hh^\ast_{\textup{reg}}\), \(S_i\in\Rep^\str\) and \(\Phi_i\in\Hom_{\cM_{\mr{adm}}^\str}(M_{\lambda_i},M_{\lambda_{i-1}}\tens S_i)\), let us write \(S:= S_1\tens S_2\tens S_3\) and set
	\begin{equation*}
	\Phi:=(\Phi_1\tens\id_{S_2\tens S_3})(\Phi_2\tens\id_{S_3})\Phi_3\in\Hom_{\cM_{\mr{adm}}^\str}(M_{\lambda_3},M_{\lambda_0}\tens S).
	\end{equation*}
Then we have 
	\begin{equation}
	\label{ABRR 3-point}
	q^{\langle\lambda_1+\lambda_2+2\rho,\lambda_2-\lambda_1\rangle}\langle\Phi\rangle = (\cR^{21})_{\underline{S_1},\underline{S_2}}(q^{\lambda_0+\lambda_3+2\rho})_{\underline{S_2}}(\cR^{21})^{-1}_{\underline{S_2},\underline{S_3}}\,\langle\Phi\rangle,
	\end{equation}
viewed as identity in $\textup{Hom}_{\cN_{\textup{fd}}^\str}(\mathbb{C}_{\lambda_3},\mathbb{C}_{\lambda_0}\tens\underline{S})$.
\end{proposition}
%%%%%%%%%%%%%%%%%%%%%%%%%%
\begin{proof}
From Proposition \ref{prop bulk ABRR} we get
\begin{center}
	\includegraphics[scale = 0.65]{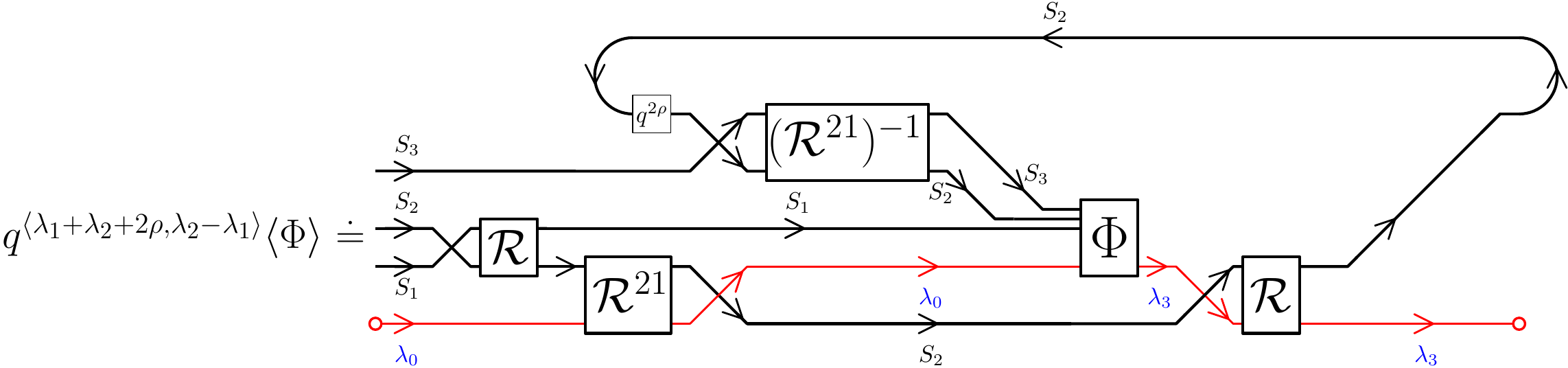}
\end{center}
Upon applying Lemma \ref{lemma braiding highest weight}, the right-hand side is dot-equal to the $\cN^\str$-colored ribbon-braid graph diagram
\begin{center}
	\includegraphics[scale = 0.7]{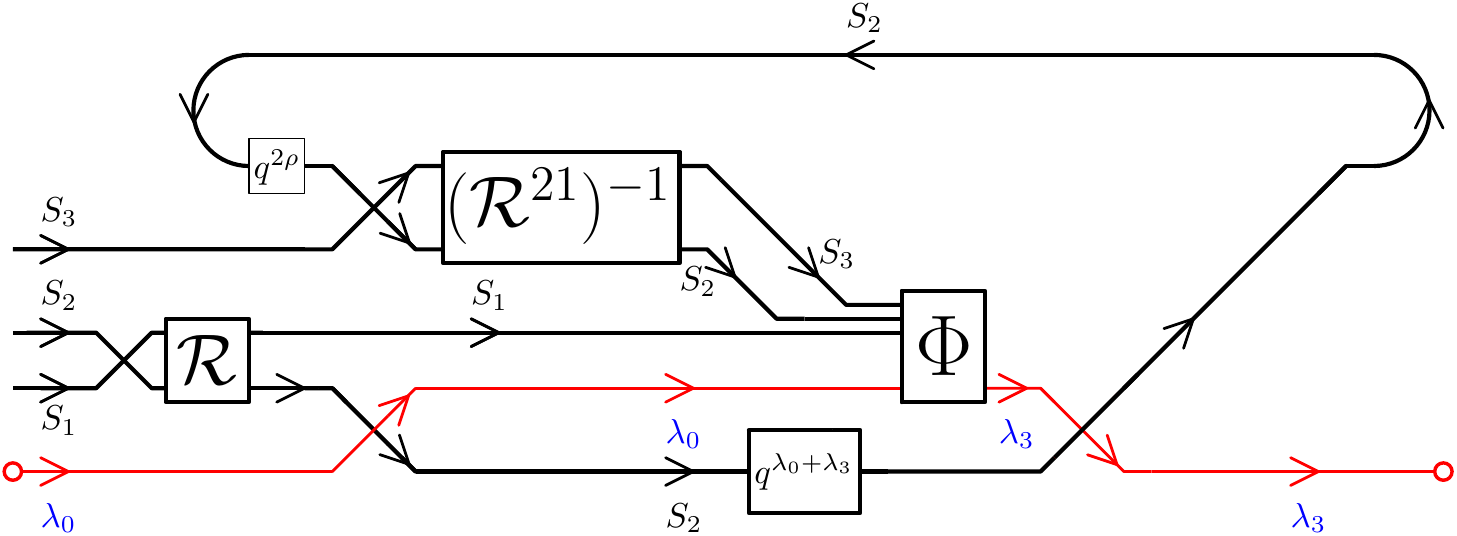}
\end{center}
In $\mathbb{B}_{\cN^\str}$ one can now pull the strand colored by \(S_2\) through the red strand and through several of the occurring coupons, leading to dot-equality with the $\cN^\str$-colored ribbon-braid graph diagram
\begin{center}
	\includegraphics[scale = 0.7]{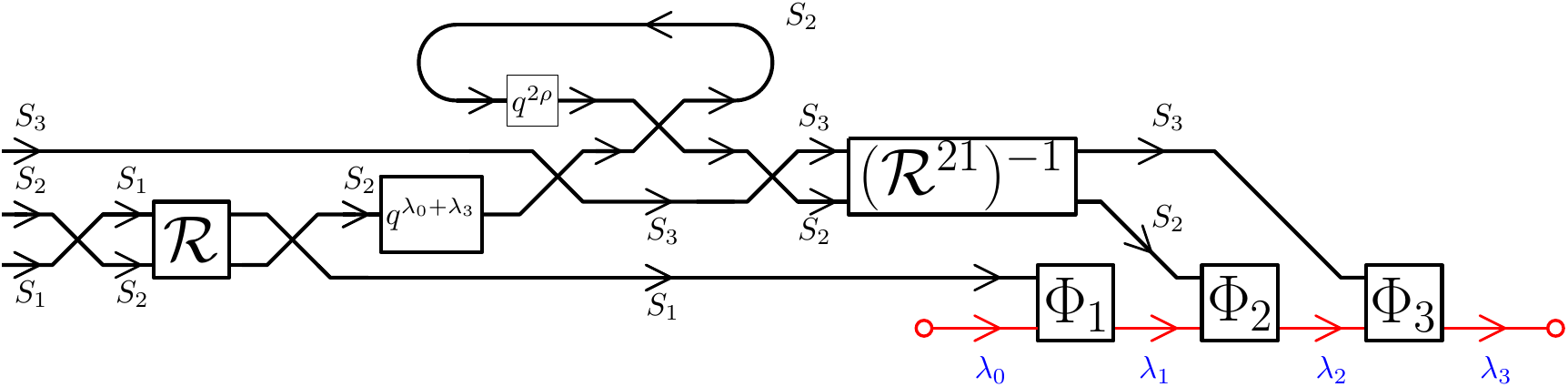}
\end{center}
Here we have used the fact that modulo dot-equality, the coupon depicted by a solid crossing can be replaced by either an over-crossing or an under-crossing in \(\cN\), such that the strand labeled by \(S_2\) can really be thought of as a contractable loop.
Recalling that the twist is trivial in the category $\cN_{\textup{fd}}^\str$, one can eliminate the closed loop colored by \(S_2\) up to dot-equality. After simplifying the graphics in \(\cN_{\textup{fd}}^\str\) by removing any double crossings, we obtain dot-equality with the diagram
\begin{center}
	\includegraphics[scale = 0.7]{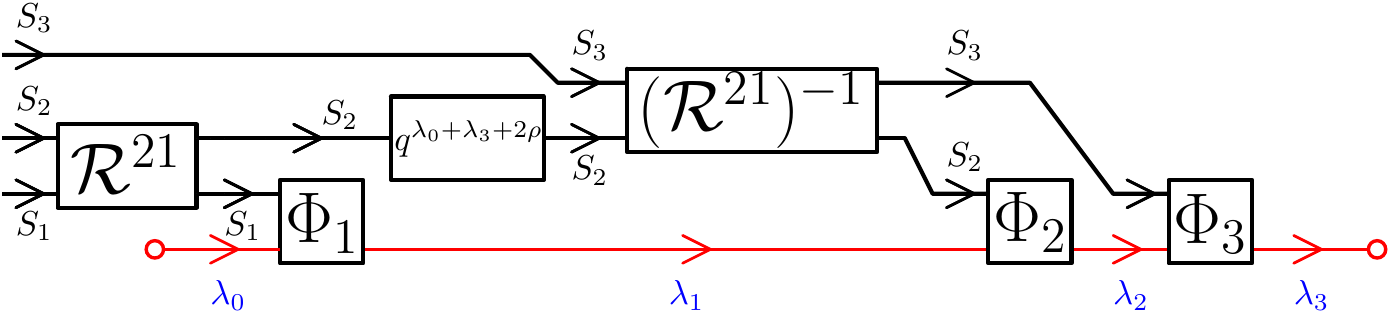}
\end{center}
This diagram is mapped to the right-hand side of (\ref{ABRR 3-point}) by $\widetilde{\cF}_{\cN}^{\textup{br}}$, which completes the proof.
\end{proof}
%%%%%%%%%%%%%%%%%%%%%%%%%%
We now use \eqref{ABRR 3-point} to derive asymptotic $q$-KZ equations for the highest weight to highest weight components of $k$-point quantum vertex operators.

For $\lambda\in\mathfrak{h}^*$ and $M\in\cM_{\textup{adm}}$, let $(q^{2\theta(\lambda)})_{\underline{M}}\in\textup{End}_{\cN}(\underline{M})$ be the map
acting as $q^{\langle 2(\lambda+\rho)-\mu,\mu\rangle}\textup{id}_{M[\mu]}$ on $M[\mu]$.
Informally, $\theta(\lambda)$ is the element $\lambda+\rho-\tfrac12 \sum\nolimits_{i=1}^r x_i^2\in U(\mathfrak{h})$. This is not be confused with the notation \(\vartheta\) for the ribbon element defined in Subsection \ref{Section ribbon element}. Recall also the notation \(\kappa\) for the element in \(\mathcal{U}^{(2)}\) that relates the R-matrix to the quasi R-matrix (see Definition \ref{univRdef}).
%%%%%%%%%%%%%%%%%%%%%%%%%%
\begin{proposition}
Let $\lambda\in\mathfrak{h}_{\textup{reg}}^*$ and $S=(V_1,\ldots,V_k)\in\cM_{\textup{fd}}^\str$. Fix weight vectors $v_i\in V_i[\nu_i]$ and set $\lambda_i:=\lambda-\nu_{i+1}-\cdots-\nu_k$ for $0\leq i< k$, with the convention that $\lambda_k=\lambda$. 

The highest weight to highest weight component
\[
\langle\Phi_\lambda^{v_1,\ldots,v_k}\rangle\in\textup{Hom}_{\cN_{\textup{fd}}^\str}(\mathbb{C}_\lambda,\mathbb{C}_{\lambda-\sum_\ell\nu_\ell}\tens \underline{S})
\]
of the $k$-point quantum vertex operator $\Phi_\lambda^{v_1,\ldots,v_k}\in\textup{Hom}_{\cM_{\textup{adm}}^\str}(M_\lambda,M_{\lambda-\sum_\ell\nu_\ell}\tens S)$ satisfies the operator $q$-KZ equations
\begin{equation}\label{qKZeqi}
\begin{split}
&q^{\langle \lambda_i+\lambda_{i-1}+2\rho,\lambda_i-\lambda_{i-1}\rangle}\langle\Phi_\lambda^{v_1,\ldots,v_k}\rangle\\
=\ &(\mathcal{R}^{21})_{V_{i-1},V_i}\cdots (\mathcal{R}^{21})_{V_1,V_i}A_i^S(\lambda)
(\mathcal{R}^{21})_{V_i,V_k}^{-1}
\cdots (\mathcal{R}^{21})_{V_i,V_{i+1}}^{-1}\langle\Phi_\lambda^{v_1,\ldots,v_k}\rangle
\end{split}
\end{equation}
for $i=1,\ldots,k$, where 
$A_i^S(\lambda)\in\textup{End}_{\cN_{\textup{fd}}^\str}(\underline{S})$ is defined by
\begin{equation}\label{AiS}
A_i^S(\lambda):=\kappa_{V_1,V_i}^{-1}\cdots \kappa_{V_{i-1},V_i}^{-1}(q^{2\theta(\lambda)})_{V_i}\kappa_{V_i,V_{i+1}}^{-1}\cdots \kappa_{V_i,V_k}^{-1}.
\end{equation}
Here we have simplified the notation by omitting the underlining of sublabels in the formulas.
\end{proposition}
%%%%%%%%%%%%%%%%%%%%%%%%
\begin{proof}
We apply \eqref{ABRR 3-point} to the three quantum vertex operators $\Phi_1:=\Phi_{\lambda_{i-1}}^{v_1,\ldots,v_{i-1}}$, 
$\Phi_2=\Phi_{\lambda_i}^{v_i}$ and $\Phi_3=\Phi_\lambda^{v_{i+1},\ldots,v_k}$ (compare with the proof of Theorem \ref{operatorqKZthm}).
Then $S_1=(V_1,\ldots,V_{i-1})$, $S_2=(V_i)$, $S_3=(V_{i+1},\ldots,V_k)$ and $\Phi=\Phi_{\lambda}^{v_1,\ldots,v_k}$.
The weights $(\lambda_0,\lambda_1,\lambda_2,\lambda_3)$ in \eqref{ABRR 3-point} become
$(\lambda_0,\lambda_{i-1},\lambda_i,\lambda)$.

With these choices, equation \eqref{ABRR 3-point} reduces to
\begin{equation}\label{highlevel}
q^{\langle \lambda_i+\lambda_{i-1}+2\rho,\lambda_i-\lambda_{i-1}\rangle}\langle\Phi_\lambda^{v_1,\ldots,v_k}\rangle=
(\mathcal{R}^{21})_{S_1,V_i}(q^{2(\lambda+\rho)-\sum_\ell\nu_\ell})_{V_i}(\mathcal{R}^{21})^{-1}_{V_i,S_3}
\langle\Phi_\lambda^{v_1,\ldots,v_k}\rangle
\end{equation}
in $\textup{Hom}_{\cN^\str_{\textup{fd}}}(\mathbb{C}_\lambda,\mathbb{C}_{\lambda_0}\tens\underline{S})$. 

Note that the endomorphisms $(q^{2(\lambda+\rho)-\sum_\ell\nu_\ell})_{V_i}$ and $A_i^S(\lambda)$ of $\cF^\str(\underline{S})$ act in the same way on $\cF^\str(\underline{S})[\sum_\ell\nu_\ell]$. Furthermore, the image of $\langle\Phi_\lambda^{v_1,\ldots,v_k}\rangle$ lies in 
$\mathbb{C}_{\lambda_0}\otimes\cF^\str(\underline{S})[\sum_\ell\nu_\ell]$ and $(\mathcal{R}^{21})_{V_i,S_3}^{-1}$ preserves
$\mathbb{C}_{\lambda_0}\otimes\cF^\str(\underline{S})[\sum_\ell\nu_\ell]$, so we conclude that the operator
$(q^{2(\lambda+\rho)-\sum_\ell\nu_\ell})_{V_i}$ in the right-hand side of \eqref{highlevel} may be replaced by $A_i^S(\lambda)$. 

We thus have
\[
q^{\langle \lambda_i+\lambda_{i-1}+2\rho,\lambda_i-\lambda_{i-1}\rangle}\langle\Phi_\lambda^{v_1,\ldots,v_k}\rangle=
(\mathcal{R}^{21})_{S_1,V_i}A_i^S(\lambda)(\mathcal{R}^{21})^{-1}_{V_i,S_3}
\langle\Phi_\lambda^{v_1,\ldots,v_k}\rangle
\]
in $\textup{Hom}_{\cN^\str_{\textup{fd}}}(\mathbb{C}_\lambda,\mathbb{C}_{\lambda_0}\tens\underline{S})$. The result now follows, since we have
\begin{equation}\label{Rexpand}
\begin{split}
(\mathcal{R}^{21})_{S_1,V_i}&=(\mathcal{R}^{21})_{V_{i-1},V_i}\cdots (\mathcal{R}^{21})_{V_1,V_i},\\
(\mathcal{R}^{21})_{V_i,S_3}&=(\mathcal{R}^{21})_{V_i,V_{i+1}}\cdots (\mathcal{R}^{21})_{V_i,V_k}
\end{split}
\end{equation}
in $\textup{End}_{\cN_{\textup{fd}}^\str}(\underline{S})$ by \eqref{relcR} and the hexagon identities.
\end{proof}
%%%%%%%%%%%%%%%%%%%%%%%%

We now translate the equations \eqref{qKZeqi} to topological $q$-KZ type equations for the dynamical $k$-point fusion operators.

%%%%%%%%%%%%%%%%%%%%%
\begin{theorem}[Topological $q$-KZ equations for dynamical $k$-point fusion operators]\hfill

\noindent
Let $\lambda\in\mathfrak{h}_{\textup{reg}}^*$ and $S=(V_1,\ldots,V_k)\in\cM_{\textup{fd}}^\str$. 
For $i=1,\ldots,k$ we have
\begin{equation}\label{qKZi}
\begin{split}
&\ol{J_S}(\lambda)(q^{2\theta(\lambda)})_{V_i}\kappa_{V_i,V_{i+1}}^{-2}\cdots \kappa_{V_i,V_k}^{-2}\\
=\ &(\mathcal{R}^{21})_{V_{i-1},V_i}\cdots (\mathcal{R}^{21})_{V_1,V_i}A_i^S(\lambda)
(\mathcal{R}^{21})_{V_i,V_k}^{-1}
\cdots (\mathcal{R}^{21})_{V_i,V_{i+1}}^{-1}\ol{J_S}(\lambda)
\end{split}
\end{equation}
as identity in $\textup{Hom}_{\cN^\str}(\underline{S},\cF^\str(\underline{S}))$,
with $A_i^S(\lambda)$ and the $R$-matrices in the right-hand side of \eqref{qKZi} viewed as endomorphisms of $\cF^\str(\underline{S})$.
\end{theorem}
%%%%%%%%%%%%%%%%%%%%%
\begin{proof}
Fix weight vectors $v_i\in V_i[\nu_i]$ for $i=1,\ldots,k$. 
By definition of the dynamical fusion operator $j_S(\lambda)$ (see \eqref{defjSlambda}),
 we conclude from \eqref{qKZeqi} that 
\begin{align*}
&q^{\langle 2(\lambda+\rho-\nu_{i+1}-\cdots-\nu_k)-\nu_i,\nu_i\rangle}
j_S(\lambda)(v_1\otimes\cdots\otimes v_k)\\
=\ &(\mathcal{R}^{21})_{V_{i-1},V_i}\cdots (\mathcal{R}^{21})_{V_1,V_i}A_i^S(\lambda)(\mathcal{R}^{21})_{V_i,V_k}^{-1}
\cdots (\mathcal{R}^{21})_{V_i,V_{i+1}}^{-1}j_S(\lambda)(v_1\otimes\cdots\otimes v_k)
\end{align*}
as identity in $\cF^\str(\underline{S})$.
The left-hand side can be rewritten in terms of operators that do not depend on the weights $\nu_j$:
\[
q^{\langle 2(\lambda+\rho-\nu_{i+1}-\cdots-\nu_k)-\nu_i,\nu_i\rangle}v_1\otimes\cdots\otimes v_k=
(q^{2\theta(\lambda)})_{V_i}\kappa_{V_i,V_{i+1}}^{-2}\cdots \kappa_{V_i,V_k}^{-2}v_1\otimes\cdots\otimes v_k.
\]
As a result we obtain 
\begin{equation}\label{qKZidown}
\begin{split}
&j_S(\lambda)(q^{2\theta(\lambda)})_{V_i}\kappa_{V_i,V_{i+1}}^{-2}\cdots \kappa_{V_i,V_k}^{-2}\\
=\ &(\mathcal{R}^{21})_{V_{i-1},V_i}\cdots (\mathcal{R}^{21})_{V_1,V_i}A_i^S(\lambda)
(\mathcal{R}^{21})_{V_i,V_k}^{-1}
\cdots (\mathcal{R}^{21})_{V_i,V_{i+1}}^{-1}j_S(\lambda)
\end{split}
\end{equation}
as identity in $\textup{End}_{\cN}(\cF^\str(\underline{S}))$. This immediately implies the desired result.
\end{proof}
%%%%%%%%%%%%%%%%%%%%%

As a special case we obtain the following equations for the dynamical $2$-point fusion operator, the first of which (equation \eqref{ABRR ordinary}) is originally due to Arnaudon, Buffenoir, Ragoucy and Roche \cite{Arnaudon&Buffenoir&Ragoucy&Roche-1998} (see also \cite{Etingof&Varchenko-1999}).
%%%%%%%%%%%%%%%%%%%%%%%%
\begin{corollary}
	\label{cor ABRR D}
	For any \(V_1,V_2\in\Rep\) and \(\lambda\in\hh_{\mathrm{reg}}^\ast\), the dynamical $2$-point fusion operator \(j_{(V_1,V_2)}(\lambda)\in\textup{End}_{\cN}(\underline{V_1}\otimes\underline{V_2})\) satisfies 
	\begin{equation}
	\label{ABRR ordinary}
	j_{(V_1,V_2)}(\lambda)(q^{2\theta(\lambda)})_{V_2}
	= (\mathcal{R}^{21})_{V_1,V_2}\kappa_{V_1,V_2}^{-1}(q^{2\theta(\lambda)})_{V_2}\,j_{(V_1,V_2)}(\lambda),
	\end{equation}
as well as 
\begin{equation}\label{ABRR dual}
j_{(V_1,V_2)}(\lambda)(q^{2\theta(\lambda)})_{V_1}\kappa^{-2}_{V_1,V_2}=
(q^{2\theta(\lambda)})_{V_1}\kappa_{V_1,V_2}^{-1}(\mathcal{R}^{21})_{V_1,V_2}^{-1}\,
j_{(V_1,V_2)}(\lambda).
\end{equation}
\end{corollary}
%%%%%%%%%%%%%%%%%%%%%%%%
\begin{proof}
Equation \eqref{ABRR ordinary} is the case $i=k=2$ of \eqref{qKZi}, while \eqref{ABRR dual}
is the case $i=1$ and $k=2$ of \eqref{qKZi}.
\end{proof}
%%%%%%%%%%%%%%%%%%%%%%%%

\begin{remark}
By \eqref{dynamical cocycle} and \eqref{Rexpand}, equation
\eqref{ABRR ordinary} \textup{(}resp. \eqref{ABRR dual}\textup{)} for the dynamical $2$-point fusion operator
implies the topological $q$-KZ equation \eqref{qKZi} for $i=k$ \textup{(}resp. $i=1$\textup{)} and arbitrary $k\geq 2$. 

The topological $q$-KZ equations \eqref{qKZi} for $k>2$ and $1<i<k$ are a direct consequence of the topological $q$-KZ equation 
\[
j_{(V_1,V_2,V_3)}(\lambda)(q^{2\theta(\lambda)})_{V_2}\kappa_{V_2,V_3}^{-2}=
(\mathcal{R}^{21})_{V_1,V_2}\kappa_{V_1,V_2}^{-1}(q^{2\theta(\lambda)})_{V_2}\kappa^{-1}_{V_2,V_3}
(\mathcal{R}^{21})^{-1}_{V_2,V_3}j_{(V_1,V_2,V_3)}(\lambda)
\]
for dynamical $3$-point fusion operators \textup{(}which corresponds to the case $k=3$ and $i=2$ of equation \eqref{qKZi}\textup{)}.
\end{remark}

\section*{Conclusion}

In this paper we have laid the foundations of graphical calculus for ribbon categories and braided monoidal categories with twist, by refining and generalizing the Reshetikhin-Turaev functor. We have applied this graphical framework to categories of \(U_q\)-modules, in order to obtain topological \(q\)-KZ equations satisfied by dynamical fusion operators. In our upcoming paper \cite{DeClercq&Reshetikhin&Stokman-2021}, which will serve as a sequel to the present paper, we will exploit the full potential of the graphical calculus, by introducing a rigorous process for dynamicalization of morphisms in \(\Rep^\str\), and by deriving purely graphical proofs of the dual \(q\)-KZB and dual Macdonald-Ruijsenaars equations satisfied by normalized traces of \(k\)-point quantum vertex operators.

%%%%%%%%%%%%%%%%%%%%%%

\end{document}